\documentclass[oneside,11pt,reqno,a4paper,final]{amsart}
\usepackage[T1]{fontenc}
\usepackage[utf8]{inputenc}
\usepackage[top=1cm,bottom=2cm,right=2cm,left=2cm,includehead]{geometry}
\usepackage{amsmath,amsfonts,amssymb,amscd,amsthm}
\usepackage[usenames,dvipsnames,svgnames]{xcolor}
\usepackage[numbers,sort&compress]{natbib}
\usepackage{doi}
\usepackage{graphicx}
\usepackage[font={small}]{caption} 
\usepackage[font={small}]{subcaption}
\usepackage{autonum} 
\usepackage{hyperref} 
\usepackage{enumitem}
\usepackage{changes}
\newcommand{\myadded}[1]{{\leavevmode #1}}

\newcommand{\thetitle}{\uppercase{Robust and scalable h-adaptive aggregated 
unfitted \\ finite elements for interface elliptic problems}}

\newcommand{\theauthors}{Eric Neiva\textsuperscript{a,b,}\footnote{Corresponding
author. \\ Emails: \theemailsanddate} and Santiago Badia\textsuperscript{c,a}} 
\newcommand{\theauths}{E. Neiva and S. Badia}

\newcommand{\theaffiliations}{
	\textsuperscript{a} CIMNE – Centre Internacional de M\`etodes Num\`erics en
	Enginyeria, \\ Edifici C1, Campus Nord UPC, C. Gran Capit\`a S/N, 08034
	Barcelona, Spain. \\ [0.5em]
	\textsuperscript{b} Department of Civil and	Environmental Engineering, 
	Universitat Polit\`ecnica de Catalunya, \\ Edifici C2, Campus Nord UPC, 
	C. Jordi Girona 1-3, 08034 Barcelona, Spain. \\ [0.5em]
	\textsuperscript{c} School of Mathematics, Monash University, Clayton, 
	Victoria, 3800, Australia.
}

\newcommand{\theemailsanddate}{
	\texttt{eneiva@cimne.upc.edu},
	\texttt{santiago.badia@monash.edu} \\ 
	\today
}

\newcommand{\thethanks}{
\begin{small}
Financial support from the European Commission under the FET-HPC ExaQUte project
(Grant agreement ID: 800898) within the Horizon 2020 Framework Programme is
gratefully acknowledged. This work has been partially funded by the project
MTM2014-60713-P from the ``Ministerio de Econom\'ia, industria y Competitividad''
of Spain. E. Neiva gratefully acknowledges the support received from the Catalan
Government through a FI fellowship (2019 FI-B2-00090; 2018 FI-B1-00095; 2017
FI-B-00219). S. Badia gratefully acknowledges the support received from the
Catalan Government through the ICREA Acad\`emia Research Program. The authors
thankfully acknowledge the computer resources at Marenostrum-IV and the technical
support provided by the Barcelona Supercomputing Center (RES-ActivityIDs:  
IM-2019-3-0008, IM-2020-1-0002). Financial support to CIMNE via the CERCA
Programme / Generalitat de Catalunya is also acknowledged.
\end{small}
}

\hypersetup{
    bookmarks=true,
    breaklinks=true,
    bookmarksopen=true,
    pdftitle={\thetitle},  
    pdfauthor={\theauths}, 
    colorlinks=true,       
    linkcolor=black,       
    citecolor=blue,        
    filecolor=black,       
    urlcolor=blue          
}

\usepackage{fancyhdr}

\usepackage{newtxtext}
\usepackage{newtxmath}

\usepackage{acronym}
\usepackage[most]{tcolorbox}

\usepackage{booktabs}

\usepackage{tikz}
\definecolor{myellow}{RGB}{255,230,128}
\definecolor{gray20}{RGB}{204,204,204}
\definecolor{mygray}{RGB}{204,204,204}
\definecolor{mygreen}{RGB}{138,203,95}
\definecolor{myblue}{RGB}{77,151,214}

\definecolor{lstgrey}{rgb}{0.95,0.95,0.95}
\usepackage{listings}
\usepackage[vlined]{algorithm2e}
\usepackage{xcolor}

\SetAlFnt{\small}
\SetArgSty{textnormal}

\SetAlCapSty{xAlCapSty}

\SetCommentSty{xCommentSty}

\SetNlSty{mynlfont}{}{} 

\LinesNumbered

\SetSideCommentRight

\DontPrintSemicolon

\RestyleAlgo{algoruled}

\SetInd{0.5em}{0.5em}

\acrodef{pde}[PDE]{partial differential equation}
\acrodef{bvp}[BVP]{boundary value problem}
\acrodef{amr}[AMR]{Adaptive mesh refinement and coarsening}
\acrodef{ls}[LS]{Level-Set}
\acrodef{dof}[DOF]{degree of freedom}
\acrodefplural{dof}[DOFs]{degrees of freedom}
\acrodef{vef}[VEF]{vertex, edge, and face}
\acrodefplural{vef}[VEFs]{vertices, edges, and faces}
\acrodef{cg}[CG]{continuous Galerkin}
\acrodef{dg}[DG]{discontinuous Galerkin}
\acrodef{fe}[FE]{finite element}
\acrodef{fem}[FEM]{finite element method}
\acrodef{xfem}[XFEM]{extended finite element method}
\acrodef{agfe}[agFE]{aggregated finite element}
\acrodef{agfem}[AgFEM]{aggregated finite element method}
\acrodef{cgm}[CG]{conjugate gradient}
\acrodef{amg}[AMG]{algebraic multigrid}
\acrodef{lb}[LB]{Li and Bettess}
\acrodef{ob}[OB]{O\~{n}ate and Bugeda}
\acrodef{dd}[DD]{domain decomposition}
\acrodef{mpi}[MPI]{message passing interface}
\acrodef{hpc}[HPC]{high-performance computing}
\acrodef{sipm}[SIPM]{symmetric interior penalty method}

\newtheorem{theorem}{Theorem}[section]
\newtheorem{lemma}[theorem]{Lemma}
\newtheorem{proposition}[theorem]{Proposition}

\graphicspath{{figures/}}

\def\fempar{{\texttt{FEMPAR}}}

\def\p4est{{\texttt{p4est}}}
\def\t8code{{\texttt{t8code}}}

\def\petsc{{\texttt{PETSc}}}

\def\matlab{{\texttt{MATLAB}}}
\def\gamg{{\texttt{GAMG}}}

\def\d{{\mathrm{d}}}
\def\cell{T}
\def\T{{\mathcal{T}}}
\def\V{{\mathcal{V}}}

\def\x{{\boldsymbol{x}}}
\def\u{{\boldsymbol{u}}}
\def\v{{\boldsymbol{v}}}
\def\f{{\boldsymbol{f}}}
\def\g{{\boldsymbol{g}}}
\def\j{{\boldsymbol{j}}}
\def\w{{\boldsymbol{w}}}
\def\h{{\boldsymbol{h}}}
\def\n{{\boldsymbol{n}}}
\def\H{{\boldsymbol{H}}}
\def\L{{\boldsymbol{L}}}
\def\s{{\boldsymbol{\sigma}}}
\def\eps{{\boldsymbol{\varepsilon}}}
\def\R{{\mathcal{R}}}

\def\C{{\mathrm{C}}}

\def\ag{{\mathrm{ag}}}
\def\act{{\mathrm{A}}}

\def\std{{\mathrm{std}}}
\def\sz{{\mathrm{SZ}}}
\def\meas{{\mathrm{meas}}}
\def\W{{\mathrm{W}}}
\def\I{{\mathrm{I}}}
\def\E{{\mathrm{E}}}
\def\bsnabla{\boldsymbol{\nabla}}

\newcommand{\normstab}[2]{{\| #1 \|_{\V(h)}^{#2}}}

\newcommand{\normcont}[2]{{\left\vert\kern-0.25ex\left\vert\kern-0.25ex\left\vert
 #1 
 \right\vert\kern-0.25ex\right\vert\kern-0.25ex\right\vert_{\V}^{#2}}}
\newcommand{\brokendomain}{{\Omega^+ \cup \Omega^-}}

\newcommand*{\ldblbrace}{\{\mskip-3.5mu\{}
\newcommand*{\rdblbrace}{\}\mskip-3.5mu\}}

\newcommand{\restrict}[2]{{\left. #1 \right|_{#2}}}

\newcommand{\jump}[1]{{\left\llbracket #1 \right\rrbracket}}
\newcommand{\average}[1]{{\left\ldblbrace #1 \right\rdblbrace}}
\newcommand{\norm}[3]{{\left\| #1 \right\|_{#2}^{#3}}}
\newcommand{\seminorm}[3]{{\left| #1 \right|_{#2}^{#3}}}

\begin{document}

\thispagestyle{empty}

\renewcommand*{\thefootnote}{\fnsymbol{footnote}}

\begin{center}
{ \bf {\thetitle}}

\vspace*{1em}

\theauthors

\vspace*{1em}

\theaffiliations

\end{center}

\setcounter{footnote}{0}
\renewcommand*{\thefootnote}{\arabic{footnote}}

\begin{center}

{\bf Abstract}

\vspace*{1em}

\begin{minipage}{0.9\textwidth}
\begin{small}

This work introduces a novel, fully robust and highly-scalable, $h$-adaptive
aggregated unfitted finite element method for large-scale interface elliptic
problems. The new method is based on a recent distributed-memory implementation of
the aggregated finite element method atop a highly-scalable Cartesian
forest-of-trees mesh engine. It follows the classical approach of weakly coupling
nonmatching discretisations at the interface to model internal discontinuities at
the interface. We propose a natural extension of a single-domain parallel cell
aggregation scheme to problems with a finite number of interfaces; it
straightforwardly leads to aggregated finite element spaces that have the
structure of a Cartesian product. We demonstrate, through standard numerical
analysis and exhaustive numerical experimentation on several complex Poisson and
linear elasticity benchmarks, that the new technique enjoys the following
properties: well-posedness, robustness with respect to cut location and material
contrast, optimal (h-adaptive) approximation properties, high scalability and easy
implementation in large-scale finite element codes. As a result, the method offers
great potential as a useful finite element solver for large-scale
\replaced{interface}{multiphase and multiphysics} problems modelled by partial
differential equations.

\end{small}

\end{minipage}
\end{center}

\vspace*{1em}






\section{Introduction}
\label{sec:intro}

Unfitted \ac{fe} methods are generating considerable interest in many practical
situations. Their ability to handle complex geometries, avoiding cumbersome and
time-consuming body-fitted mesh generation, makes them especially appealing for
large-scale simulations. They have been successfully exploited in
\replaced{many}{multiphase and multiphysics} applications with moving interfaces,
\myadded{such as fracture
mechanics~\cite{Sukumar2000,Waisman2013,berger-vergiat_inexact_2012},
fluid-structure
interaction~\cite{schott2019monolithic,alauzet2016nitsche,zonca2018unfitted,Massing2015},
two-phase and free surface
flows~\cite{Sauerland2011,SAYE2017683,kirchhart2016analysis}, and in applications
with varying domains, such as shape or topology
optimisation~\cite{Burman2018,feppon2019shape}, additive
manufacturing~\cite{neiva2020numerical,carraturo2020modeling}, and stochastic
geometry problems~\cite{badia2019embedded}}. In the numerical community, unfitted
\ac{fe} methods receive different denominations. When the motivation is to capture
(moving) interfaces, they are usually referred to as eXtended \ac{fe} methods
(XFEM)~\cite{belytschko_arbitrary_2001}. On the other hand, when the goal is to
simulate a problem using a (usually simple) background mesh, they are denoted as
\emph{unfitted} or \emph{embedded} or \emph{immersed} techniques; see, e.g.~the
cutFEM method~\cite{burman_cutfem_2015}, the cutIGA method~\cite{Elfverson2018},
the immersed boundary method~\cite{Mittal2005}, the finite cell
method~\cite{Schillinger2015}, \myadded{the shifted boundary
method~\cite{main2018shifted}, the immersogeometric
method~\cite{kamensky2015immersogeometric} and \ac{dg} methods
with cell
aggregation~\cite{saye2017implicit,engwer2012dune,johansson2013high,muller2017high,sollie2011space}.}

This work investigates unfitted \ac{fe} methods in \emph{large scale} simulations
of \replaced{interface}{multiphysics} problems modelled with \acp{pde}.
\deleted{In particular, it centres upon interface problems.} \myadded{Typical
approaches pursued to model internal discontinuities across the unfitted interface
are (1) weak coupling of nonmatching discretisations~\cite{hansbo2002unfitted},
(2) local partition-of-unity enrichments~\cite{melenk1996partition} and (3)
Lagrange multiplier or mortar
methods~\cite{bechet2009stable,burman2010fictitious}, although all three are
closely connected~\cite{areias2006comment,stenberg1995some}.} This work focuses on
the first approach. It broadly consists in dividing the mesh into two (sub)meshes
that overlap in cut cells. It leads to \ac{fe} approximations that have the
structure of a Cartesian product. Transmission conditions on the unfitted
interface are then weakly enforced by means of penalty~\cite{babuvska1973finite}
or Nitsche~\cite{nitsche_uber_2013} formulations, among others.

In the context of unfitted interface methods, the main challenge is to derive
robust methods for large material contrast across the interface. Indeed, naive
variational formulations may exhibit poor stability in this regime, e.g.~average
numerical flux weighting in Nitsche methods produces inaccurate and oscillating
approximation of interface quantities~\cite{annavarapu2012robust}. On the other
hand, large material contrast problems are prone to the so-called \emph{small cut
cell problem}. This issue is formally circumscribed to the unfitted boundary case
and it is associated with cut cells with arbitrarily small intersection with the
physical domain. Unless a specific technique mitigates the problem, numerical
integration on these badly-cut cells leads to severe ill-conditioning
problems~\cite{DePrenter2017,Badia2018}. Since unfitted boundary problems can be
interpreted as a limiting case of large contrast interface ones, the latter are
not completely immune to the issue~\cite{burman2011numerical}.

Despite vast literature on the
topic~\cite{Kummer2017,lehrenfeld2016high,guzman2017finite,li2019shifted}, fewer
authors achieve formulations that are fully robust and optimal, regardless of cut
location and material contrast. A notable exception is the family of methods that
rely on \emph{ghost penalty}~\cite{burman_cutfem_2015,gurkan2019stabilized}. These
works adopt approach (1) and enrich the variational formulation with suitable
stabilization terms defined in the faces of cut cells; the resulting formulation
is robust to cut location. Besides, robustness w.r.t.~material contrast is
achieved by using the so-called harmonic weights in the Nitsche formulation, a
typical approach in body-fitted \ac{dg} methods~\cite{codina2013design}. As a
result, the condition number of the diagonally-scaled system matrix becomes
independent of the material
contrast~\cite{burman_cutfem_2015,burman2011numerical}. However, research in this
area has tended to overlook scalability and $hp$-adaptivity, which are essential
aspects in applications to large-scale problems. These aspects have been
considered by the finite cell method
community~\cite{ruess2014weak,elhaddad2018multi}, but robustness w.r.t.~material
contrast has barely received their attention.

Research over the past few years is turning to an alternative approach to ensure
robustness with respect to cut location, the so-called \emph{cell aggregation} or
\emph{cell agglomeration} techniques. This approach \myadded{is very natural in
\ac{dg} methods, as they can be easily formulated on agglomerated
meshes~\cite{helzel_high-resolution_2005,Kummer2017,bastian2009unfitted}.} The
extension of these ideas to conforming discretisations is less obvious, since such
aggregation process requires to keep trace continuity among cells. With this aim,
the (\ac{cg}) aggregated unfitted \ac{fem}, referred to as
Ag\ac{fem}~\cite{Badia2018}, is grounded on a \emph{discrete extension} operator
from well-posed to ill-posed \acp{dof}. This operator is defined in terms of a
cell aggregation and is amenable to arbitrarily  complex 3D geometries and
$h$-adaptivity~\cite{inpreparation2020}. In spite of this, research has been
restricted so far to unfitted boundary elliptic~\cite{Badia2018} or
Stokes~\cite{Badia2018a} problems. Aggregation has also been used for
\ac{cg}~\cite{huang2017unfitted} methods, but the resulting scheme relies on the
assumption that the aggregates can always be rectangles. However, such assumption
is wrong, even in two-dimensions; aggregates have more complicated shapes in
general geometries and meshes. The authors in~\cite{huang2017unfitted} picked an
elementary 2D circular Poisson problem in a square with a circular inclusion,
discretised with a uniform Cartesian grid, in a mesh in which rectangular
aggregates only where possible. Aggregation has been recently employed for
hybrid-high order (HHO)~\cite{burman2019unfitted}, even though these methods are
not considering face aggregation strategies and thus, their trace unknowns can
lead to ill-posed problems. 

The main goal of this work is to present a novel aggregated \ac{fe} method for
interface elliptic \acp{bvp}. In contrast with other existing \myadded{\ac{cg}}
methods, we clearly show that interface Ag\ac{fem} enjoys \emph{overall}
well-behaved numerical properties and remarkable large-scale capability. In
particular, we demonstrate, with theoretical results and thorough numerical
experimentation, well-posedness, robustness w.r.t.~to cut location and material
contrast, optimal ($h$-adaptive) approximation properties, high scalability and
ease of implementation in \ac{hpc} \ac{fe} codes. The paper gives full insight
into Ag\ac{fem}, as a large-scale \ac{fe} solver for complex
\replaced{interface}{multiphase and multiphysics} problems modelled by \acp{pde}.
It is also intended to provide guidance in exploiting other unfitted \ac{cg}
methods by aggregation for interface problems.

The outline of this work is as follows. We assume first an embedded (multiple)
$n$-interface geometrical setting in Section~\ref{sec:embedded-geom}. Next, we
extend the single-domain cell aggregation method in~\cite{Badia2018} to
$n$-interface problems, in Section~\ref{sec:interface-ag}. Cell aggregation can be
carried out independently on each subdomain and reuse, with little effort,
existing distributed-memory implementations of the single-domain
algorithm~\cite{Verdugo2019}. In Section~\ref{sec:ag-fe-space}, we define
Ag\ac{fe} spaces for embedded $n$-interfaces; we see that they easily accommodate
the interface-overlapping mesh approach in~\cite{hansbo2002unfitted}. Afterwards,
we restrict ourselves to the approximation of single interface linear elasticity
problems, see Section~\ref{sec:model-problem}. We derive a similar formulation to
body-fitted \ac{dg} methods~\cite{arnold2002unified}, using the symmetric interior
penalty method and harmonic average weights, to weakly enforce interface
conditions, see Section~\ref{sec:dis-for}. Numerical analysis, proving
well-posedness and a priori error estimates, are also covered there; all results
are stable \myadded{with respect} to cut location and material contrast. We
implement the method in the large-scale \ac{fe} software package
\fempar~\cite{badia-fempar}, which exploits the highly-scalable forest-of-tree
mesh engine \p4est~\cite{burstedde_p4est_2011} for $h$-adaptivity. In the
numerical tests of Section~\ref{sec:num-exps}, we consider both the linear
elasticity and Poisson equations as model problems on several complex geometries
and several $hp$-\ac{fem} standard benchmarks. We numerically assess optimal
convergence rates on uniform and $h$-adaptive meshes, robustness \myadded{with
respect} to cut location and material contrast, and weak-scalability. Finally, we
report the main conclusions and contributions of the work in
Section~\ref{sec:conclusions}.

\section{The aggregated unfitted finite element method on interface problems}
\label{sec:interface-agfem}

\subsection{Embedded interface geometry setup}
\label{sec:embedded-geom}

Let $\Omega \subset \mathbb{R}^d$, with $d = 2,3$ denoting the space dimension, be
an open, bounded, connected domain, with Lipschitz boundary $\partial \Omega$.
Since we seek to analyse problems with multiple physics and/or phases, let $\{
\Omega^i \}_{i=1}^N$ be a partition of $\Omega$ into $N$ subdomains $\Omega^i$
with Lipschitz boundaries $\partial \Omega^i$. Let now $\Gamma_0 \doteq
\bigcup_{i=1}^N \partial \Omega^i \setminus \partial \Omega$ denote the skeleton
of the partition. Equivalently, there is a partition of $\Gamma_0$ into
$\Gamma^{ij} \doteq \partial \Omega^i \cap \partial \Omega^j$, such that $\Gamma_0
\doteq \bigcup_{i,j=1}^N \Gamma^{ij}$. Let $N_0$ denote the number of non-empty
$\Gamma^{ij}$, for all $i,j = 1,\ldots,N$. The setting is represented in
Figure~\ref{fig:emb-int-a}.

\begin{figure}[ht!]
	\centering
	\begin{subfigure}{\textwidth}
		\vspace{0.1cm}
		\centering
		{\Large \color{blue} $\bullet$} $\Sigma_\W^i$ \enskip
		{\color[RGB]{255,0,255} $\boldsymbol{\times}$} $\Sigma_\I^i$ \qquad
		\fbox{{\color[RGB]{255,230,128} \rule{10pt}{8pt}}} $\T_{h,\W}^1$ \enskip
		\fbox{{\color[RGB]{255,246,213} \rule{10pt}{8pt}}} $\T_{h,\I}^1$ \qquad
		\fbox{{\color[RGB]{255,179,128} \rule{10pt}{8pt}}} $\T_{h,\W}^2$ \enskip
		\fbox{{\color[RGB]{255,230,213} \rule{10pt}{8pt}}} $\T_{h,\I}^2$ \qquad
		\fbox{{\color[RGB]{229,255,128} \rule{10pt}{8pt}}} $\T_{h,\W}^3$ \enskip
		\fbox{{\color[RGB]{246,255,213} \rule{10pt}{8pt}}} $\T_{h,\I}^3$ \qquad
		\fbox{{\color{white} \rule{10pt}{8pt}}} $\T_{h,\E}^i$
	\end{subfigure} \\
	\begin{subfigure}{0.24\textwidth}
	  \centering
	  \includegraphics[width=0.9\textwidth]{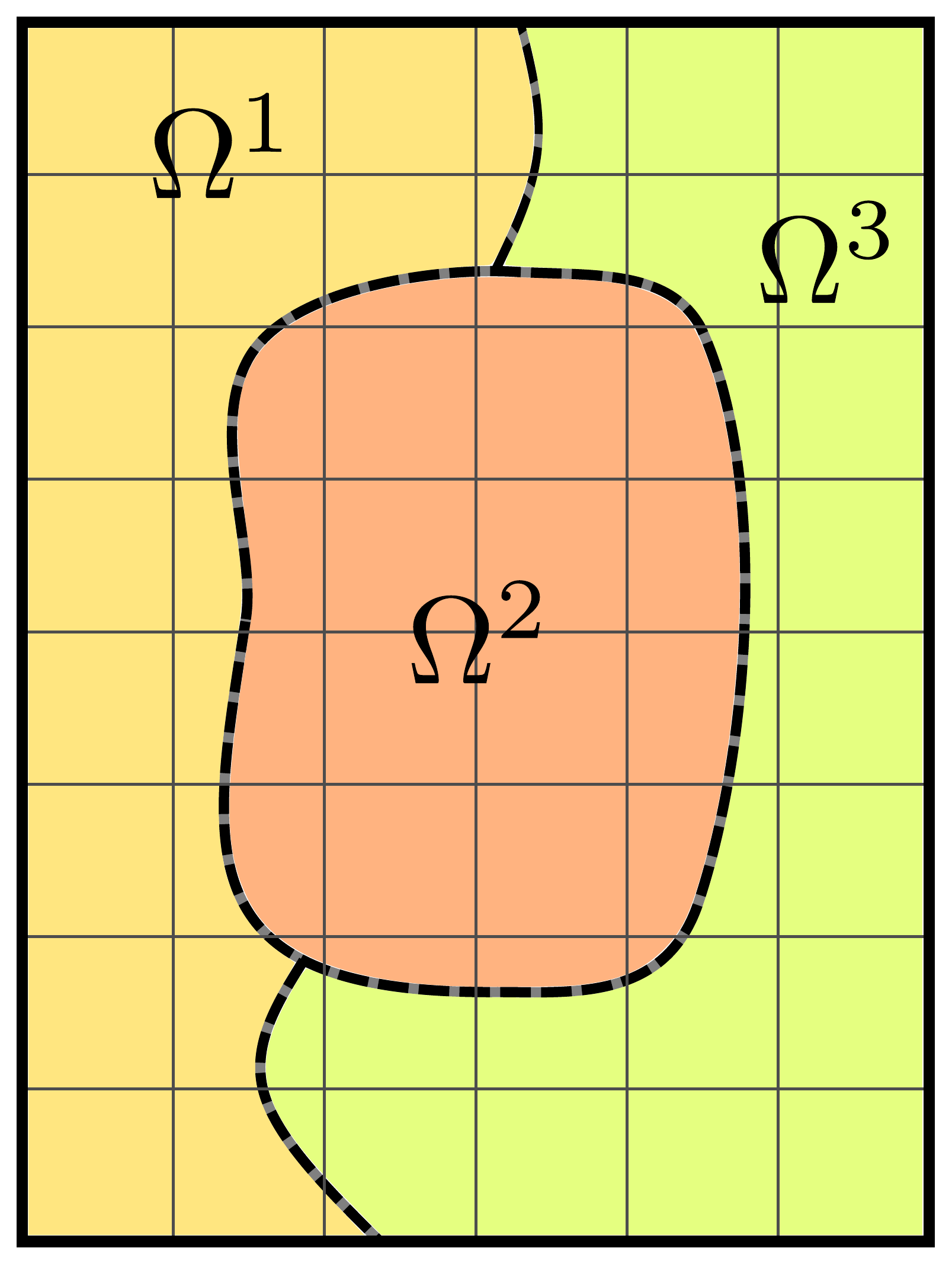}
	  \caption{$\Omega$, $\T_h$ and (dashed) $\Gamma_0$}
	  \label{fig:emb-int-a}
	\end{subfigure}
	\begin{subfigure}{0.24\textwidth}
	  \centering
	  \includegraphics[width=0.9\textwidth]{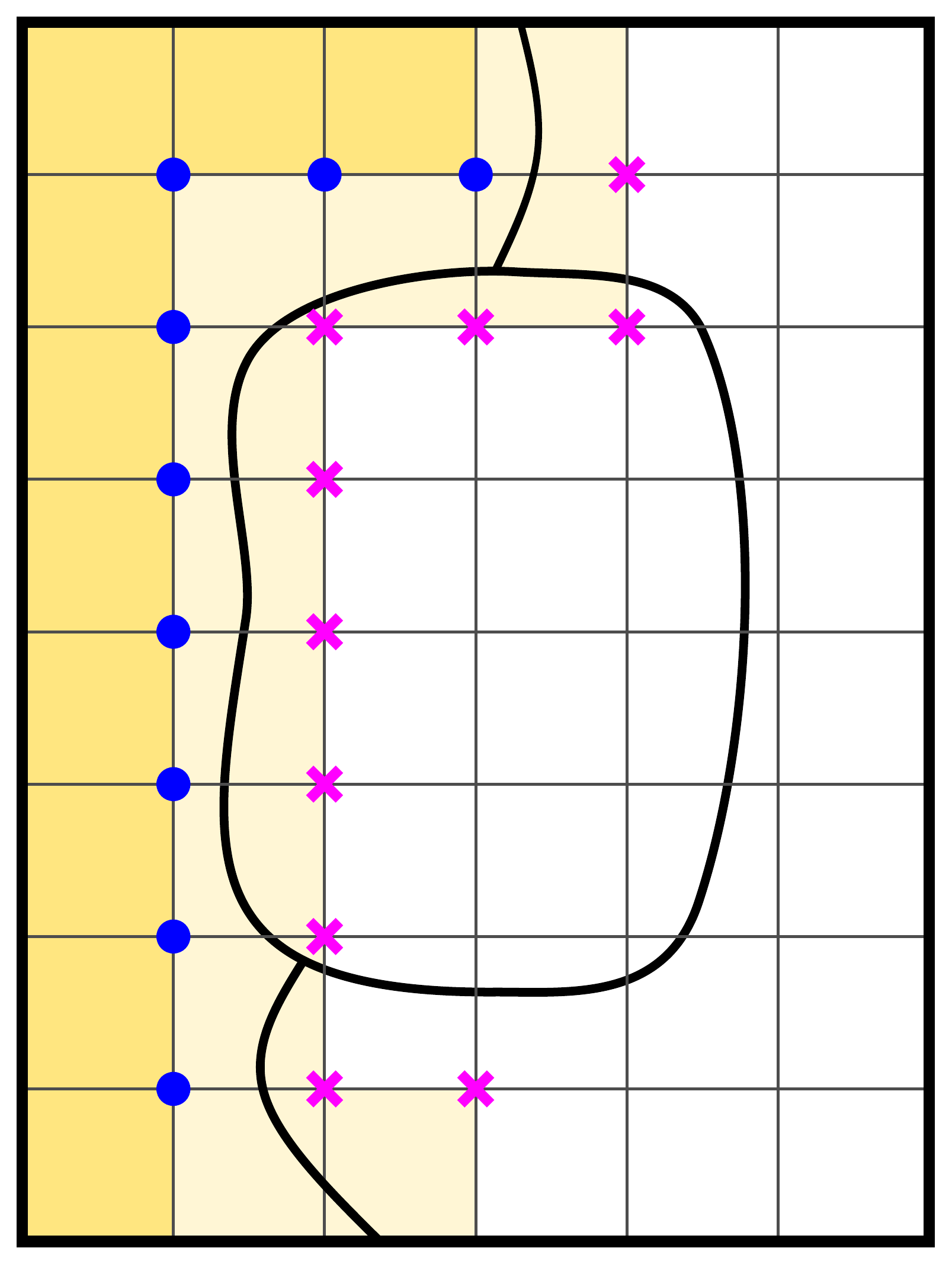}
	  \caption{$\T_{h,\act}^1$ and $\V_{h,\act}^1$}
	  \label{fig:emb-int-b}
	\end{subfigure}
	\begin{subfigure}{0.24\textwidth}
	  \centering
	  \includegraphics[width=0.9\textwidth]{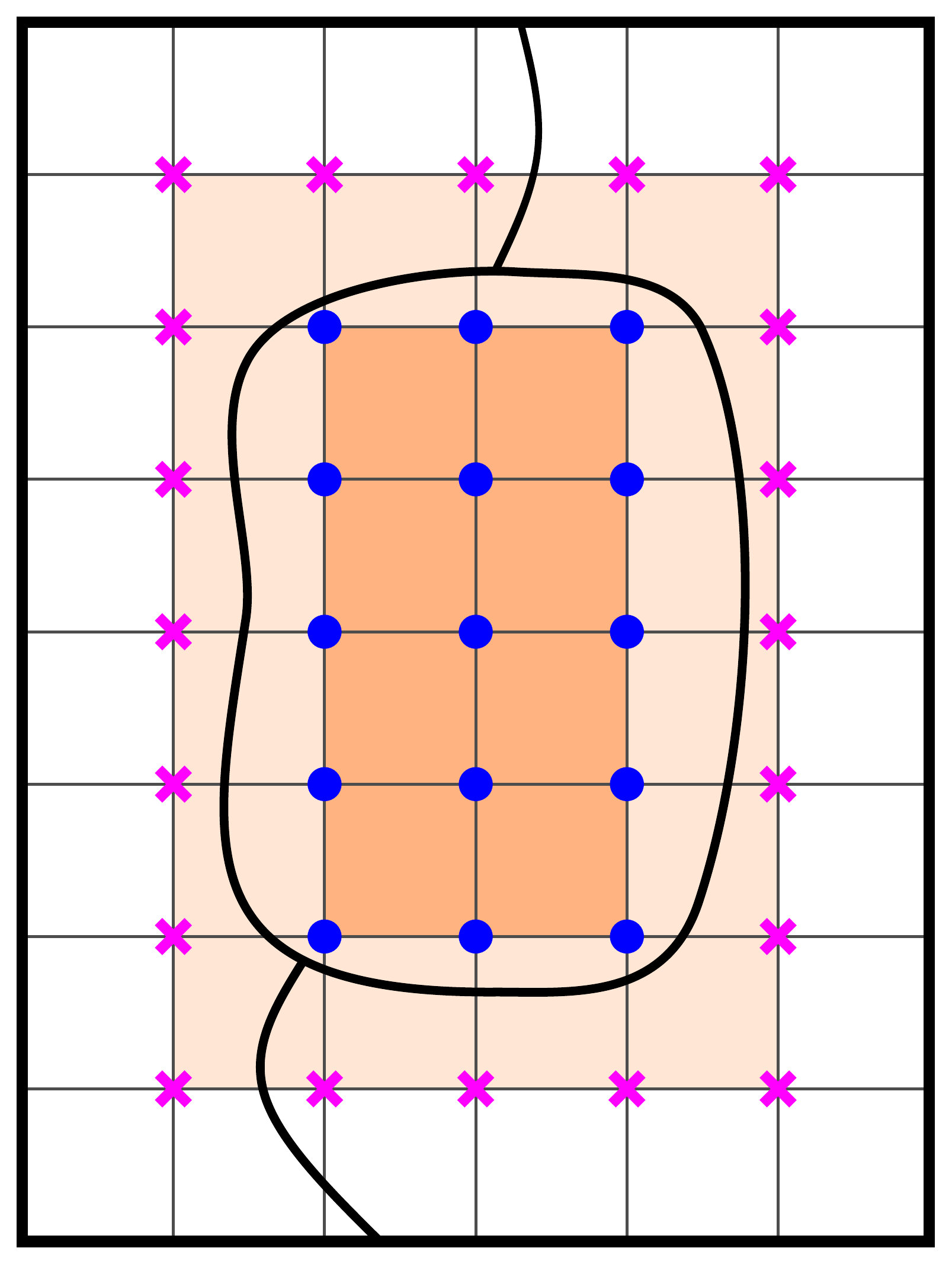}
	  \caption{$\T_{h,\act}^2$ and $\V_{h,\act}^2$}
	  \label{fig:emb-int-c}
	\end{subfigure}
	\begin{subfigure}{0.24\textwidth}
	  \centering
	  \includegraphics[width=0.9\textwidth]{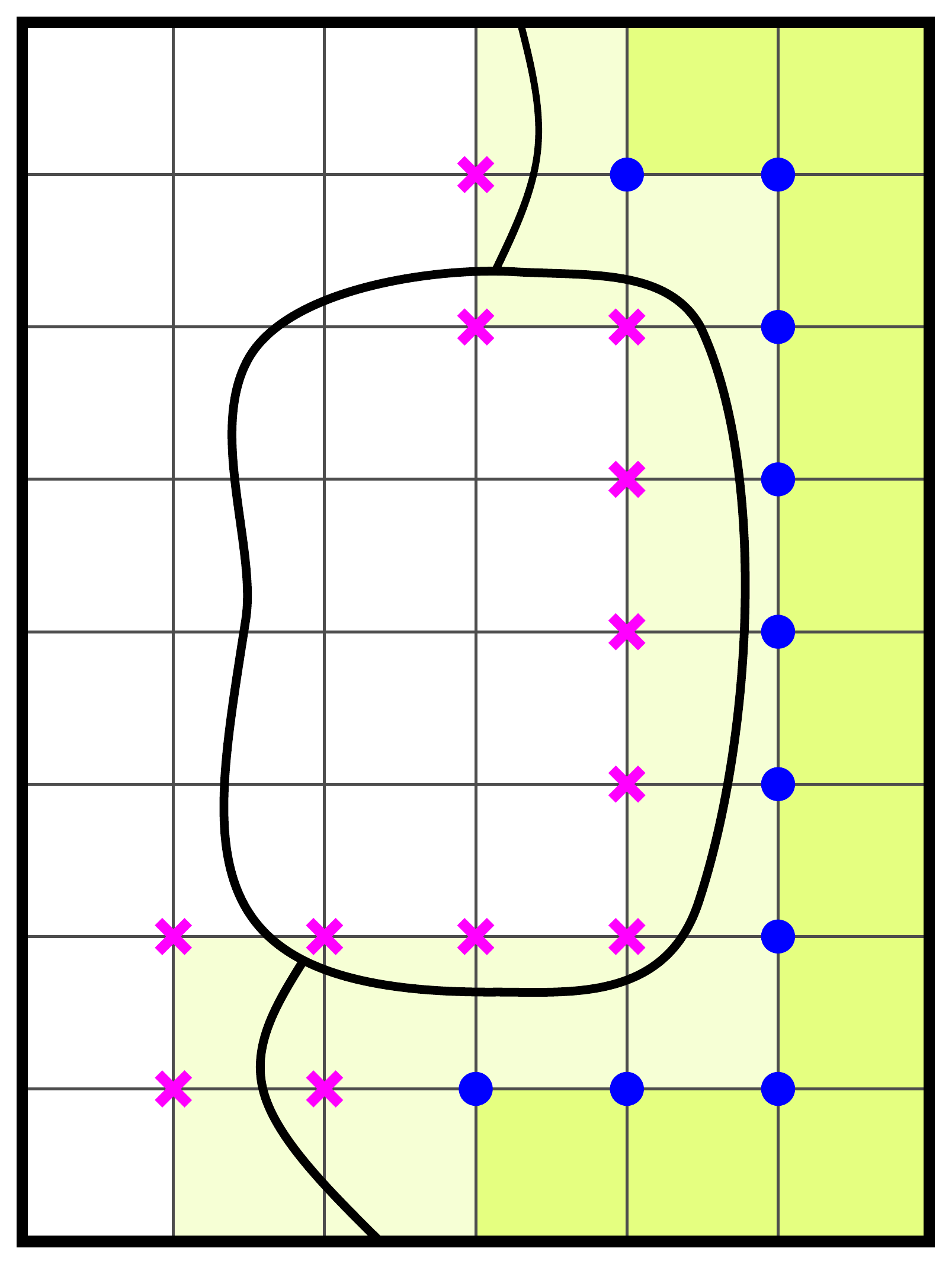}
	  \caption{$\T_{h,\act}^3$ and $\V_{h,\act}^3$}
	  \label{fig:emb-int-d}
	\end{subfigure}
	\caption[An embedded interface geometry setup for $N = 3$ and $\eta_0 = 1$]{An
	embedded interface geometry setup for $N = 3$ and $\eta_0 = 1$,
	i.e.~well-posed if and only if interior and ill-posed if and only if cut. The
	boundary of the physical domain $\partial \Omega$ conforms to the mesh $\T_h$,
	whereas the skeleton $\Gamma_0$ is immersed in it. $\{\T_{h,\act}^i
	\}_{i=1}^3$ forms a partition of $\T_h$, \emph{overlapping} at cells cut by
	the skeleton $\Gamma_0$. As a result, degrees of freedom on cut cells are
	doubled or tripled (assuming linear lagrangian \acp{fe}). We consider
	partitions of \acp{dof} in $\V_{h,\act}^i$ into well-posed $\Sigma_\W^i$ and
	ill-posed $\Sigma_\I^i$ \acp{dof} { (note that we omit Dirichlet \acp{dof})}.
	Ill-posed \acp{dof} are constrained in terms of well-posed \acp{dof}, see
	Equation~\eqref{eq:agfem-constraint} and Figure~\ref{fig:agfem_constr}.}
	\label{fig:emb-int}
\end{figure}

We introduce now a typical embedded interface setup. To focus on the interface
problem, we assume that $\Omega$ can be easily meshed with, e.g.~Cartesian grids
or unstructured $d$-simplexes, such that \emph{the external boundary $\partial
\Omega$ conforms to the mesh, whereas $\Gamma_0$ remains immersed}, as shown
in~Figure~\ref{fig:emb-int-a}. For simplicity in the exposition, let us consider
that the mesh is body-fitted with respect to $\partial \Omega$, even though the
general case can readily be tackled using the techniques in~\cite{Badia2018}.
Instead, in this article, we focus on the extension of these techniques to
resolve immersed interfaces. According to this, let $\T_h$ be a partition of
$\Omega$ into cells{, the so-called \emph{background} mesh}. Any $\cell \in
\T_h$ is the image of a differentiable homeomorphism $\Phi_\cell$ over a set of
admissible open reference $d$-polytopes~\cite{badia-fempar}, such as
$d$-simplexes or $d$-cubes. We let $\T_h$ be \emph{non-conforming}, i.e.~there
can be hanging vertices, edges or faces. We assume that the mesh is
\emph{shape-regular} and $h_\cell$ represents the characteristic size of the
cell $\cell \in \T_{h}$.

We assume, without loss of generality, that the immersed skeleton $\Gamma_0$ is
represented by the zero level-set of one or several known scalar functions, the
so-called level-set functions, or by other means, e.g.~from 3D CAD data, using
techniques to compute the intersection between cell edges and surfaces (see,
e.g.~\cite{marco_exact_2015}). We also assume that we have suitable techniques
(e.g.~for local integration) to deal with cells that are intersected by more
than one interface $\Gamma^{ij}$. For any cell $\cell \in \T_h$, we define the
quantity 
\[
	\eta_\cell^i \doteq \frac{\meas_d(\cell \cap \Omega^i)}{\meas_d(\cell)},
	\quad \eta_\cell \in [0,1], \quad i = 1,\ldots,N,
\]
and a user-defined parameter $\eta_0 \in (0,1]$, referred to as the
\emph{well-posedness threshold}. To isolate badly cut cells, we classify cells
of $\T_h$ in terms of $\eta_\cell^i$ and $\eta_0$; it leads to subsets of $\T_h$
of the form
\[
	\T_{h,\W}^i = \{ \cell \in \T_h : \eta_\cell^i \geq \eta_0  \}, \quad
	\T_{h,\I}^i = \{ \cell \in \T_h : \eta_0 > \eta_\cell^i > 0 \}, \quad
	\T_{h,\E}^i = \{ \cell \in \T_h : \eta_\cell^i = 0          \},
\]
for $i = 1,\ldots,N$. $\T_{h,\W}^i$, $\T_{h,\I}^i$ and $\T_{h,\E}^i$ are the
well-posed ($\W$), ill-posed ($\I$) and exterior cells ($\E$) associated with
subdomain $\Omega^i$. $\T_{h,\W}^i$ contains interior cells or those with a
large portion inside $\Omega^i$, $\T_{h,\I}^i$, those with small cut portions in
$\Omega^i$, and $\T_{h,\E}^i$ those with empty intersection with $\Omega^i$. We
remark that, for $\eta_0 = 1$, well- or ill-posed cells coincide with interior
or cut cells. By definition, each triplet $\{ \T_{h,\W}^i, \T_{h,\I}^i,
\T_{h,\E}^i \}$, $i = 1,\ldots,N$, forms a nonoverlapping partition of $\T_h$.
We denote the union of cells of $\T_{h,\W}^i$, $\T_{h,\I}^i$ and $\T_{h,\E}^i$
by $\Omega_\W^i$, $\Omega_\I^i$ and $\Omega_\E^i$, e.g.~$\Omega_\W^i =
\bigcup_{\cell \in \T_{h,\W}^i} \overline{\cell}$. We also introduce the
\emph{active} meshes and domains, given by $\T_{h,\act}^i \doteq \T_{h,\W}^i
\cup \T_{h,\I}^i$ and $\Omega_\act^i \doteq \Omega_\W^i \cup \Omega_\I^i$, $i =
1,\ldots,N$; note that $\Omega^i \subset \Omega_\act^i$. It follows that
$\{\T_{h,\act}^i \}_{i=1}^N$ is a partition of $\T_h$, \emph{overlapping} at
cells cut by the skeleton $\Gamma_0$, see
Figures~\ref{fig:emb-int-b}-\ref{fig:emb-int-c}-\ref{fig:emb-int-d}. We observe
that our geometrical configuration generalises to multiple interfaces the
classical approach adopted in,
e.g.~\cite{burman_cutfem_2015,hansbo2002unfitted}, for single interface
problems. Indeed, for $N = 2$ ($N_0 = 1$), $\{ \T_{h,\act}^i \}_{i=1}^2$ is an
overlapping partition of $\T_h$, that divides the mesh into two (sub)meshes,
where cells cut by the interface are doubled.

\subsection{Cell aggregation with multiple interfaces}
\label{sec:interface-ag}

Cell aggregation for single-domain problems is well-covered in previous works,
e.g.~\cite{Badia2018}; here, we limit ourselves to lay out the extension of the
rationale to problems posed in domains with multiple interfaces, introduced in
Section~\ref{sec:embedded-geom}. We recall that aggregated \ac{fe} spaces are
grounded on a map, the so-called \emph{root cell map}. This map associates any
ill-posed cell with a well-posed cell, by means of a cell aggregation scheme,
described in, e.g.~\cite[Algorithm 2.2]{inpreparation2020}.

In our context, we assume we carry out cell aggregation independently on each
active mesh $\T_{h,\act}^i$, $i = 1,\ldots,N$, as illustrated in
Figure~\ref{fig:agg-int}; it yields the $i$-th root cell maps $\R^i :
\T_{h,\act}^i \to \T_{h,\W}^i$. For any $\cell \in \T_{h,\W}^i$, we refer to
$A_\cell^i \doteq (\R^i)^{-1}(\cell)$ as a cell aggregate rooted at $\cell$. By
construction of $\R^i$, aggregates take the form $A_\cell^i = \{ \cell_j\}_{0 \leq
j \leq m_\cell}$, where $\cell_0 = \cell \in \T_{h,\W}^i$ and $\cell_j \in
\T_{h,\I}^i$, $1 \leq j \leq m_\cell$, i.e.~they are composed of several ill-posed
cells and a unique (root) well-posed cell. Furthermore, aggregates are connected;
they are also disjoint in $\T_{h,\act}^i$, i.e.~for any $\cell,\cell' \in
\T_{h,\act}^i$, we have that $A_{\R^i(\cell)} \cap A_{\R^i(\cell')} = \emptyset$
or $\R^i(\cell) \equiv \R^i(\cell')$. It follows that $\T_{h,\ag}^i \doteq \{
A_\cell^i \}_{\cell \in \T_{h,\W}^i}$ are partitions of $\T_{h,\act}^i$ into cell
aggregates, for all $i = 1,\ldots,N$. We observe that cell aggregation schemes
only use the local information of each $\T_{h,\act}^i$, $i = 1,\ldots,N$; there is
no coupling between active (sub)meshes. As a result, implementation of a
multiple-domain cell aggregation scheme can fully reuse a single-domain
counterpart.

\begin{figure}[ht!]
	\centering
	\begin{subfigure}{\textwidth}
		\vspace{0.1cm}
		\centering
		$\T_{h,\ag}^1$:
		\fbox{{\color[RGB]{255,230,128} \rule{10pt}{8pt}}} aggregated \enskip
		\fbox{{\color[RGB]{255,246,213} \rule{10pt}{8pt}}} not aggregated \quad
		$\T_{h,\ag}^2$:
		\fbox{{\color[RGB]{255,179,128} \rule{10pt}{8pt}}} agg. \enskip
		\fbox{{\color[RGB]{255,230,213} \rule{10pt}{8pt}}} not agg. \quad
		$\T_{h,\ag}^3$:
		\fbox{{\color[RGB]{229,255,128} \rule{10pt}{8pt}}} agg. \enskip
		\fbox{{\color[RGB]{246,255,213} \rule{10pt}{8pt}}} not agg. \quad
		\fbox{{\color{white} \rule{10pt}{8pt}}} $\T_{h,\E}^i$
	\end{subfigure} \\ \vspace{0.1cm}
	\begin{subfigure}{0.105\textwidth}
		\centering
		\includegraphics[width=\textwidth]{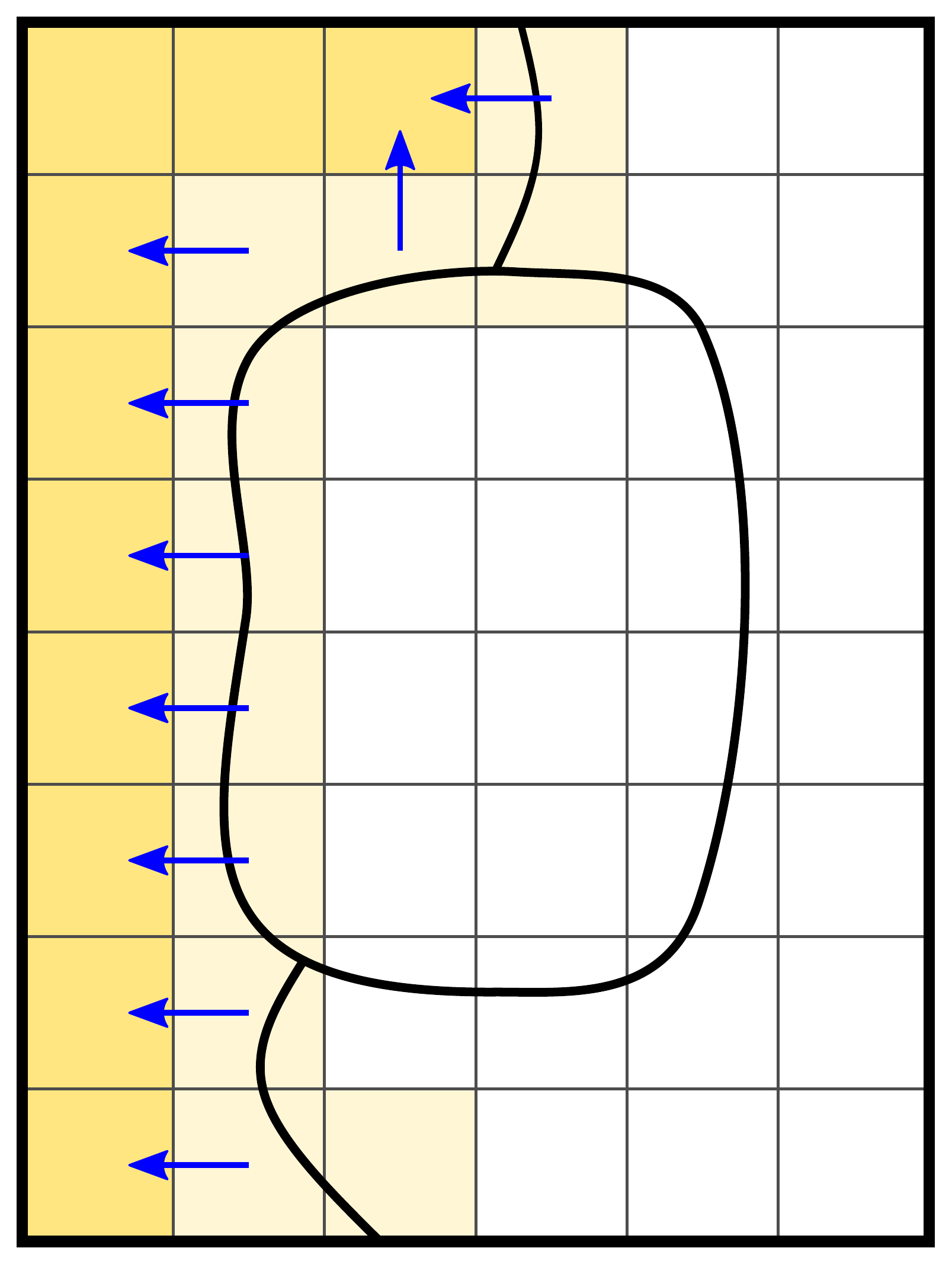}
		\caption{Step 1}
	\end{subfigure}
	\begin{subfigure}{0.105\textwidth}
		\centering
		\includegraphics[width=\textwidth]{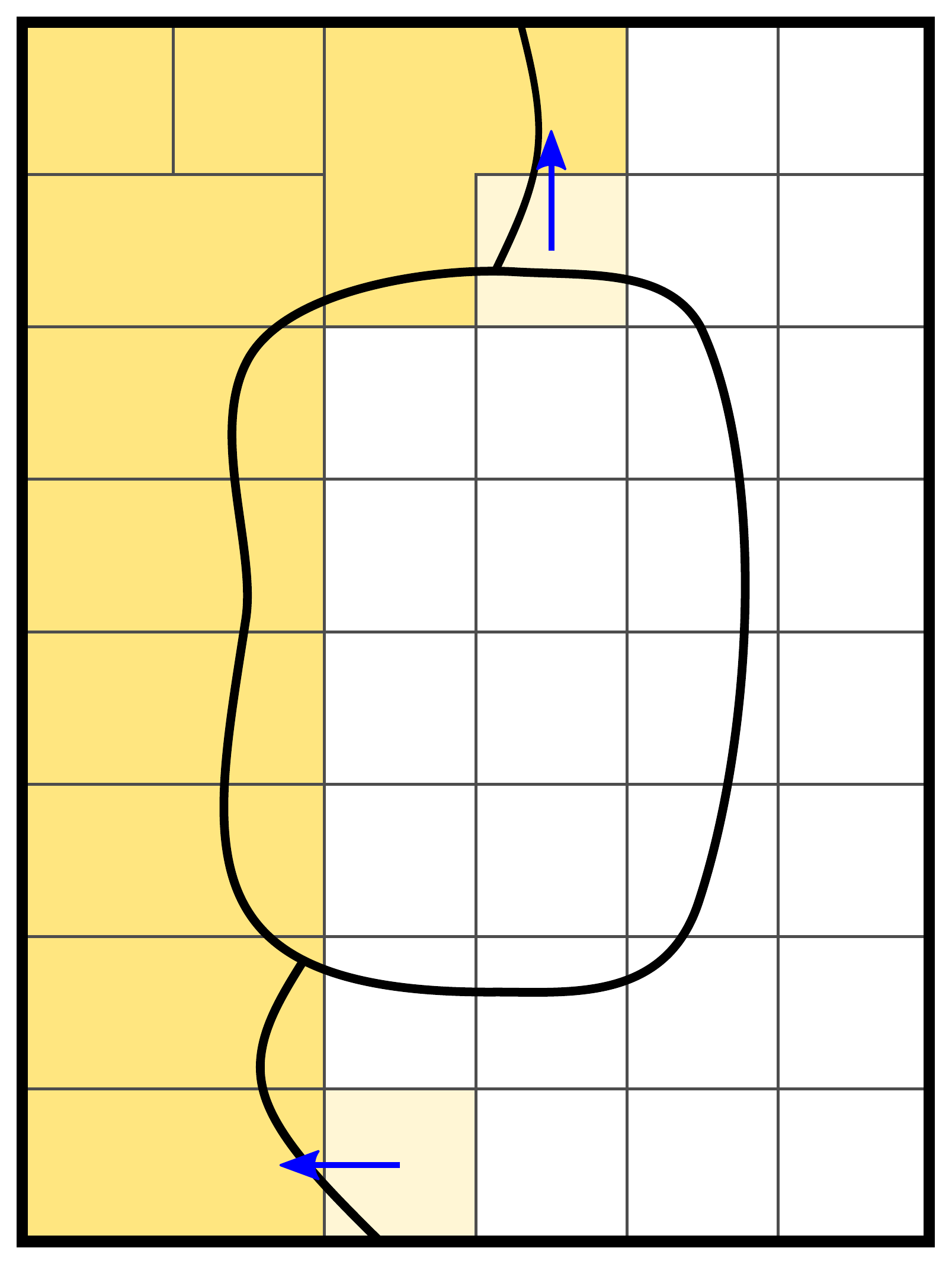}
		\caption{Step 2}
	\end{subfigure}
	\begin{subfigure}{0.105\textwidth}
		\centering
		\includegraphics[width=\textwidth]{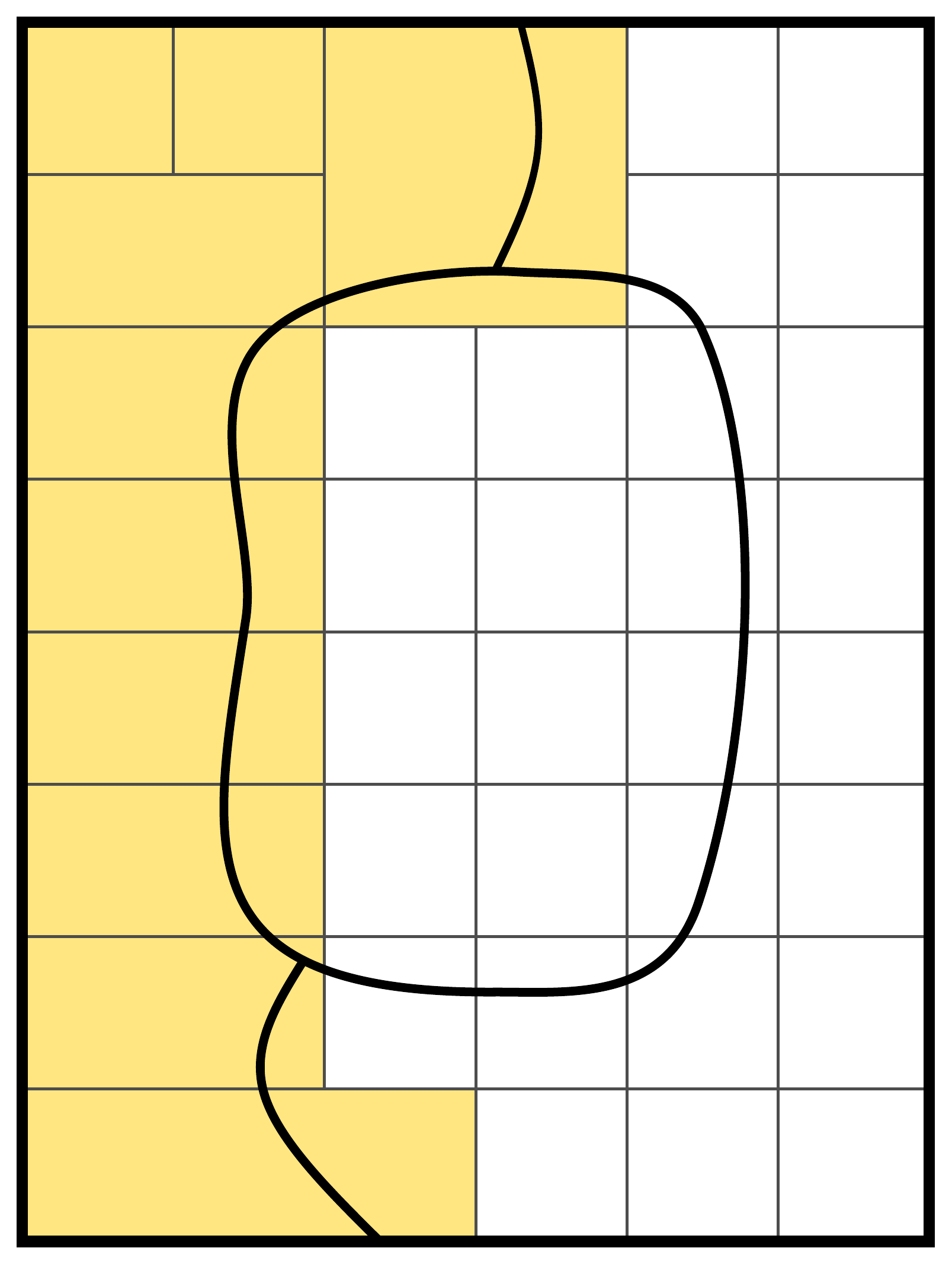}
		\caption{Step 3}
	\end{subfigure}
	\begin{subfigure}{0.105\textwidth}
		\centering
		\includegraphics[width=\textwidth]{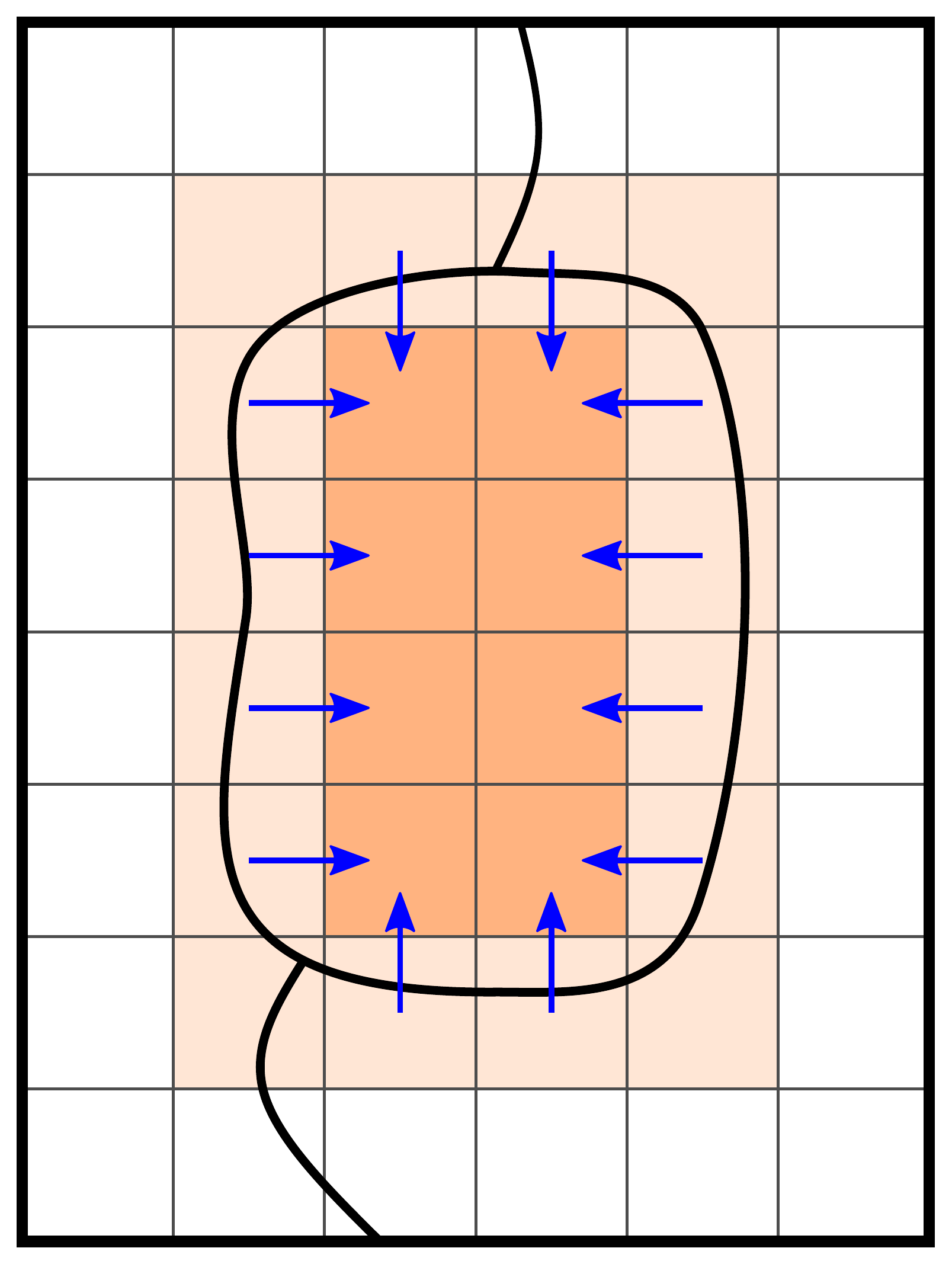}
		\caption{Step 1}
	\end{subfigure}
	\begin{subfigure}{0.105\textwidth}
		\centering
		\includegraphics[width=\textwidth]{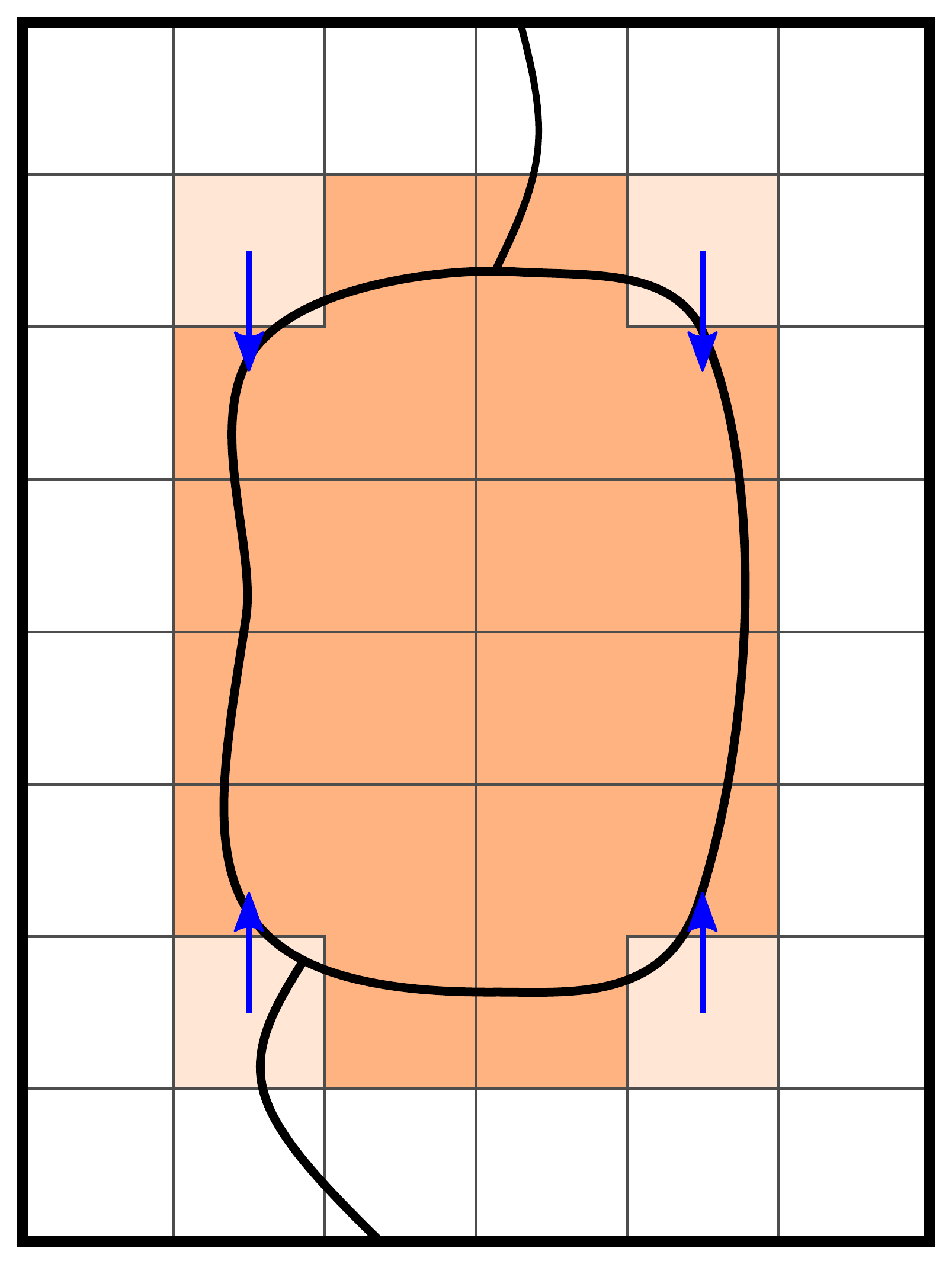}
		\caption{Step 2}
	\end{subfigure}
	\begin{subfigure}{0.105\textwidth}
		\centering
		\includegraphics[width=\textwidth]{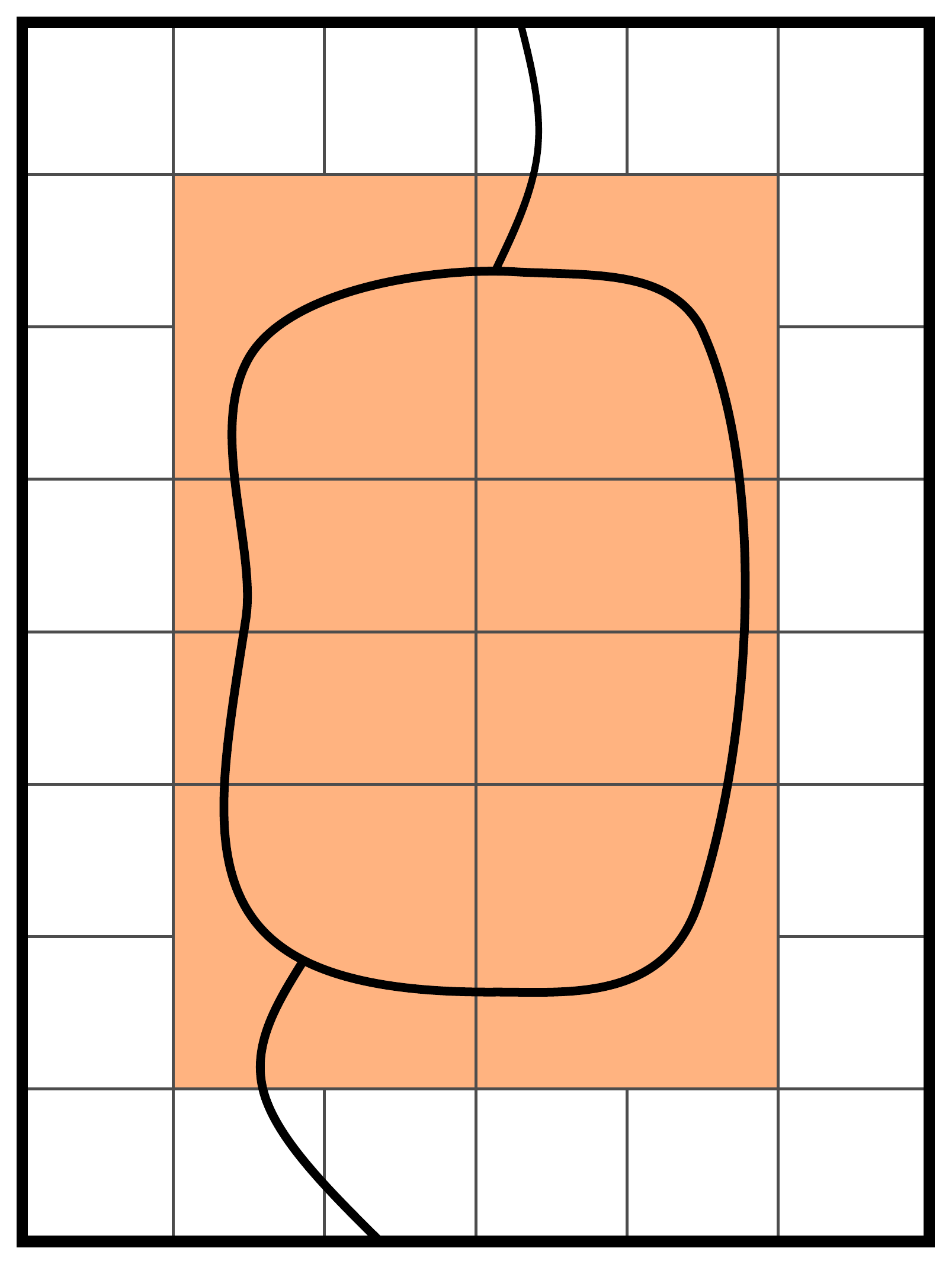}
		\caption{Step 3}
	\end{subfigure}
	\begin{subfigure}{0.105\textwidth}
		\centering
		\includegraphics[width=\textwidth]{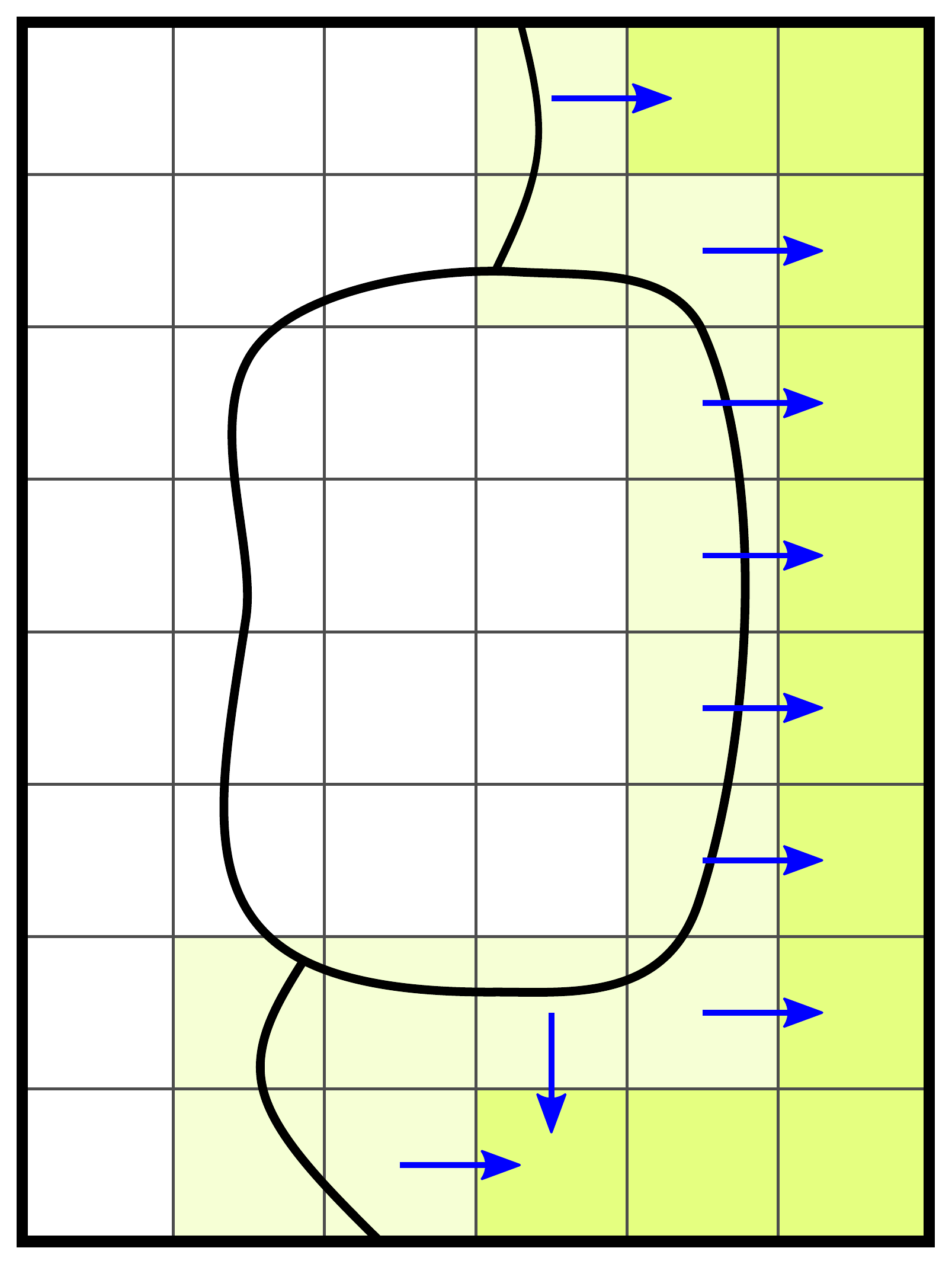}
		\caption{Step 1}
	\end{subfigure}
	\begin{subfigure}{0.105\textwidth}
		\centering
		\includegraphics[width=\textwidth]{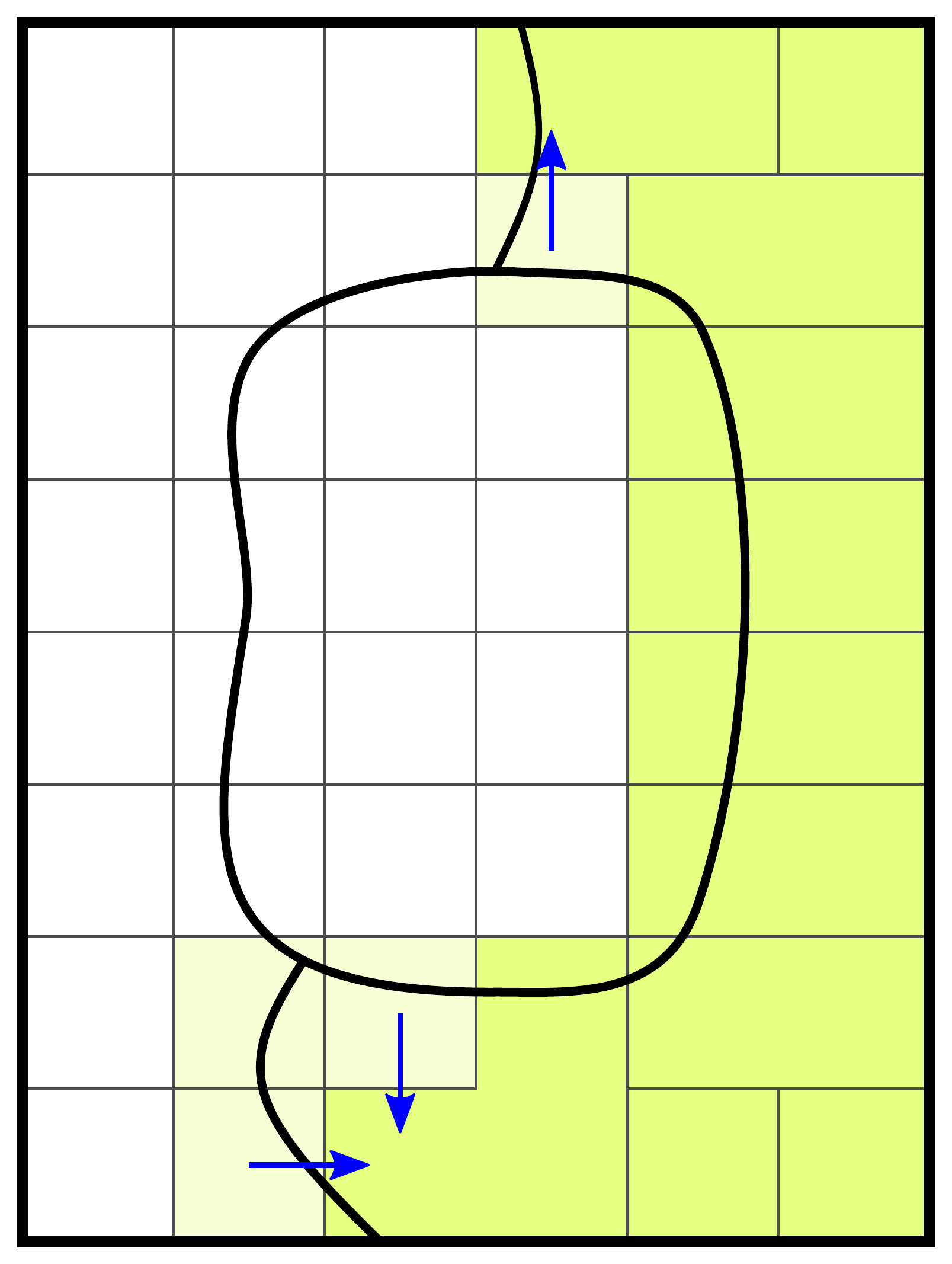}
		\caption{Step 2}
	\end{subfigure}
	\begin{subfigure}{0.105\textwidth}
		\centering
		\includegraphics[width=\textwidth]{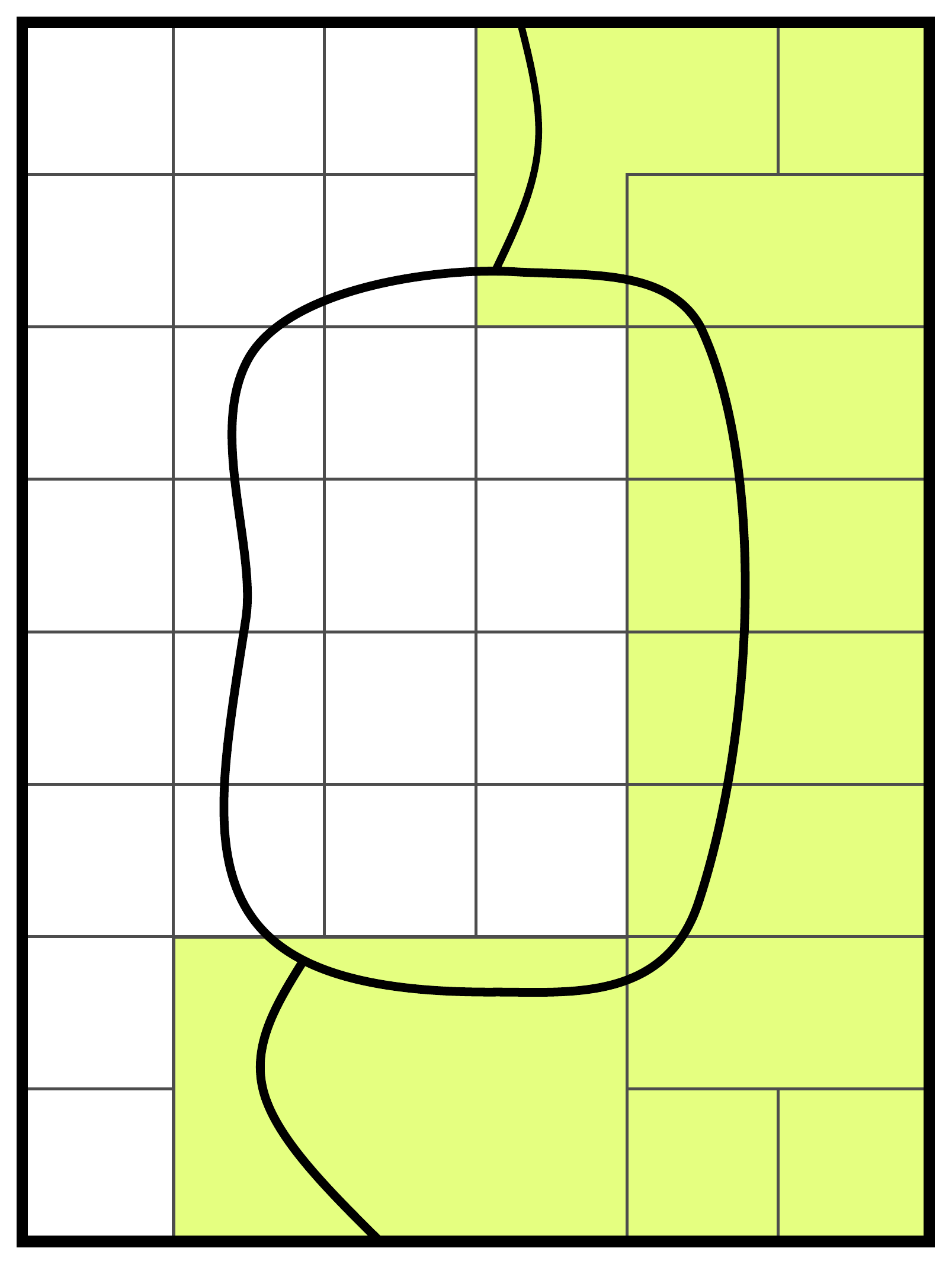}
		\caption{Step 3}
	\end{subfigure}
	\caption[Cell aggregation on the three active meshes $\T_{h,\act}^i$, $i =
	1,2,3$, of Figure~\ref{fig:emb-int}]{Cell aggregation on the three active
	meshes $\T_{h,\act}^i$, $i = 1,2,3$, of Figure~\ref{fig:emb-int}. The
	algorithm is detailed in, e.g.~\cite[Algorithm 2.2]{inpreparation2020}.
	First, it marks well-posed cells as individual aggregates (Step 1). Then,
	aggregates grow iteratively, by attaching adjacent ill-posed cells to them
	(Step 2). The procedure stops when $\T_{h,\act}^i$, $i = 1,2,3$ is covered
	by aggregates (Step 3). This operation gives the root cell map $\R^i$, $i =
	1,2,3$. We observe that the scheme runs independently on each mesh with the
	local information provided by $\T_{h,\act}^i$, $i = 1,2,3$. Hence,
	implementation can reuse a single-domain cell aggregation scheme.}
	\label{fig:agg-int}
\end{figure}


\subsection{Aggregated Lagrangian finite element spaces}
\label{sec:ag-fe-space}

As stated in Section~\ref{sec:intro}, we consider the common
approach~\cite{hansbo2002unfitted,burman_cutfem_2015,annavarapu2012robust} of
building \ac{fe} spaces on top of interface-overlapping meshes; it leads to
\ac{fe} approximations that have a Cartesian product structure. In our case, we
aim to construct a \ac{cg} Ag\ac{fe} space on top of the aggregated overlapping
mesh $\{ \T_{h,\ag}^i \}_{i=1}^N$. We will see that we can straightforwardly
exploit the single-domain methodology in~\cite{Badia2018} to derive an Ag\ac{fe}
space on each aggregated mesh $\T_{h,\ag}^i$ and, from here, a global \ac{fe}
space in $\T_h$ of the form $\V_{h,\ag}^1 \times \ldots \times \V_{h,\ag}^N$.

For the sake of simplicity, we assume that the \ac{pde} problem posed in
$\Omega$ is such that there is a single scalar-valued field associated to each
subdomain $\Omega^i$. We also assume discretisations with Lagrangian \acp{fe}.
In any case, the exposition can be generalised to other \acp{fe},
e.g.~Nédélec~\cite{Olm2018}, vector/tensor fields and multiple fields per
$\Omega^i$. We also consider same cell topology everywhere in $\T_h$ and $\T_h$
conforming; although Ag\ac{fe} spaces on top of nonconforming meshes are fully
covered in~\cite{inpreparation2020} and numerical tests in
Section~\ref{sec:num-exps} run on Cartesian tree-based (nonconforming)
meshes~\cite{Badia2019b}. Lastly, we omit treatment of strong Dirichlet boundary
conditions in the discussion below, although they can be easily taken care of,
using standard approaches.

We denote by $\V(\cell)$ a vector space of functions defined on $\cell \in \T_h$.
For $d$-simplex meshes, we define the local space $\V(\cell) \doteq
\mathcal{P}_q(\cell)$, i.e.~the space of polynomials of order less or equal to
$q$ in the variables $x_1, \ldots, x_d$. For $d$-cubes, we define $\V(\cell)
\doteq \mathcal{Q}_q(\cell)$, i.e.~the space of polynomials that are of degree
less or equal to $q$ with respect to each variable in $x_1, \ldots, x_d$. In the
numerical examples, we limit ourselves to rectangular or hexahedral cells and
linear or quadratic shape functions, i.e.~$\V(\cell) \doteq
\mathcal{Q}_1(\cell)$ or $\V(\cell) \doteq \mathcal{Q}_2(\cell)$. To simplify
notation, we define the elemental functional spaces $\V(\cell)$ in the physical
cell $\cell \subset \Omega$ (even though our computer implementation relies on
reference parametric spaces, as usual). Since we take on Lagrangian \acp{fe},
the basis for $\V(\cell)$ is the Lagrangian basis (of order $q$) on $\cell$; we
assume same order everywhere in $\T_h$. We denote by $\Sigma_\cell$ the set of
Lagrangian nodes of order $q$ of cell $\cell$, i.e.~the set of local \acp{dof}
in $\cell$. There is a one-to-one mapping between nodes $\sigma \in
\Sigma_\cell$ and shape functions $\phi^\sigma_\cell(\x)$ such that
$\phi_\cell^\sigma (\x^{\sigma'}) = \delta_{\sigma\sigma'}$, where
$\x^{\sigma'}$ are the space coordinates of node $\sigma'$ and $\delta$ is the
Kronecker delta.

Since we seek a global aggregated \ac{fe} space of the form $\V_{h,\ag}^1 \times
\ldots \times \V_{h,\ag}^N$, we start by defining the subdomain members
$\V_{h,\ag}^i$, $i = 1, \ldots, N$. Thus, all notation and definitions in the
next paragraphs are subdomain-local, i.e.~referred to any subdomain $\Omega^i
\subset \Omega$, $i = 1,\ldots,N$, unless  stated otherwise. According to this,
let $\Sigma_\act^i$ refer to the set of (subdomain-)active \acp{dof} of
$\T_{h,\act}^i$. We introduce next a local-to-subdomain \ac{dof} map $\sigma^i
(\cell,\sigma') \in \Sigma_\act^i$, with $\sigma' \in \Sigma_\cell$ and $\cell
\in \T_h$. In \ac{cg} methods, $\sigma^i$ is obtained by gluing together \acp{dof}
located in the same geometrical position; this operation leads to
$\C^0$-continuous approximations. With this notation, we can define a standard
\ac{fe} space in $\T_{h,\act}^i$ of the form
\[
	\V_{h,\act}^i \doteq \{ v^i \in \mathcal{C}^0(\Omega_\act^i) : 
	\restrict{v^i}{\cell} \in \V(\cell), \ \forall \ \cell \in \T_{h,\act}^i \}.
\]
It is well-known that, when the discrete \ac{fe} problem is \emph{only}
integrated in $\Omega^i$, \emph{direct} usage of $\V_{h,\act}^i$ leads to
arbitrarily ill-conditioned linear systems~\cite{DePrenter2017}. To solve this
issue, we resort to the aggregated \ac{fem}~\cite{Badia2018,Badia2018a}. The
main idea is to remove from $\V_{h,\act}^i$ problematic \acp{dof}, associated
with small cut cells, by constraining them as a linear combination of \acp{dof}
with local support in a (well-posed) cell of $\T_{h,\W}^i$. It leads to the
aggregated subspace of $\V_{h,\act}^i$, namely $\V_{h,\ag}^i$, that gets rid of
the aforementioned ill-conditioning issues.

In order to define $\V_{h,\ag}^i$, the key is to realise that our context is
analogous to one considering a single-domain unfitted-boundary problem, taking
$\Omega^i$ as the physical domain embedded in $\Omega$. The former case is
extensively covered in~\cite{Badia2018}. Hence, we can follow the same steps to
derive $\V_{h,\ag}^i$. According to this, let us define the set of well-posed
\acp{dof} as $\Sigma_\W^i \doteq \bigcup_{\cell \in \T_{h,\W}^i} \Sigma_\cell$
and the set of ill-posed \acp{dof} as $\Sigma_\I^i \doteq \Sigma_\act^i
\setminus \Sigma_\W^i$, see Figure~\ref{fig:emb-int}. Obviously, $\{
\Sigma_\W^i, \Sigma_\I^i \}$ forms a partition of $\Sigma_\act^i$. $\Sigma_\W^i$
gathers all \acp{dof} that have local support in (well-posed) cells of
$\T_{h,\W}^i$, while $\Sigma_\I^i$ isolates all \acp{dof}, that potentially have
arbitrarily small compact support and must be constrained in terms of well-posed
\acp{dof} of $\Sigma_\W^i$.

To compute ill-posed \ac{dof} constraints, we proceed as usual in Ag\ac{fe}
methods. First, we compose the root cell map $\R^i : \T_{h,\act}^i \to
\T_{h,\W}^i$ of Section~\ref{sec:interface-ag}, with a map between ill-posed
\acp{dof} $\Sigma_\I^i$ and ill-posed cells $\T_{h,\I}^i$. Specifically, we
assign each ill-posed \ac{dof} to one of its surrounding ill-posed cells. The
chosen cell is then mapped onto a well-posed cell via $\R^i$. Thus, the outcome
of this composition is a map $\mathcal{K}^i : \Sigma_\I^i \to \T_{h,\W}^i$, that
assigns an ill-posed \ac{dof} to a well-posed cell via cell aggregation; see
formal definitions in, e.g.~\cite{Verdugo2019,Badia2018}. Following this, given
$v^i \in \V_{h,\act}^i$ and $\sigma \in \Sigma_\I^i$, we linearly extrapolate the
nodal value of $\sigma$, namely $v_\sigma^i \in \mathbb{R}$, with the values at
the local \acp{dof} of its root cell $\mathcal{K}^i(\sigma)$. It leads to the
constraint (see Figure~\ref{fig:agfem_constr})
\begin{equation}
	v_\sigma^i = \sum_{\sigma' \in \Sigma_{\mathcal{K}^i(\sigma)}}
	C_{\sigma\sigma'} v_{\sigma'}^i, \quad \text{with} \
	C_{\sigma\sigma'} \doteq \phi_{\mathcal{K}^i(\sigma)}^{\sigma'}
	(\x^{\sigma}).
	\label{eq:agfem-constraint}
\end{equation}
As a result, the Ag\ac{fe} space can be readily defined as
\[
	\V_{h,\ag}^i \doteq \{ v^i \in \V_{h,\act}^i : v_\sigma^i = \sum_{\sigma' 
	\in \Sigma_{\mathcal{K}^i(\sigma)}} C_{\sigma\sigma'} v_{\sigma'}^i, 
	\enspace \forall \sigma \in \Sigma_\I^i \}.
\]
It is clear that $\V_{h,\ag}^i \subset \V_{h,\act}^i$. Further details, such
as the form of (subdomain-wise) shape functions of $\V_{h,\act}^i$, are not
covered here, as they are analogous to those in~\cite{Badia2018}.

\begin{figure}[!h]
	\centering
	\includegraphics[width=0.2\textwidth,trim = 0.0cm 0.75cm 0.0cm 0.0cm, clip = true]{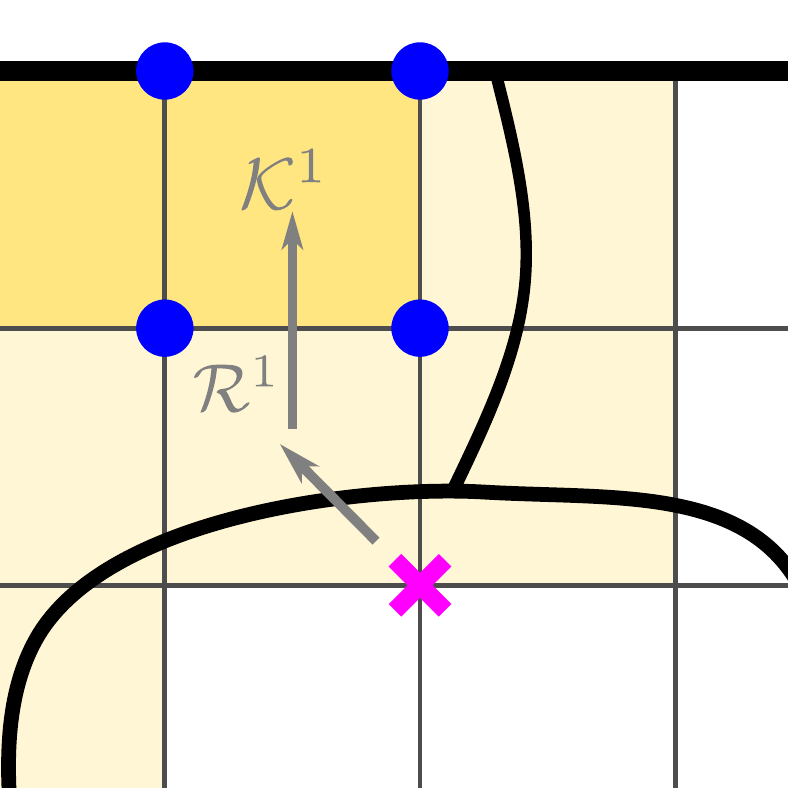}
	\vspace{-0.15cm}
	\caption[Close-up of Figure~\ref{fig:emb-int-b} illustrating an ill-posed
	\ac{dof} mapped to a well-posed cell]{Close-up of Figure~\ref{fig:emb-int-b}
	illustrating an ill-posed \ac{dof} ({\color[RGB]{255,0,255}
	$\boldsymbol{\times}$}) in $\T_{h,\act}^1$ mapped to a well-posed cell via
	$\mathcal{K}^1$. The resulting constraining \acp{dof},
	i.e.~$\Sigma_{\mathcal{K}^i}$, are marked with {\Large
	\color{blue} $\bullet$}.}
	\label{fig:agfem_constr}
\end{figure}

After defining independent Ag\ac{fe} spaces in $\Omega^i$, $i = 1, \ldots, N$, a
\emph{global} aggregated \ac{fe} space $\V_{h,\ag}$ is straightforwardly derived
as the Cartesian product of subdomain counterparts, i.e.~$\V_{h,\ag} \doteq
\V_{h,\ag}^1 \times \ldots \times \V_{h,\ag}^N$. We remark that, as
$\T_{h,\act}^i$ overlaps in cells cutting the skeleton $\Gamma_0$, \acp{dof}
lying on a cut cell are mapped to as many different global \acp{dof}, as active
meshes overlapping the cell, via the local-to-subdomain \ac{dof} map $\sigma^i$.
However, some replicated \acp{dof} may be marked as ill-posed and become
constrained. As a result, they do not increase the size of the linear
system.\footnote{In this sense, Ag\ac{fem} departs from other unfitted
techniques that rely on the same interface-overlapping mesh approach, such as
cutFEM. In those cases, the problem is incremented by the number of replicated
\acp{dof}. In particular, the total number of (free) \acp{dof} is $\sum_{i=1,N} 
\left| \Sigma_\act^i \right|$. In contrast, the size of the linear system in
Ag\ac{fem} is always smaller and \emph{bounded above} by $\sum_{i=1,N} \left|
\Sigma_\act^i \right|$; indeed, the total number of \acp{dof} is regulated by
the well-posedness threshold $\eta_0$. For $\eta_0$ equal to zero, we would
exactly have $\sum_{i=1,N} \left| \Sigma_\act^i \right|$, but this is the
standard XFEM case, which is useless because it does not get rid of the small
cut cell problem. The larger $\eta_0$ is, the more cells are marked as ill-posed
and thus the number of \acp{dof} reduced, because more \acp{dof} are constrained
and do not appear in the (reduced) linear system. In the aggregation process,
replicated \acp{dof} on the interface cells are eliminated and one can easily
end up with a problem even smaller than the original \ac{fe} problem. In any
case, the interface region usually demands more refined meshes due to small
scale local effects. This is accomplished by combining Ag\ac{fem} with adaptive
mesh refinement and coarsening (see~\cite{inpreparation2020}).}

\section{Approximation of unfitted interface elliptic problems}
\label{sec:approx}

In this section, we address the approximation of compressible linear elasticity
problems with the Ag\ac{fem}. \myadded{Extension of the method below to truly
incompressible materials can be carried out with the mixed Ag\ac{fem}
in~\cite{Badia2018a}.} We introduce first the continuous interface
problem~\eqref{eq:strong_form} and prove that the weak
formulation~\eqref{eq:weak_form1}-\eqref{eq:weak_form2} is well-posed. Afterwards,
we consider a consistent Nitsche's method~\eqref{eq:discrete_form} to discretise
the problem with Ag\ac{fem}. We conclude by examining well-posedness and
approximability properties of the discrete problem~\eqref{eq:discrete_form}, which
lead to optimal \emph{a priori} error estimates \myadded{independent of the cut
configuration}.

From this point onwards, we restrict ourselves to single interface problems with
two subdomains, i.e.~there is a unique physical interface $\Gamma_0 \equiv
\Gamma^{12}$; henceforth denoted simply by $\Gamma$. This assumption contributes
to conciseness and readability; all concepts presented here can be easily
extended to the general case with an arbitrary number of subdomains. For the
sake of the numerical analysis, let $\Gamma$ be a smooth manifold with bounded
curvature. To distinguish the two subdomains, we use superscripts $+,-$ instead
of $1,2$, e.g.~the subdomains are denoted by $\Omega^+$ and $\Omega^-$. In
addition, we employ superscript $\alpha \in \{+,-\}$ to refer to any of the
subdomains and $\pm$ to refer to the broken domain, i.e.~$\Omega^\pm \doteq
\Omega^+ \cup \Omega^-$.

Before describing the model problem and approximation, we introduce some
additional notation. Let $\v$ be a smooth enough vector or tensor function defined
in $\Omega$. We denote by $\v^\alpha \doteq \restrict{\v}{\Omega^\alpha}$ the
restriction of $\v$ into $\Omega^\alpha$; conversely, given $\v^\alpha$ defined in
$\Omega^\alpha$, we identify the pair $\{\v^+,\v^-\}$ with the function $\v$ in
$\Omega^\pm$, that is equal to $\v^\alpha$ in $\Omega^\alpha$. On the interface,
we define $\restrict{\v^+}{\Gamma}(\x) = \lim_{\epsilon\to0^{+}}
\v(\x-\epsilon\n^{+})$ and $\restrict{\v^-}{\Gamma}(\x) = \lim_{\epsilon\to0^{-}}
\v(\x+\epsilon\n^{-})$, \myadded{where $\n^\alpha$ is the \emph{outward pointing}
normal to $\Omega^\alpha$}. We define the \emph{jump} of $\v$ across $\Gamma$ by
$\jump{\v} \doteq \restrict{\v^+}{\Gamma} - \restrict{\v^-}{\Gamma}$ and the
\emph{weighted average} of $\v$ on $\Gamma$ as $\average{\v} \doteq w^+
\restrict{\v^+}{\Gamma} + w^- \restrict{\v^-}{\Gamma}$, with $0 \leq w^\alpha \leq
1$ and $w^+ + w^- = 1$.

On the other hand, we use standard notation for Sobolev spaces (see,
e.g.~\cite{Toselli2005}). For instance, the $L^2(\omega)$ norm is denoted by
$\norm{\cdot}{L^2(\omega)}{}$, the $H^1(\omega)$ norm as
$\norm{\cdot}{H^1(\omega)}{}$ and the $H^1(\omega)$ seminorm as
$\seminorm{\cdot}{H^1(\omega)}{}$. Given the two disjoint open connected
subdomains $\Omega^+,\Omega^- \subset \mathbb{R}^d$, the Sobolev spaces of the
form $H^s(\Omega^+) \times H^s(\Omega^-)$ are represented with $H^s(\Omega^\pm)$,
endowed with the norm $\norm{\cdot}{H^s(\Omega^{\pm})}{} \doteq
(\norm{\cdot}{H^s(\Omega^+)}{2} + \norm{\cdot}{H^s(\Omega^-)}{2})^{1/2}$;
analogously for seminorms. Vector-valued Sobolev spaces are represented with
boldface letters. \myadded{We use common notation $A \lesssim B$ or $A \gtrsim B$
to denote that $A \leq CB$ or $A \geq CB$ for some positive constant $C$.} In this
work, constants may depend on the order of the \ac{fe} space, \myadded{the shape
and size of $\Omega$ and $\Gamma$,} and the user-defined value $\eta_{0}$, but
they may not depend on the mesh-interface intersection (i.e.~how the cells are
intersected), the mesh size of the background mesh, or the contrast of the
physical parameters at both sides of the interface.

Moreover, let us assume that the aggregate size is bounded by a constant times
$h_{T}$, where $T$ is the root of the aggregate. This can be shown to hold when
assuming that the ratio between the size of two neighbouring cells cannot be
arbitrarily large, e.g.~using standard 2:1 balance in adaptive non-conforming
tree meshes or a patch-local quasi-regularity assumption on unstructured meshes
(see also~\cite[Lemma 2.2]{Badia2018}). 

Lastly, we introduce the set of faces $\mathcal{F}_{h}$ that are generated after
the intersection of $\Gamma$ and the mesh $\mathcal{T}_{h}$,
i.e.~$\mathcal{F}_{h} \doteq \left\{\bigcup_{\cell \in \T_h} \Gamma \cap \cell
\right\} \cup \left\{\bigcup_{\cell,\cell' \in \T_h \ : \ \cell \neq \cell'}
\Gamma \cap \left( \overline{\cell} \cap \overline{\cell'} \right) \right\}$; a
face $F$ in $\mathcal{F}_{h}$ can be on the boundary of the background mesh
cells or intersect the cells. In the subsequent analysis, there is no difference
between the two cases and, thus, we do not distinguish among them. Given $F \in
\mathcal{F}_h$, we let $\cell_F^\alpha \in \T_{h,\act}^\alpha$ such that $F \cap
\overline{\Omega^\alpha} \subset \overline{\cell_F^\alpha}$ and $h_{\cell_F}
\doteq \max \{h_{\cell_F^+}, h_{\cell_F^-} \}$. Note that $\cell_F^+ \equiv
\cell_F^-$ for faces that intersect the cells.

Our main goal is to prove that all constants being used in the analysis are
independent of $h$ and the cell-interface intersection. They may depend, though,
on the well-posedness threshold $\eta_0$, \myadded{the shape and size of $\Omega$
and $\Gamma$}, and the order of the \ac{fe} approximation. The key strategy in the
analysis, in order to prove robustness w.r.t.~the small cut cell problem, is to build
upon well-behaved properties, that enjoy Ag\ac{fe} spaces in \acp{bvp} posed on
unfitted boundaries, i.e.~where $\partial \Omega$ is unfitted, instead of
$\Gamma$; these properties have been thoroughly covered
in~\cite{Badia2018,Badia2018a}. We will often refer to them, without repeating
details, to keep the presentation short.

Besides, we also aim to gain some control on the robustness of
method~\eqref{eq:discrete_form} to material contrast. Since we rule out
incompressibility, we adopt the quotient of $\mu$ coefficients at either sides
of $\Gamma$ as the measure of material contrast, i.e.~we consider $\mu_+/\mu_-$
in the numerical experiments. Therefore, we can follow the usual approach for
the Laplacian problem, adopted in body-fitted \ac{dg}~\cite{codina2013design}
and small-cut-stable unfitted~\cite{burman_cutfem_2015} methods. In particular,
we employ the so-called \emph{harmonic} average weights, that is $w_+ \doteq
\frac{\mu_-}{\mu_+ + \mu_-}$ and $w_- \doteq \frac{\mu_+}{\mu_+ + \mu_-}$.
Clearly, $w_\alpha$, $\alpha \in \{+,-\}$, does not depend on cut location, only
on material contrast. We will denote the harmonic average of $\mu$ by
$\overline{\mu} \doteq \frac{2 \mu_+ \mu_-}{\mu_+ + \mu_-}$. We have that
$\mu_{\min} \leq \overline{\mu} \leq \mu_{\max}$ and $\overline{\mu} \leq 2
\mu_{\min}$.

\subsection{Model problem:}
\label{sec:model-problem}

We consider the linear isotropic elasticity problem with discontinuous Lamé
parameters across $\Gamma$, even though the following discussion and analysis can
also be particularised to the Poisson equation, or any other elliptic problem with
$H^1$-stability. We adopt a pure-displacement (irreducible)
model~\cite{ern2013theory}. For simplicity, we assume homogeneous Dirichlet
boundary conditions on $\partial \Omega$, although non-homogeneous Dirichlet or
Neumann boundary conditions can be considered too, using standard arguments. We
also assume non-homogeneous (immersed) interface transmission conditions.
According to this, \myadded{the model
problem~\cite{becker2009nitsche,burman2017deriving}} seeks to find the
displacement field $\u : \brokendomain \to \mathbb{R}^d$ such that
\begin{equation}
	\left\{
	\begin{array}{ll}
		-\bsnabla \cdot \s(\u) = \f  & \text{in} \ \brokendomain, \\
		\u = 0 & \text{on} \ \partial \Omega, \\
		\jump{\u} = \j_\Gamma & \text{on} \ \Gamma, \text{ and} \\
		\jump{\s(\u)} \cdot \n^+ = \g_\Gamma & \text{on} \ \Gamma, \\
	\end{array}
	\right.
	\label{eq:strong_form} 
\end{equation}
where $\eps,\s : \brokendomain \to \mathbb{R}^{d,d}$ are the strain tensor
$\eps(\u) \doteq \frac{1}{2} (\bsnabla \u + {\bsnabla \u}^T)$ and stress tensor
$\s(\u) = 2 \mu \eps(\u) + \lambda \mathrm{tr} (\eps(\u)) \mathbf{Id}$; where
$\mathbf{Id}$ denotes the identity matrix in $\mathbb{R}^\d$. Apart from that,
we let $\f \in \L^{2}(\Omega)$ represent the body forces, whereas $\j_\Gamma$
and $\g_\Gamma$ denote the fixed jump and forcing terms on $\Gamma$. We assume
that $\j_\Gamma \in \H_{00}^{1/2}(\Gamma)$ and $\g_\Gamma \in \H^{1/2}(\Gamma)$.
We recall that $\H_{00}^{1/2}(\Gamma)$ is the subspace of functions in
$\H^{1/2}(\Gamma)$, whose extension by zero on $\partial \Omega$ is in
$\H^{1/2}(\partial \Omega \cup \Gamma)$~\cite[Appendix A.2]{Toselli2005}. Since
$\j_{\Gamma} \in \H_{00}^{1/2}(\Gamma)$, its extension by zero to $\partial
\Omega^{\alpha}$, $\alpha \in \{+,-\}$, is bounded in $\H^{1/2}(\partial
\Omega^\alpha)$, which we represent with $\j_{\partial \Omega^\alpha}$.

We assume the Lamé coefficients to be subdomain constant, i.e.~$\lambda(\x)
\doteq \lambda_\alpha \geq 0$ and $\mu(\x) \doteq \mu_\alpha > 0$ for $\x \in
\Omega^{\alpha}$, $\alpha \in \{ +,-\}$, but can have different values across
$\Gamma$. Furthermore, we consider the Poisson ratio $\nu_\alpha \doteq
\lambda_\alpha / ( 2 ( \lambda_\alpha + \mu_\alpha ) )$ is bounded away from
$1/2$, i.e.~the material is compressible. Since $\lambda_\alpha = 2\nu_\alpha
\mu_\alpha / (1 - 2\nu_\alpha)$, $\lambda_\alpha$ is bounded above by
$\mu_\alpha$, i.e.~$\lambda_\alpha \leq C \mu_\alpha$, $C > 0$. Combined with
the Cauchy-Schwarz inequality, it leads to the upper bound
\begin{equation}\label{eq:cont-strong-form}
  \int_{\Omega^{\alpha}}^{} \s(\u) : \eps(\v) \ \mathrm{d} \Omega  
  \lesssim \mu_\alpha \norm{\bsnabla\u}{\L^2(\Omega^{\alpha})}{}
  \norm{\bsnabla\v}{\L^2(\Omega^{\alpha})}{}, \qquad \forall \u,\v \in 
  \H^{1}(\Omega^{\alpha}). 
\end{equation}
\myadded{On the other hand, letting $\V \doteq \{ \v \in \H^1(\Omega^{\pm}) : \v =
\boldsymbol{0} \ \text{on} \ \partial \Omega \}$, we have the Korn
inequality~\cite[(1.19)]{brenner2004korn}
	\begin{equation}\label{eq:korn}
		\int_{\Omega}^{} \s(\u) : \eps(\u) \ \mathrm{d} \Omega \
		+ \sum_{F \in \mathcal{F}_h} h_{\cell_F}^{-1} \| \jump{\u} 
		\|^2_{\L^2(F)} \geq \sum_{\alpha \in \{+,-\}} 
		C_{\s} C_\Omega \mu_{\alpha} \| \bsnabla \u^\alpha 
		\|^2_{\L^2(\Omega^{\alpha})}, \qquad \forall \u \in \V,
	\end{equation}
where $C_{\Omega}>0$ is the related Korn constant.} We can now
use~\eqref{eq:cont-strong-form}\myadded{,~\eqref{eq:korn}} and the fact that
$\j_{\Gamma} \in \H_{00}^{1/2}(\Gamma)$ to show that the weak form
of~\eqref{eq:strong_form} is well-posed. To this end, we let the decomposition $\u
\doteq \w + \h_\j \in \H^1(\Omega^\pm)$, such that the weak solution
of~\eqref{eq:strong_form} becomes: \emph{find} $\u = \w + \h_\j \in \V$\emph{,
where}
\begin{align}
	\label{eq:weak_form1}
	\h_\j \in \H^1(\Omega^{+}) \ &: \ \int_{\Omega^{+}} \s(\h_\j) : \eps(\v) \ 
	\mathrm{d} \Omega = 0, \qquad \h_\j = \j_{\partial \Omega^{+}} \ \text{in} 
	\ \partial \Omega^+, \qquad \text{and} \\ 
	\w \in \H_{0}^{1}(\Omega) \ &: \ \int_{\Omega} \s(\w) : \eps(\v) \ \mathrm{d} 
	\Omega = -\int_{\Omega^{+}} \s(\h_\j) : \eps(\v) \ \mathrm{d} \Omega \ + 
	\int_{\Omega} \f \cdot \v \ \mathrm{d} \Omega \ + \int_{\Gamma} \g_\Gamma 
	\cdot \v \ \mathrm{d} \Gamma,
	\label{eq:weak_form2}
\end{align}
\emph{for all} $\v \in \H_{0}^{1}(\Omega)$.

Continuity of the bilinear form in $\H^1(\Omega^\pm)$ is a direct consequence
of~\eqref{eq:cont-strong-form}. \myadded{Since the jump term in~\eqref{eq:korn}
vanishes for $\w \in \H_{0}^{1}(\Omega)$, we can combine~\eqref{eq:korn} with the
first Poincaré-Friedrichs inequality to prove coercivity of~\eqref{eq:weak_form2}.
If we consider a continuous lifting of the Dirichlet data $\j_{\partial
\Omega^{+}}$~\cite[Remark A.42]{Toselli2005}, we can
rewrite~\eqref{eq:weak_form1}, as an homogeneous Dirichlet problem, and
apply~\eqref{eq:korn} in $\Omega^+$ (again with null jump term). As a result, we
can repeat the previous argument to show coercivity of~\eqref{eq:weak_form1} in
$\Omega^+$. Thus,} we can readily apply Lax-Milgram's lemma
on~\eqref{eq:weak_form1}, leading to $\| \h_\j \|_{\H^1(\Omega^{+})} \lesssim \|
\j_{\partial \Omega^{+}} \|_{\H^{1/2}(\partial \Omega^{+})} \lesssim \|
\j_{\Gamma} \|_{\H^{1/2}(\Gamma)}$. Finally, continuity of the right-hand side of
\eqref{eq:weak_form2} follows from the Cauchy-Schwarz inequality and a trace
theorem:
\begin{align}
  - \int_{\Omega^{+}}^{} \s(\h_\j) : \eps(\v) \ \mathrm{d} \Omega 
  &\lesssim \mu_{+}^{1/2} \| \j_{\Gamma} \|_{\H^{1/2}(\Gamma)} 
  \| \mu^{1/2} \v \|_{\H^{1}\left(\Omega^{+}\right)}, \\ \int_{\Omega}^{} 
  \f \cdot \v \ \mathrm{d} \Omega & \lesssim \| \mu^{-1/2} \f 
  \|_{\L^{2}(\Omega)} \| \mu^{1/2} \v \|_{\L^{2}(\Omega)}, \\
  \int_{\Gamma}^{} \g_{\Gamma} \cdot \v \ \mathrm{d} \Gamma & \lesssim
  \| \overline{\mu}^{-1/2} \g_{\Gamma} \|_{\L^{2}(\Gamma)} \| 
  \overline{\mu}^{1/2} \v \|_{\L^{2}(\Gamma)} \lesssim \| 
  \overline{\mu}^{-1/2} \g_{\Gamma} \|_{\H^{1/2}(\Gamma)} \| 
  \mu^{1/2} \v \|_{\H^{1}(\Omega^{\pm})}, 
  \qquad \forall \v \in \H_0^1(\Omega).
\end{align}
Combining all these results, existence and uniqueness of the weak solution
to~\eqref{eq:strong_form} is ensured by the Lax-Milgram theorem. Moreover, the
problem is well-posed, since the unique solution is bounded by the data as
follows:
\[
  \| \mu^{1/2} \u \|_{\H^1(\Omega^\pm)} \lesssim \mu_{+}^{1/2} \| \j_{\Gamma}
  \|_{\H^{1/2}(\Gamma)} + \| \mu^{-1/2} \f \|_{\L^{2}(\Omega)} + \| 
  \overline{\mu}^{-1/2} \g_{\Gamma} \|_{\H^{1/2}(\Gamma)}.
\]
\subsection{Discrete formulation}\label{sec:dis-for}

We consider as approximation space of $\mathcal{V}$ the aggregated \ac{fe}
space, see Section~\ref{sec:ag-fe-space},
\[
	\V_h \doteq \{ \v_h \in \V^+_\ag \times \V^-_\ag : 
	\v_h = \boldsymbol{0} \ \text{on} \ \partial \Omega \}.
\]
We consider an approximation of \eqref{eq:weak_form2} with this discrete space,
which reads:
\begin{equation}\label{eq:discrete_form} 
	\u_h \in \V_h \ : \ a_h(\u_h, \v_h) = \ell_h(\v_h), 
	\quad \forall \v_h \in \V_h,
\end{equation}
where the global \ac{fe} operators $a_h$ and $\ell_h$ are given by
\begin{align}
	a_h(\u_h,\v_h) &\doteq \int_{\Omega} \s(\u_h) : \eps(\v_h) \ \mathrm{d} 
	\Omega \\
	& \enskip + \sum_{F \in \mathcal{F}_{h}} \left[ 
	\frac{\beta\overline{\mu}}{h_{\cell_F}} \int_{F} \jump{\u_h} \cdot 
	\jump{\v_h} \ \mathrm{d} \Gamma - \int_{F} \n^+ \cdot \average{\s(\v_h)} 
	\cdot \jump{\u_h} \ \mathrm{d} \Gamma - \int_{F} \n^+ \cdot 
	\average{\s(\u_h)} \cdot \jump{\v_h} \ \mathrm{d} \Gamma  \right], \\
	\ell_h(\v_h) &\doteq \int_{\Omega} \f \cdot \v_h \ \mathrm{d} 
	\Gamma \\ 
	& \enskip + \sum_{F \in \mathcal{F}_{h}} \left[ 
	\frac{\beta\overline{\mu}}{h_{\cell_F}} \int_{F} \j_\Gamma \cdot 
	\jump{\v_h} \ \mathrm{d} \Gamma - \int_{F} \n^+ \cdot \average{\s(\v_h)} 
	\cdot \j_\Gamma \ \mathrm{d} \Gamma + \int_{F} \g_\Gamma \cdot \left( w_- 
	\v_h^+ + w_+ \v_h^- \right) \ \mathrm{d} \Gamma \right].
\end{align}
We observe that $a_h$ and $\ell_h$ contain the usual terms in Nitsche's
formulations, i.e.~terms that weakly impose the interface conditions,
symmetrizing terms and stabilization terms. The latter terms are those
premultiplied by $\beta$, which has to be large-enough to ensure coercivity of
the bilinear form $a_h$. Furthermore, the above formulation is
\emph{consistent}, by the following result:
\begin{lemma}[Consistency]
	Let $\u \in \H^2(\Omega^{\pm}) \cap \V$ solve~\eqref{eq:strong_form}. 
	Then, it holds $a_h(\u,\v_h) = \ell_h(\v_h)$, $\forall \v_h \in \V_h$.
	\label{lem:consistency}
\end{lemma}
\begin{proof}
	Since $\u$ solves problem~\eqref{eq:strong_form} (in a weak sense), 
	integration by parts leads to:
	\begin{align}
		\int_{\Omega} \s(\u) : \eps(\v_h) \ \mathrm{d} \Omega &= 
		- \int_{\brokendomain} \v_h \cdot \bsnabla \cdot \s(\u) \ \mathrm{d} 
		\Omega \\ & \enskip + \int_{\Gamma}^{} \n^ + \cdot \average{\s(\u)} \cdot 
		\jump{\v_h} \ \mathrm{d} \Gamma + \int_{\Gamma} \n^+ \cdot 
		\jump{\s(\u)} \cdot \left( w^- \v_h^+ + w^+ \v_h^- \right) \ 
		\mathrm{d} \Gamma, 
	\end{align}
	for	any $\v_h \in \V_h$. Combining this result with $-\bsnabla
	\cdot \s(\u) = \f$, $\jump{\u} = \j_{\Gamma}$ and $\n^+ \cdot \jump{\s(\u)}
	= \g_{\Gamma}$, we can check that all terms in the discrete
	formulation~\eqref{eq:discrete_form} cancel out. 
\end{proof}

For the sake of proving coercivity, we need the following auxiliary result.
\begin{lemma}\label{lem:bound-l2-cut}
  Let $\cell \in \T_{h,\W}^\alpha$ and $\u_\cell \in \mathcal{Q}_{q}(T)$. There
  exists $C_{\eta_{0}} > 0$, dependent on the well-posedness threshold $\eta_0$,
  such that $\norm{\u_\cell}{\L^2(\cell)}{2} \leq C_{\eta_{0}}
  \norm{\u_\cell}{\L^2(\cell \cap \Omega^\alpha)}{2}$, $\alpha \in \{+,-\}$.
\end{lemma}
\begin{proof}
The proof is direct for \emph{interior} well-posed cells; we restrict ourselves
to well-posed \emph{cut} cells. Let us consider a cell $\cell$ and its interior
portion $\cell \cap \Omega$. Using the inverse of the geometrical map, which
maps $\cell$ into the reference cell $\hat{\cell}$, one can map the interior
portion to the reference cell, which is represented with $\hat{\cell}_{\rm in}$.
It is easy to check that $\meas_d(\hat{\cell}_{\rm in}) \ge C \eta_{0}
\meas_d(\hat{\cell})$. In fact, the constant is 1 for affine maps.
$\norm{\cdot}{\L^2(\hat{T}_{\mathrm{in}})}{2}$ is a norm for
$\mathcal{Q}_{q}(\hat{T})$, since a polynomial that vanishes in a domain of
non-zero measure is equal to zero. We prove the result by using the equivalence
of norms in finite-dimensional vector spaces and a scaling argument.
\end{proof}

Given $\cell \in \T_{h,\act}^\alpha$, $\alpha \in \{+,-\} $, let us denote by
$\cell_1, \ldots, \cell_{n_\cell^\alpha}$, $n_\cell^\alpha \geq 1$, the set of
\emph{constraining} well-posed cells of $\cell$ in $\T_{h,\W}^\alpha$, i.e.~the
set of well-posed cells that constrain at least one \ac{dof} of $\cell$ in
$\T_{h,\W}^\alpha$. \myadded{Given $F \in \mathcal{F}_h$, we recall that
$\cell_F^\alpha$ is the cell in $\T_{h,\act}^\alpha$ satisfying $F \cap
\overline{\Omega^\alpha} \subset \overline{\cell_F^\alpha}$. Hence,} we define
$\Omega_{\cell_F^\alpha} \doteq \Omega^\alpha \cap \left( \cell_F^\alpha \cup
\bigcup_{i=1}^{n_\cell^\alpha} \cell_i \right)$. With this notation, we can state
an inequality for discrete functions in cut cells (see \cite[Lemma
A.7]{inpreparation2020}):
\begin{equation}\label{eq:b3wp}
	\| \bsnabla \v_h^\alpha \|_{\L^2(F)}^2 \lesssim C_{\eta_0} h_{\cell_F^\alpha}^{-1} 
	\| \bsnabla \v_h^\alpha \|_{\L^2(\Omega_{\cell_F^\alpha})}^2, \quad \alpha \in 
	\{+,-\}, \quad \forall \v_h \in \V_h, \quad \forall F \in \mathcal{F}_h.
\end{equation}
We also make use of the following inequality for continuous functions on cut
cells (see~\cite{hansbo2002unfitted}):
\begin{equation}\label{eq:ineqc1}
  \| \psi \|^2_{L^2(\partial(\Omega \cap \cell))} \lesssim h_{\cell}^{-1} \| \psi 
  \|^2_{L^2(\Omega \cap \cell)} + h_{\cell} \left| \psi \right|^{2}_{H^1
  (\Omega \cap \cell)}, \qquad \forall \psi \in H^1(\Omega \cap \cell),
\end{equation}
where $\partial(\Omega \cap T)$ is the boundary of $\Omega\cap T$.\footnote{We
note that the proof in~\cite{hansbo2002unfitted} assumes that $\Omega \cap T$ is
connected, together with the assumption that $\Gamma$ has a bounded curvature.
The connected intersection can be handled either replicating cells (as commented
above) or assuming a \emph{fine enough} mesh.}

Let us define the space $\V(h) \doteq \V_h + \H^2(\Omega^\pm) \cap \V$. We endow
$\V(h)$ with the broken norm:
\[
	\normstab{\v}{2} \doteq \sum_{\alpha \in \{+,-\}} \mu_{\alpha} 
	\| \bsnabla \v^\alpha \|^2_{\L^2(\Omega^{\alpha})} + \sum_{F \in 
	\mathcal{F}_h} \frac{\overline{\mu}}{h_{\cell_F}} \| \jump{\v} 
	\|^2_{\L^2(F)} + \sum_{\alpha \in \{+,-\}} \sum_{\cell \in 
	\T_{h,\act}^{\alpha}} \mu_{\alpha} h_{\cell}^{2} \| \v 
	\|^2_{\H^2(\cell \cap \Omega^\alpha)}.
\]
It can be checked that $\| \v \|_{\L^2(\Omega)} \lesssim \normstab{\v}{}$, for
$\v \in \V(h)$, see, e.g.~\cite[Lemma 5.8]{Badia2018}. The following lemma
restricted to the discrete space $\mathcal{V}_{h}$ provides the well-posedness
of the discrete problem. Its extension to $\mathcal{V}(h)$ will be required in
the convergence analysis.  

\begin{lemma}[Well-posedness]\label{lem:wellposed}
	The bilinear form in the discrete formulation~\eqref{eq:discrete_form}
	satisfies the following properties uniformly w.r.t.~the mesh size $h$ of the
	background mesh and interface intersection:
	\begin{enumerate}
	  \item[(i)] Coercivity: 
	  \[
	    a_h(\u_h,\u_h) \gtrsim \normstab{\u_h}{2}, \qquad \forall \u_h \in \V_h,
	  \]
	  if $\beta > C$, for some (large-enough) positive constant $C$.
	  \vspace{0.5em}
	  \item[(ii)] Continuity:
	  \[
		a_h(\u,\v) \lesssim \normstab{\u}{} \normstab{\v}{}, \qquad \forall 
		\u, \v \in \V(h).
	  \]
	\end{enumerate}
	Therefore, there exists a unique solution to
	problem~\eqref{eq:discrete_form}. 
\end{lemma}
\begin{proof}
	By definition of the bilinear form $a_h$ and~\eqref{eq:korn}, we have that
	\begin{equation}\label{eq:b2-wp}
	  a_h(\u_h,\u_h) \gtrsim \hspace{-0.2cm} \sum_{\alpha \in \{+,-\}} 
	  \hspace{-0.3cm} C_{\s} C_\Omega \mu_{\alpha} \| \bsnabla \u_h^\alpha 
	  \|^2_{\L^2(\Omega^{\alpha})} + \sum_{F \in \mathcal{F}_h} \frac{\beta 
	  \overline{\mu}}{h_{\cell_F}} \| \jump{\u_h} \|^2_{\L^2(F)} - 2 \sum_{F 
	  \in \mathcal{F}_h} \int_{F} \n^{+} \cdot \average{\s(\u_h)} \cdot 
	  \jump{\u_h} \ \mathrm{d} \Gamma,
	\end{equation}
	for any $\u_h \in \V_h$. In order to prove coercivity, we have to bound the
	indefinite term. Let us pick an arbitrary $\u_h \in \V_h$. Using the fact
	that $w_\alpha \mu_\alpha = \overline{\mu}$, Cauchy-Schwarz and triangle 
	inequalities and \eqref{eq:b3wp}, we get 
	\begin{equation}\label{eq:b4wp}
	  \hspace{-0.2cm} \| \n^+ \cdot \average{\s(\u_h)} \|_{\L^2(F)}^2 \lesssim 
	  \hspace{-0.3cm} \sum_{\alpha \in \{+,-\}} \hspace{-0.3cm} w_\alpha^2 
	  \mu_\alpha^2 \| \bsnabla \u_h^\alpha \|_{\L^2(F)}^2 = \overline{\mu}^2 
	  \hspace{-0.3cm} \sum_{\alpha \in \{+,-\}} \hspace{-0.3cm} \| \bsnabla 
	  \u_h^\alpha \|_{\L^2(F)}^{2} \leq C_{\eta_0} \overline{\mu} \hspace{-0.3cm} 
	  \sum_{\alpha \in \{+,-\}} \hspace{-0.1cm} 
	  \frac{\mu_\alpha}{h_{\cell_F^\alpha}} \| \bsnabla 
	  \u_h^\alpha \|_{\L^2(\Omega_{\cell_F^\alpha})}^2.
	\end{equation}
	Usage of the Cauchy-Schwarz and Young inequalities and the previous result 
	leads to
	\begin{equation}\label{eq:b1-wp}
	  \begin{aligned}
		\left| \ 2 \int_{F} \n^+ \cdot \average{\s(\u_h)} \cdot \jump{\u_h} \ 
		\mathrm{d} \Gamma \ \right| &\lesssim \frac{h_{\cell_F}}{\gamma 
		\overline{\mu}} \| \n^+ \cdot \average{\s(\u_h)} \|_{\L^2(F)}^2 + 
		\frac{\gamma \overline{\mu}}{h_{\cell_F}} \| \jump{\u_h} 
		\|_{\L^2(F)}^2 \\
		&\lesssim C_{\eta_0} \hspace{-0.3cm} \sum_{\alpha 
		\in \{+,-\}} \hspace{-0.1cm} \frac{\mu_\alpha}{\gamma} \| \bsnabla \u_h^\alpha 
		\|_{\L^2(\Omega_{\cell_F^\alpha})}^2 + 
		\frac{\gamma\overline{\mu}}{h_{\cell_F}} \| \jump{\u_h} 
		\|_{\L^2(F)}^2, \quad \forall \u_h \in \V_h,
	  \end{aligned}
	\end{equation}
	with $\gamma > 0$ an arbitrary positive constant. Combining~\eqref{eq:b2-wp}
	and~\eqref{eq:b1-wp}, and using the fact that the number of neighbouring
	cells is bounded, we obtain: 
	\[
	  a_h(\u_h,\u_h) \gtrsim \left( C_{\s} C_\Omega - \frac{C_{\eta_0}}{\gamma} 
	  \right) \sum_{\alpha \in \{+,-\}} \mu_\alpha \| \bsnabla \u_h^\alpha 
	  \|^2_{\L^2(\Omega^\alpha)} + \left( 1 - \frac{\gamma}{\beta} \right) 
	  \sum_{F \in \mathcal{F}_h} \frac{\beta \overline{\mu}}{h_{\cell_F}} 
	  \| \jump{\u_h} \|_{\L^2(F)}^2.
	\]
	Let us pick $\gamma = \frac{2C_{\eta_0}}{C_{\s} C_\Omega}$. Assuming $\beta
	\geq 2 \gamma$, the terms in the right-hand side are positive. In order to
	check that $a_h(\u_h,\u_h)$ is also a bound for the $\H^2$ broken semi-norm
	in $\normstab{\cdot}{}$, we proceed as follows. The local discrete inverse
	inequality $\| \bsnabla \boldsymbol{\xi}_h \|_{\L^2(\cell \cap
	\Omega^\alpha)} \leq C h^{-1} \| \boldsymbol{\xi}_h \|_{\L^2(\cell)}$ can
	readily be applied to finite element functions (and its gradients) in
	Ag\ac{fe} spaces (see, e.g.~\cite[(12)]{Badia2018}). On the other hand, by
	Lemma~\ref{lem:bound-l2-cut} we have that $\| \boldsymbol{\xi}_h
	\|_{\L^2(\cell)} \leq C \| \boldsymbol{\xi}_h \|_{\L^2(\cell \cap
	\Omega^\alpha)}$. As a result, we have that $h_\cell \left| \v_h
	\right|_{\H^2(\cell \cap \Omega^\alpha)} \leq C \| \bsnabla \v_h
	\|_{\L^2(\cell \cap \Omega^\alpha)}$, for any $\v_h \in \V_h$. Hence,
	bilinear form $a_h$ satisfies coercivity; it is non-singular.
	
	In order to prove continuity, we need a continuous version
	of~\eqref{eq:b4wp} for functions in $\H^2(\Omega^{\pm}) \cap \V$.
	Using~\eqref{eq:ineqc1}, we get the sought-after bound:
	\begin{equation}\label{eq:b5wp}
      \hspace{-0.2cm} \| \n^+ \cdot \average{\s(\u)} \|_{\L^2(F)}^{2} \lesssim 
	  \overline{\mu}^{2} \hspace{-0.3cm} \sum_{\alpha \in \{+,-\}} \hspace{-0.3cm} 
	  \| \bsnabla \u^\alpha \|_{\L^2(F)}^{2} \lesssim \overline{\mu} 
	  \hspace{-0.3cm} \sum_{\alpha \in \{+,-\}} \left( 
	  \frac{\mu_{\alpha}}{h_{\cell_F^\alpha}} \| \bsnabla \u^\alpha 
	  \|_{\L^2(\cell_F^{\alpha} \cap \Omega^\alpha)}^{2} + \mu_\alpha 
	  h_{\cell_F^\alpha} \left| \u^\alpha \right|_{\H^{2}(\cell_F^{\alpha} 
	  \cap \Omega^\alpha)}^{2} \right).
	\end{equation}
	It follows that continuity is a consequence of~\eqref{eq:cont-strong-form},
	\myadded{\eqref{eq:b4wp} for discrete functions in $\V_h$, \eqref{eq:b5wp} for
	functions in $\H^2(\Omega^{\pm}) \cap \V$}, and the Cauchy-Schwarz inequality.
	Since the problem is finite-dimensional and the corresponding linear system
	matrix is non-singular, there exists a unique solution to
	problem~\eqref{eq:discrete_form}.
\end{proof}

Let us assume that the background mesh $\T_h$ is quasi-uniform, with
characteristic size $h \doteq \max_{\cell \in \T_h} h_\cell$. We adopt now an
extended Scott-Zhang interpolant $\Pi_h^{\sz} : \V \to \V_h$ given by
$\Pi_h^{\sz}(\u) = \left\{ \Pi_{h,+}^{\sz}(\u), \Pi_{h,-}^{\sz}(\u) \right\}$
with $\Pi_{h,\alpha}^{\sz}(\u) \in \V_\ag^\alpha$, $\alpha \in \{+,-\}$, defined
in~\cite{Badia2018a}. The local approximability property in~\cite[Theorem
4.4]{Badia2018a} and the trace inequality~\eqref{eq:ineqc1} applied to $
\psi = \u^\alpha-\Pi_{h,\alpha}^{\sz}(\u)$ yield the following result.
\begin{proposition}\label{prop:local-approx}
  If $\u \in \H^m(\Omega^{\pm})$, $m \geq 2$, and the order of $\V_h$ is greater
  or equal than $m-1$, then
  \begin{equation}
	\normstab{\u-\Pi_h^{\sz}(\u)}{} \lesssim h^{m-1} \seminorm{\u}{\H^m(\Omega^\pm)}{}.
  \end{equation}
\end{proposition}
\noindent{In order to prove \emph{a priori} error estimates, we must assume
additional regularity on the solution. For $\Omega$ being a convex polygon,
$\Gamma$ of class $\mathcal{C}^2$ and $\g_\Gamma \in \H_{00}^{1/2}(\Gamma)$, the
interface problem enjoys smoothing properties and its solution $\u \in
\H^2(\Omega^{\pm})\cap \mathcal{V}$ (see~\cite{chen1998finite}).
Neglecting the geometrical error,
the consistency in Lemma~\ref{lem:consistency}, well-posedness in
Lemma~\ref{lem:wellposed} and the approximability property in
Proposition~\ref{prop:local-approx} can be combined to prove an estimate in the
$\V(h)$ norm. Furthermore, under the previous assumptions, a duality argument
analogous to~\cite[Theorem 6]{hansbo2002unfitted} can be used to obtain the
$L^2$ estimate. The geometrical error in the approximation could be incorporated
into the discussion with the same ideas as, e.g.~in~\cite{chen1998finite}.}
\begin{proposition}\label{prop:a-priori-V}
	If $\u \in \H^m(\Omega^{\pm})$, $m \geq 2$, is the solution of
	\eqref{eq:weak_form1}-\eqref{eq:weak_form2} and $\u_h \in \V_h$ is the
	solution of \eqref{eq:discrete_form}, with the order of $\V_h$ greater or
	equal than $m-1$, then
	\begin{equation}
		\normstab{\u-\u_h}{} \lesssim h^{m-1} 
		\seminorm{\u}{\H^m(\Omega^\pm)}{}, \qquad 
		\norm{\u-\u_h}{\L^2(\Omega)}{} \lesssim h^{m} 
		\seminorm{\u}{\H^m(\Omega^\pm)}{}.
	\end{equation}
\end{proposition}

\section{Numerical experiments}
\label{sec:num-exps}

Our goal, in this section, is to analyse numerically the accuracy, optimality,
robustness and performance of $h$-Ag\ac{fem} for interface elliptic \acp{bvp}.
We consider as model problems the Poisson and linear elasticity equations, with
non-homogeneous Dirichlet conditions on the external boundary and discretised
with the variational formulation detailed in Section~\ref{sec:dis-for}. We
describe first the experimental setup in Section~\ref{sec:exp-setup}, consisting
of several manufactured problems defined in a set of complex geometries. We lay
out next the experimental environment of the $h$-Ag\ac{fem} parallel
implementation in \fempar~\cite{badia-fempar}. After these preliminaries, we
move to report and discuss the numerical results of three different sets of
experiments: convergence tests in Section~\ref{sec:conv-tests}, material
contrast and cut location robustness tests in Section~\ref{sec:robustness} and,
finally, weak scaling tests in Section~\ref{sec:weak-scaling}.

\subsection{Experimental benchmarks}
\label{sec:exp-setup}

Numerical tests consider the variational formulation of
Section~\ref{sec:dis-for}, with non-homogeneous Dirichlet boundary conditions,
applied to the Poisson and linear elasticity problems. Although exposition was
restricted to linear elasticity, the formulation for the Poisson equation can be
easily derived, as a particular case. This leads to an analogous formulation to
the ones in~\cite{hansbo2002unfitted,burman_cutfem_2015,huang2017unfitted}. We
observe that, with little effort, the Poisson equation inherits well-posedness
and approximability results proven in Section~\ref{sec:approx}. Moreover,
harmonic weights become $w^+ = \frac{k^-}{k^+ + k^-}$ and $w^- = \frac{k^+}{k^+
+ k^-}$, where $k^\alpha > 0$, $\alpha \in \{+,-\}$, represents the
subdomain-wise constant diffusion coefficient.

Numerical experiments are carried out on both serial and parallel,
distributed-memory, environments. We generally report parallel results; serial
ones are only shown when informing about condition numbers. We also observe that
parallelisation of interface Ag\ac{fem} basically reuses ideas that have already
been covered in~\cite{Verdugo2019}. In addition, all examples run on background
Cartesian grids, endowed with standard isotropic 1:4 (2D) or 1:8 (3D) refinement
rules; also known as \emph{quadtrees} (2D) or \emph{octrees}
(3D)~\cite{Badia2019b}. We have also addressed in~\cite{inpreparation2020} how
to build Ag\ac{fe} spaces on top of these (generally) nonconforming meshes. In
the experiments, we consider both uniform and $h$-adaptive refinements. The
latter follow an iterative \ac{amr} process~\cite{inpreparation2020} that
exploits the Li and Bettess convergence (or acceptability)
criterion~\cite{li1995theoretical,diez1999unified}. As usual, the goal of the
procedure is to find an optimal mesh, that minimises the number of cells
required to achieve a given discretisation error. Nonetheless, we remark that
remeshing is not driven by \emph{a posterior} error estimation, since we can
compute the exact error in all cases studied, and we do not consider the
geometrical error in approximating the interface. In contrast
to~\cite{inpreparation2020}, we use the \emph{relative} energy norm error in the
acceptability criterion to eliminate the influence of material contrast. Seeking 
to ensure stability, without superfluous aggregation, that degrades accuracy 
and conditioning~\cite{inpreparation2020}, the well-posedness threshold $\eta_0$ 
to isolate badly-cut cells is set to $0.25$.

The \ac{fe} approximation space for all experiments is $\V_h$, described in
Section~\ref{sec:dis-for}, as the single-interface version of the general
$n$-interface $\V_{h,\ag}$ in Section~\ref{sec:ag-fe-space}. Henceforth, we
refer to $\V_h$ simply as the Ag\ac{fe} space. We employ both first and second
order Lagrangian finite elements. Following discussion in
Section~\ref{sec:dis-for}, the coercivity coefficient is given by $\beta = 10.0
\ q^2$, where $q$ is the \ac{fe} order; this value is enough to ensure
well-posedness for all the tests below. Apart from that, robustness tests, in
Section~\ref{sec:robustness}, additionally consider a \emph{standard} \ac{fe}
(or Std\ac{fe}) space defined by $\V_h^\std \doteq \V_{h,\act}^+ \times
\V_{h,\act}^-$. Although $\V_h^\std$ is stable to cut location, under suitable
mesh and interface regularity conditions~\cite{hansbo2002unfitted}, it leads to
much more ill-conditioned systems than the Ag\ac{fe} space~\cite{DePrenter2017}.
For this reason, usage of Std\ac{fe} space is merely intended to provide a
numerical reference to assess the condition number of Ag\ac{fe} space. When
using the Std\ac{fe} space, the $\beta$ coefficient at each (well- or ill-posed)
cut cell is computed by solving a generalised eigenvalue problem, detailed
in~\cite[Section 4.2]{Badia2018}.

The physical domain in all cases is a cuboid (of varying sizes), but the physical
interface dividing the two phases is a non-trivial surface, described as the
0-level set of a (piecewise-)smooth function. We consider eight different
level-set interfaces: (a) a circle, (b) a flower and (c) a "pacman" shape, in 2D;
(d) a cylinder, (e) a popcorn flake, (f) a spiral, (g) a popcorn flake without a
wedge (popcorn pacman) and (h) a gyroid, in 3D. All these geometries are covered
in the literature~\cite{burman_cutfem_2015,lehrenfeld2016high}; they are typically
chosen to examine the behaviour of unfitted \ac{fe} methods. \myadded{We consider
linear approximations of the embedded interfaces; in the numerical results, the
geometrical error does not affect \emph{global} optimality (of quadratic
\acp{fe}). However, in general, high-order geometrical approximations of the
interface are required in order to retain optimality of Ag\ac{fem} with high-order
\acp{fe}.} For illustration purposes, descriptive figures of the considered
interfaces (or the interior region that they enclose) are drawn along the
convergence plots of Section~\ref{sec:conv-tests}. Besides, the geometry for the
gyroid problem is represented in Figure~\ref{fig:gyroid}.

\begin{figure}[!h]
	\centering
	\begin{subfigure}[t]{0.32\textwidth}
		\includegraphics[width=\textwidth]{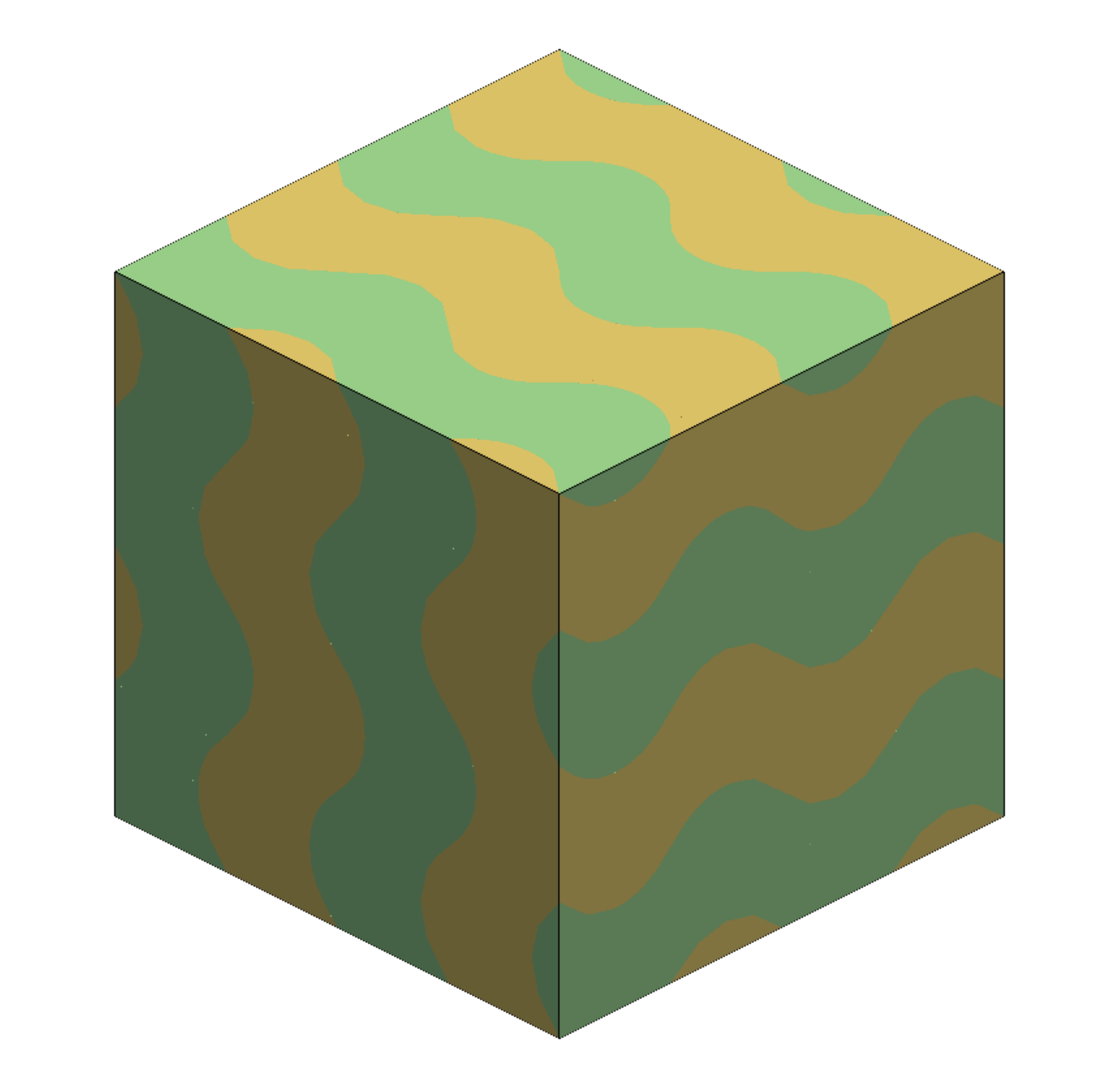}
		\caption{}
	\end{subfigure}
	\begin{subfigure}[t]{0.32\textwidth}
		\includegraphics[width=\textwidth]{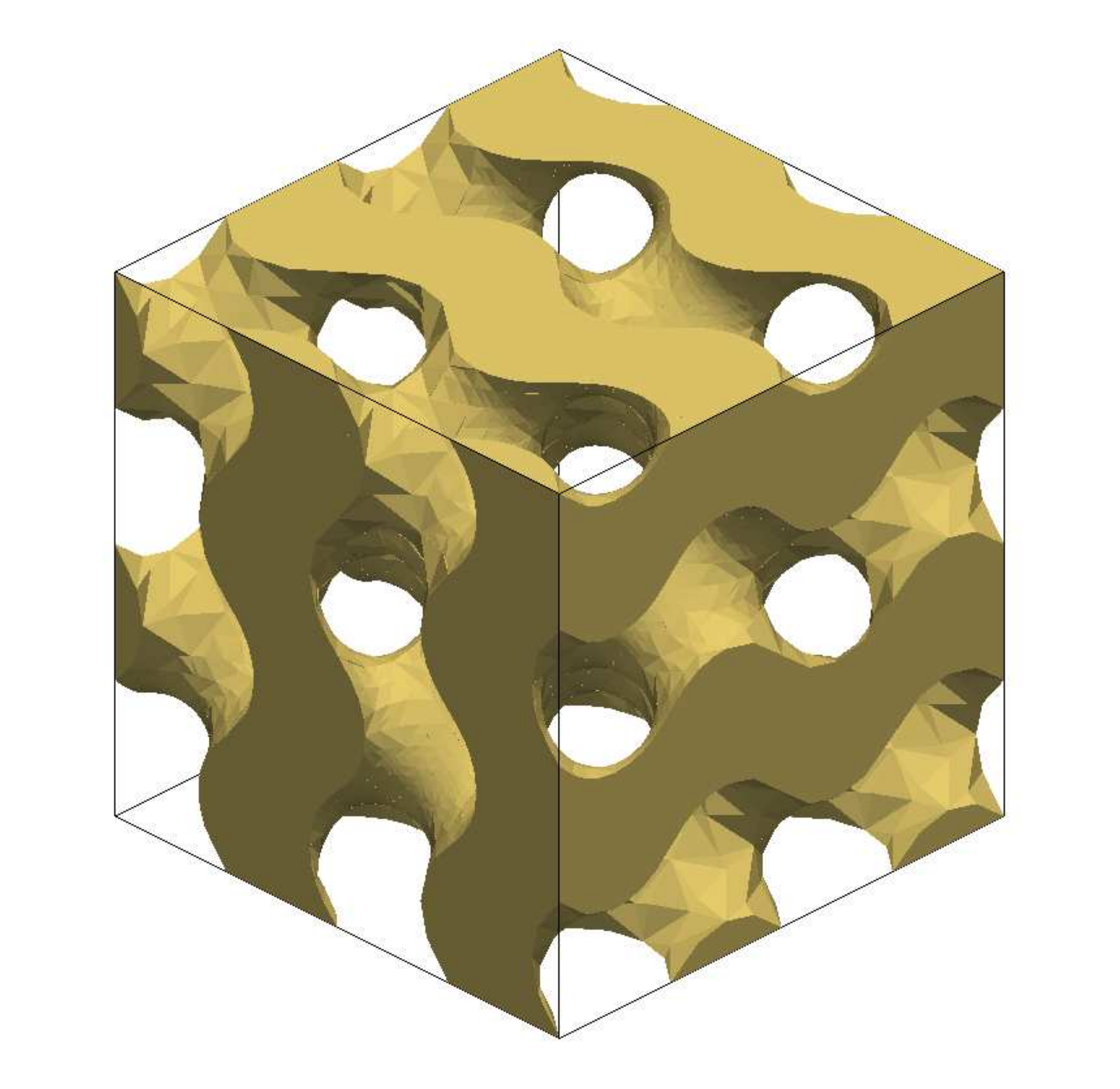}
		\caption{}
	\end{subfigure}
	\begin{subfigure}[t]{0.32\textwidth}
		\includegraphics[width=\textwidth]{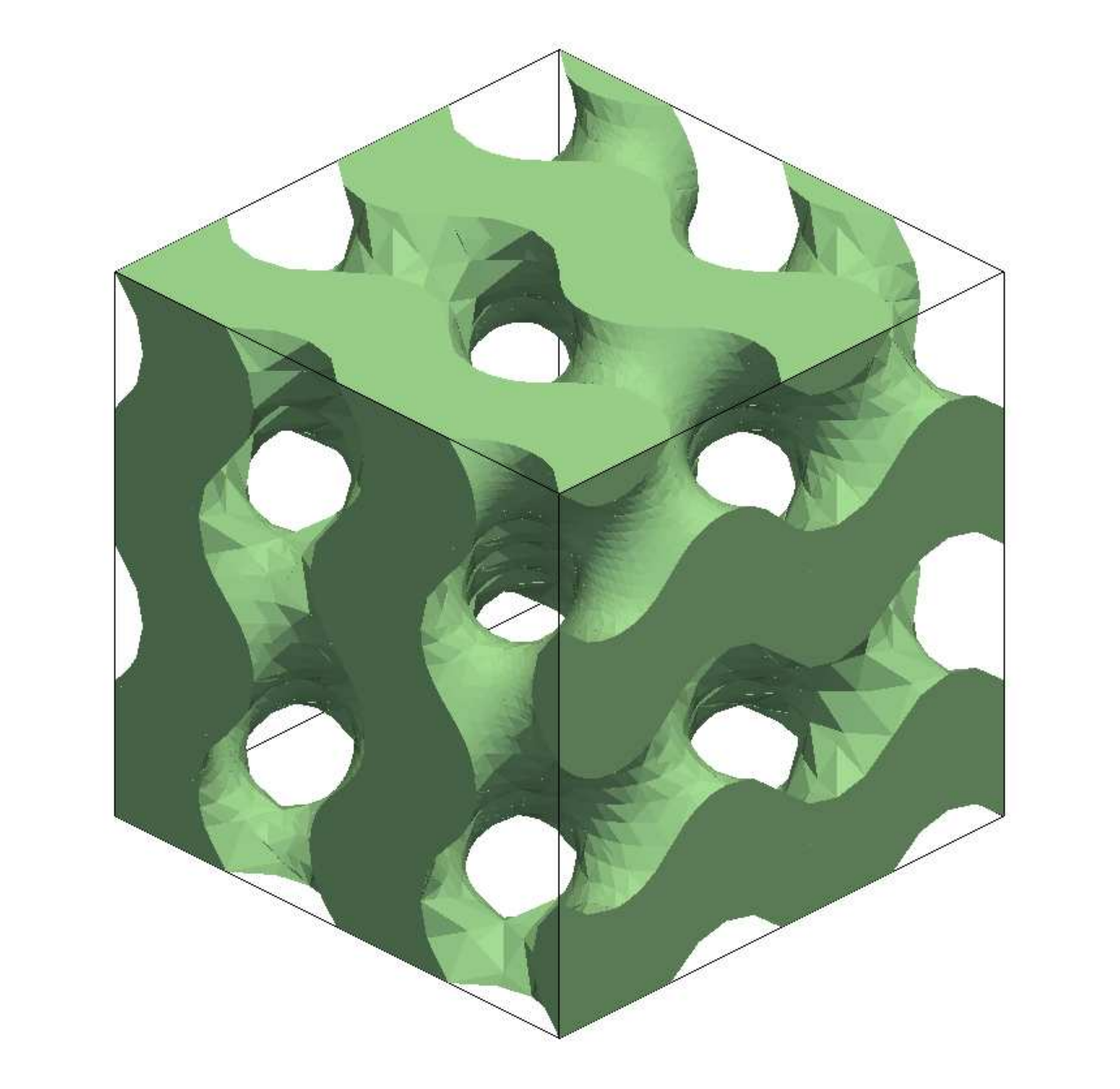}
		\caption{}
	\end{subfigure} \\
	\begin{subfigure}[t]{0.32\textwidth}
		\includegraphics[width=\textwidth]{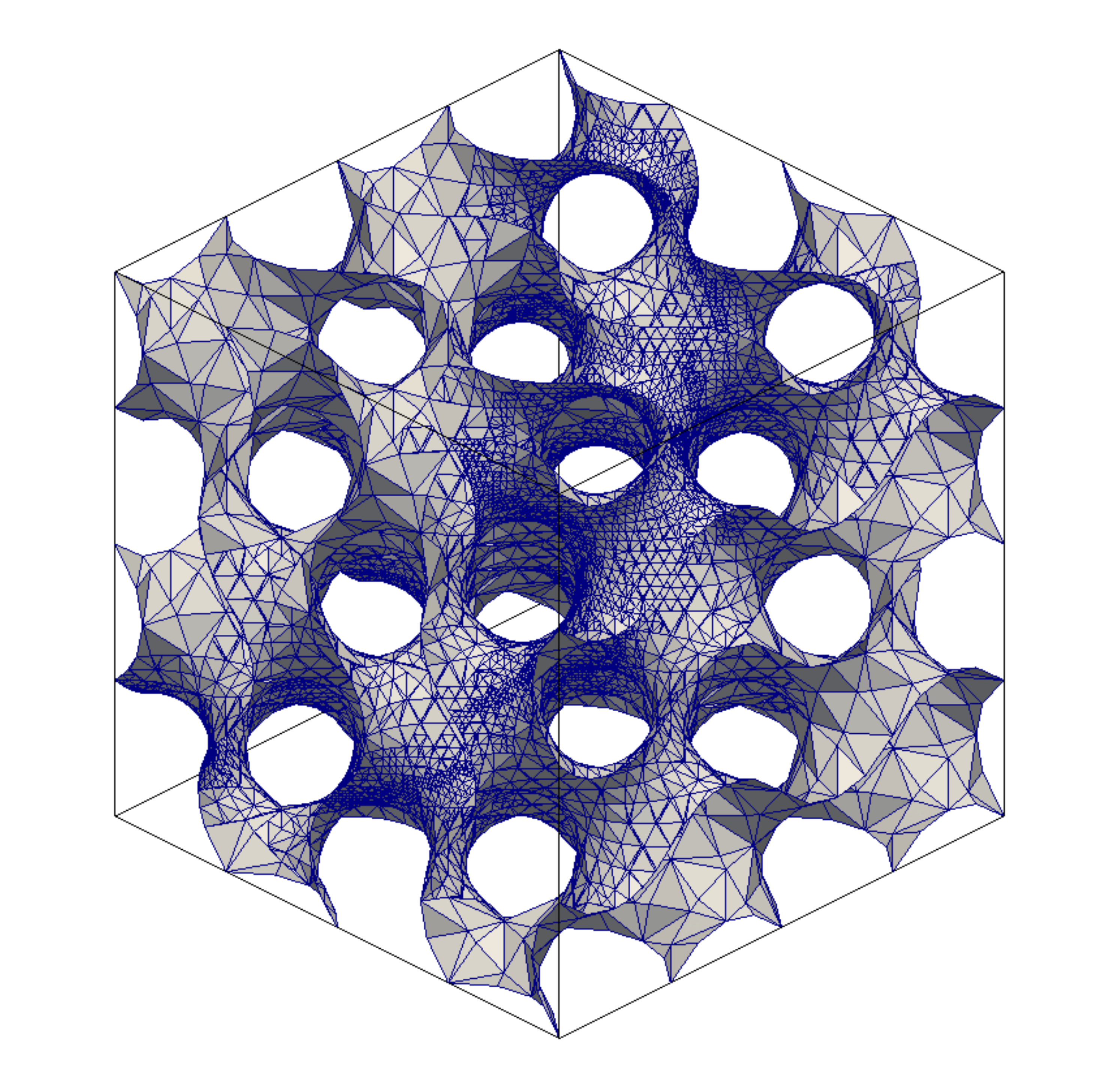}
		\caption{}\label{fig:gyroid-int}
	\end{subfigure}
	\begin{subfigure}[t]{0.32\textwidth}
		\includegraphics[width=\textwidth]{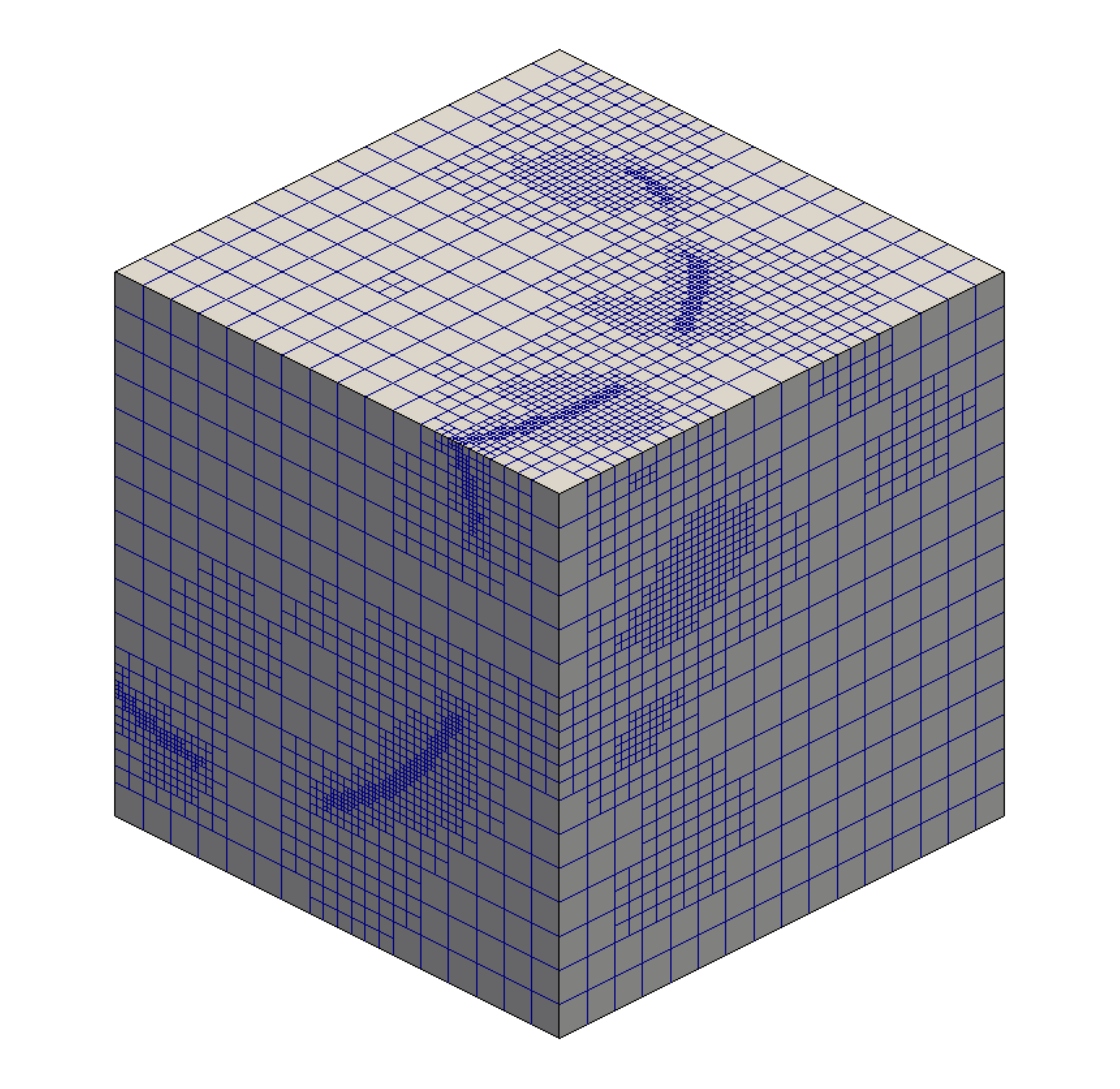}
		\caption{}\label{fig:gyroid-mesh}
	\end{subfigure}
	\begin{subfigure}[t]{0.32\textwidth}
		\includegraphics[width=\textwidth]{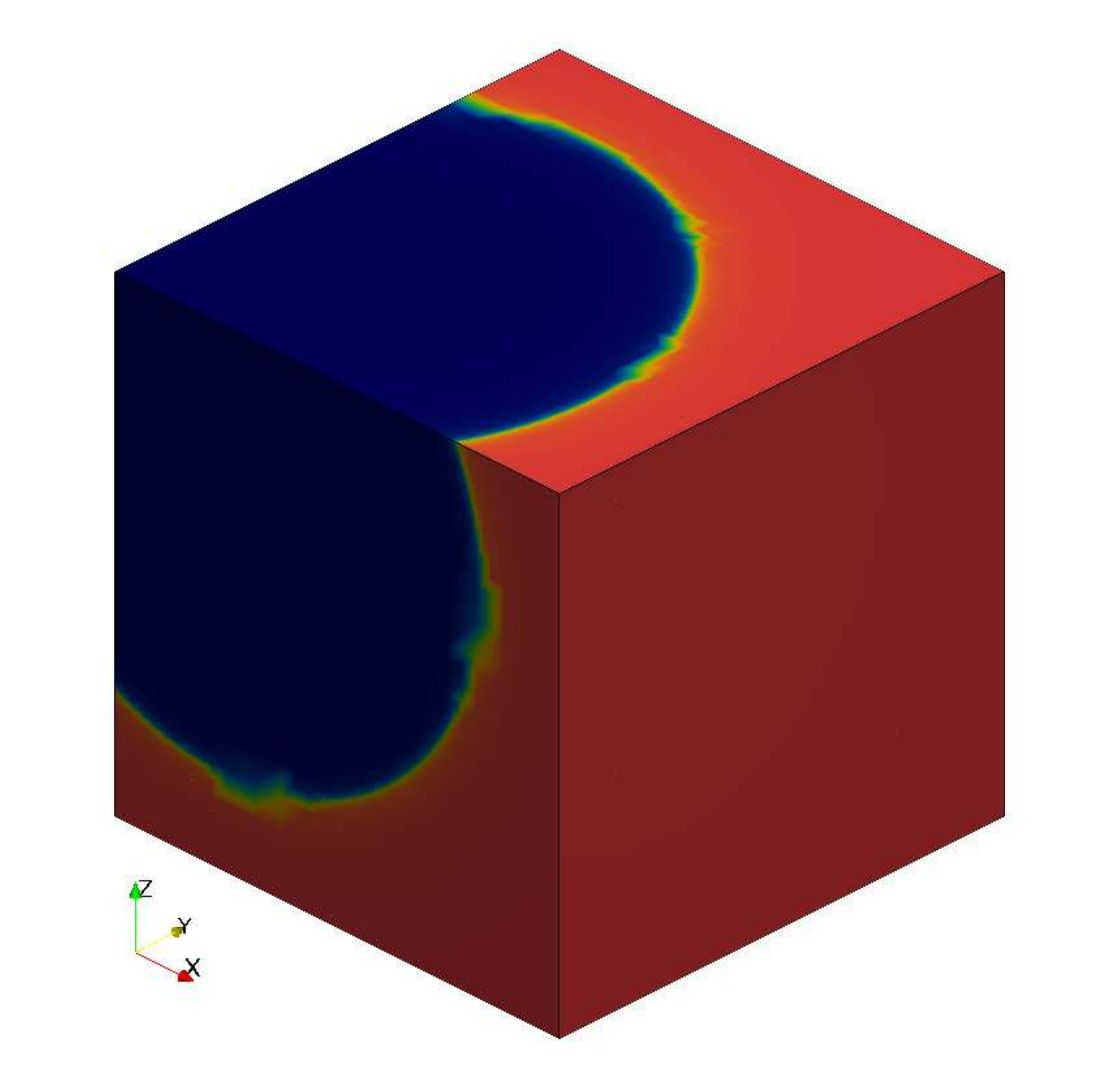}
		\caption{}\label{fig:gyroid-sol}
	\end{subfigure}
	\caption[The gyroid interface and single-shock benchmark]{The gyroid
	interface and single-shock benchmark. The top three figures represent the
	two regions (together and one-by-one) divided by the gyroid level-set
	function on the region $[-2,2]^3$. The bottom three figures represent the
	mesh and solution of the single-shock equation~\eqref{eq:single-shock} with
	$k^+/k^- \neq 1$ on a given $h$-adaptive mesh: the discrete
	approximated interface in Figure~\ref{fig:gyroid-int}, the mesh in
	Figure~\ref{fig:gyroid-mesh} and the solution in
	Figure~\ref{fig:gyroid-sol}. Different mesh resolution is due to dependency
	of the energy norm error on material contrast.}
	\label{fig:gyroid}
\end{figure}

We study four different analytical benchmarks; all of them are derived with the
so-called method of manufactured solutions~\cite{roache2002code}, i.e.~we
propose a solution of the problem with known analytical solution and then we
compute source term and interface conditions from the governing
equations~\eqref{eq:strong_form}. For the Poisson problem we consider a
benchmark for verification (convergence tests), namely the (1)
\emph{out-FE-space} benchmark. We add two more Poisson benchmarks, that
correspond to adapted versions of two classical $hp$-\ac{fem} problems, the (2)
\emph{Fichera-corner} and (3) \emph{single-shock} problems. For linear
elasticity, we address the (4) \emph{cylindrical inclusion} problem
in~\cite{sukumar_modeling_2001}. Let us next provide the analytical expressions
of the solution function for each case.

\begin{itemize}
	\item The \emph{out-FE-space} benchmark is adapted
	from~\cite{annavarapu2012robust} and applied to several interface
	geometries. The solution is given by $u(q, \mathbf{x}): \Omega \subset
	\mathbb{R}^d \to \mathbb{R}$ and $q \in \mathbb{N}$ such that
	\begin{equation}\label{eq:out-fe-space}
		u(q;\mathbf{x}) \doteq \begin{cases}
			\frac{k^+ - k^- + (3 k^- + k^+) x}{4 k^+ 
			(k^- + k^+)} - \frac{x^{q+1}}{(q+1) k^+}, &\text{if } 
			\mathbf{x} \in \Omega^+, \\
			\frac{(3 k^- + k^+) x}{4 k^- (k^- + k^+)} - 
			\frac{x^{q+1}}{(q+1) k^-}, &\text{if } \mathbf{x} \in \Omega^-.
		\end{cases}
	\end{equation}
	In our case, we take $q$ as the \ac{fe} interpolation order, then $u \notin
	\V_h$. Moreover, $u$ is discontinuous across $\Gamma$, but the jump of
	normal fluxes is null, i.e.~$\jump{k \bsnabla u} \cdot \n^+ = 0$.	
	\item The \emph{Fichera-corner} benchmark is adapted
	from~\cite{demkowicz2006computing} and applied to the pacman and
	popcorn-pacman interface shapes. The solution $u(r,\theta,z): \Omega \subset
	\mathbb{R}^d \to \mathbb{R}$ in cylindrical coordinates is
	\begin{equation}\label{eq:fichera}
		u^\alpha(r,\theta,z) \doteq r ^ {\omega^\alpha} \sin \omega^\alpha \theta, 
		\ \alpha \in \{+,-\}, \ \omega^- = 2 / 3, \ \omega^+ = 4.
	\end{equation}
	Numerical solution of~\eqref{eq:fichera} in the popcorn flake without a
	wedge is represented in Figure~\ref{fig:pacman}. We observe that the problem
	has fully non-homogeneous interface conditions. Furthermore, $u^+$ is smooth,
	whereas derivatives of $u^-$ are singular at the $r = 0$ axis; specifically,
	$u^- \in H^{1+\frac{2}{3}}(\Omega^-)$. When \emph{only} approximating $u^-$,
	convergence rates of the energy norm with uniform refinements are limited by
	regularity; they decrease at a rate $\mathcal{O}(h^{2/3})$. Optimal
	convergence rates can be restored with
	$h$-adaptivity~\cite{demkowicz2006computing}. In
	Section~\ref{sec:conv-tests}, we argue that, even though $u$ does not
	explicitly depend on the diffusion coefficients, material contrast
	determines whether convergence behaviour of $u$ (in the energy norm) is
	dictated by regularity of $u^+$ or $u^-$.
\end{itemize}

\begin{figure}[!h]
	\centering
	\begin{subfigure}[t]{0.24\textwidth}
		\includegraphics[width=\textwidth]{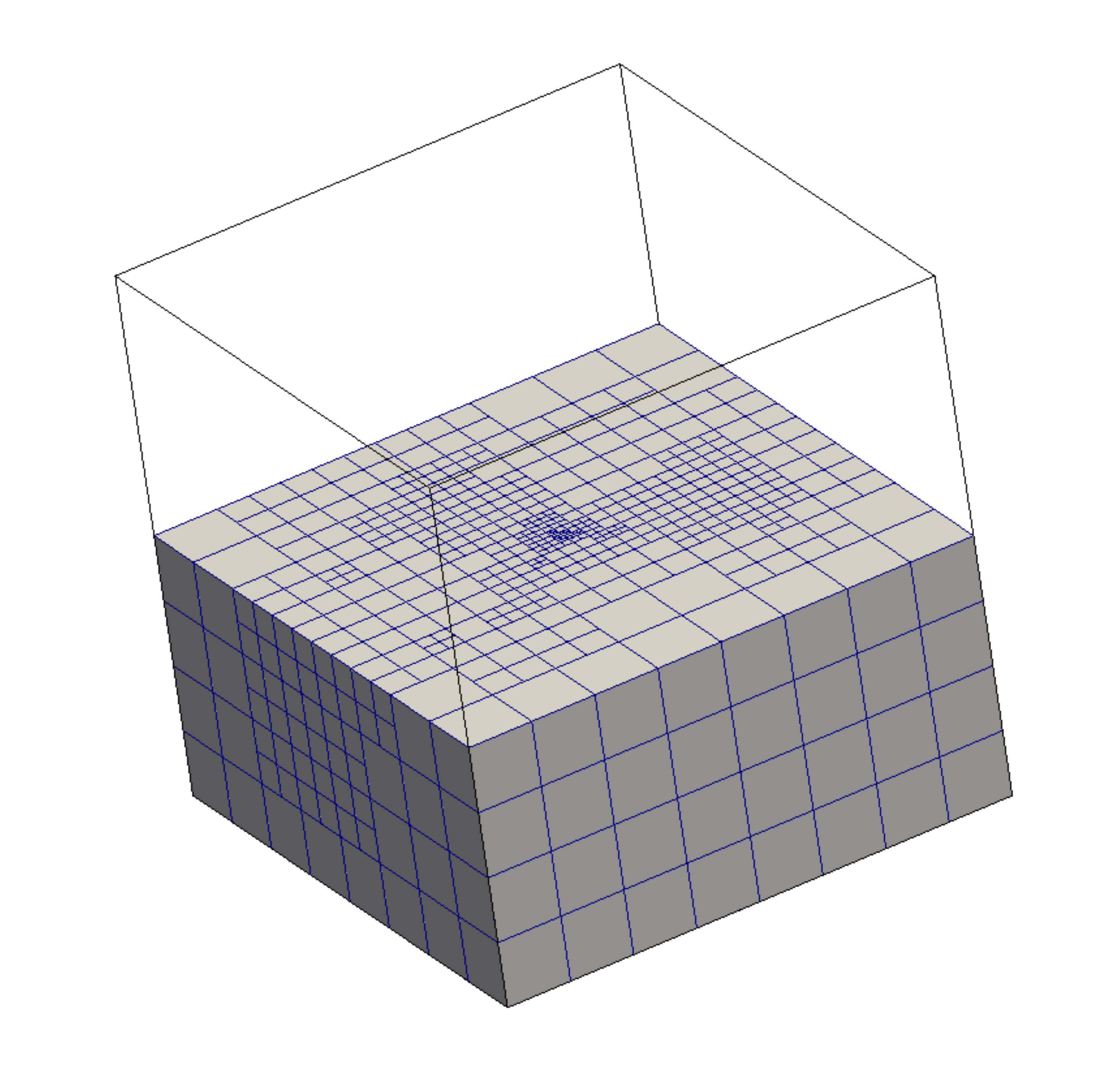}
		\caption{$k^+/k^- = 1$}
	\end{subfigure}
	\begin{subfigure}[t]{0.24\textwidth}
		\includegraphics[width=\textwidth]{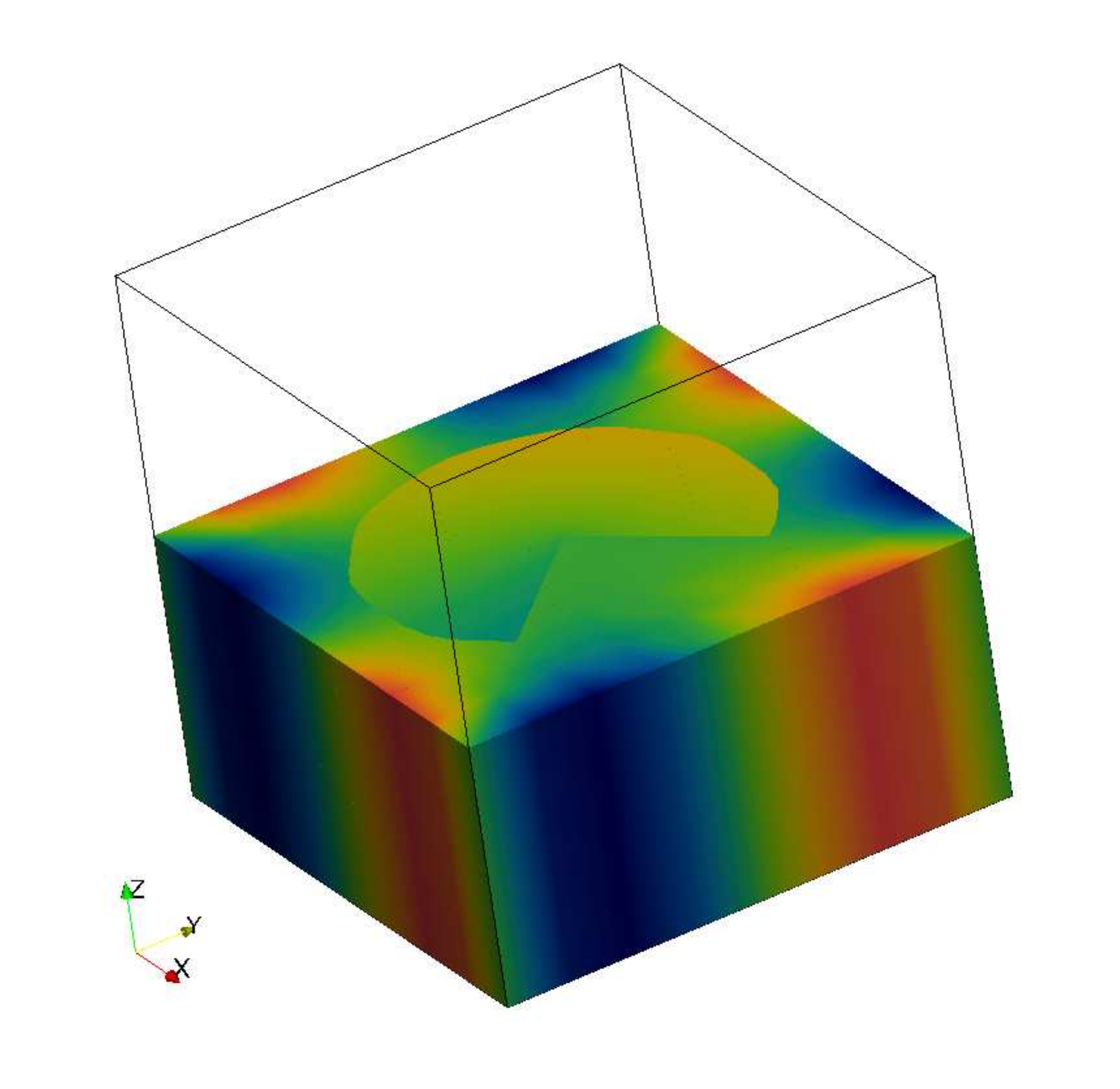}
		\caption{$k^+/k^- = 1$}
	\end{subfigure}
	\begin{subfigure}[t]{0.24\textwidth}
		\includegraphics[width=\textwidth]{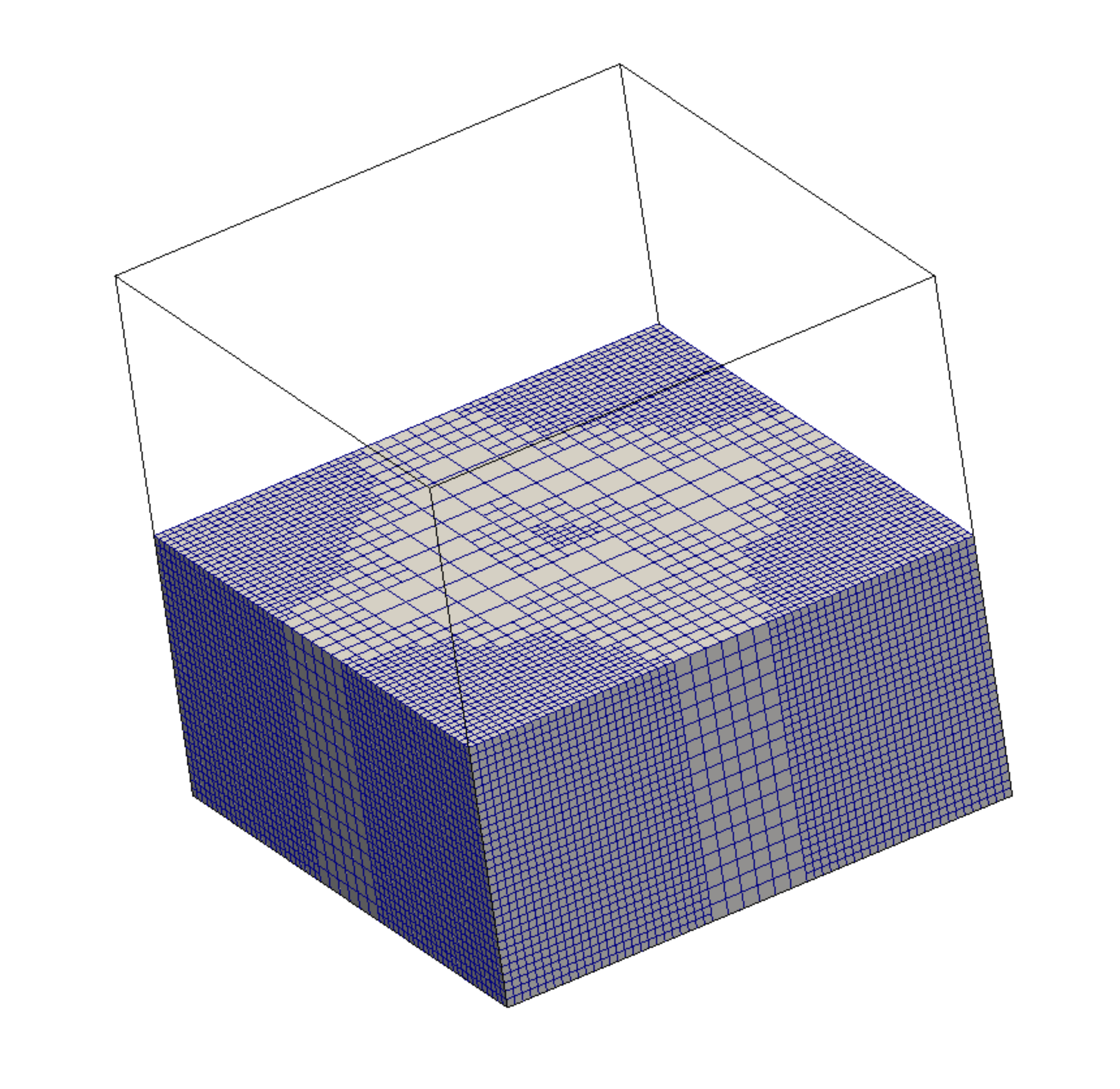}
		\caption{$k^+/k^- = 10^6$}
	\end{subfigure}
	\begin{subfigure}[t]{0.24\textwidth}
		\includegraphics[width=\textwidth]{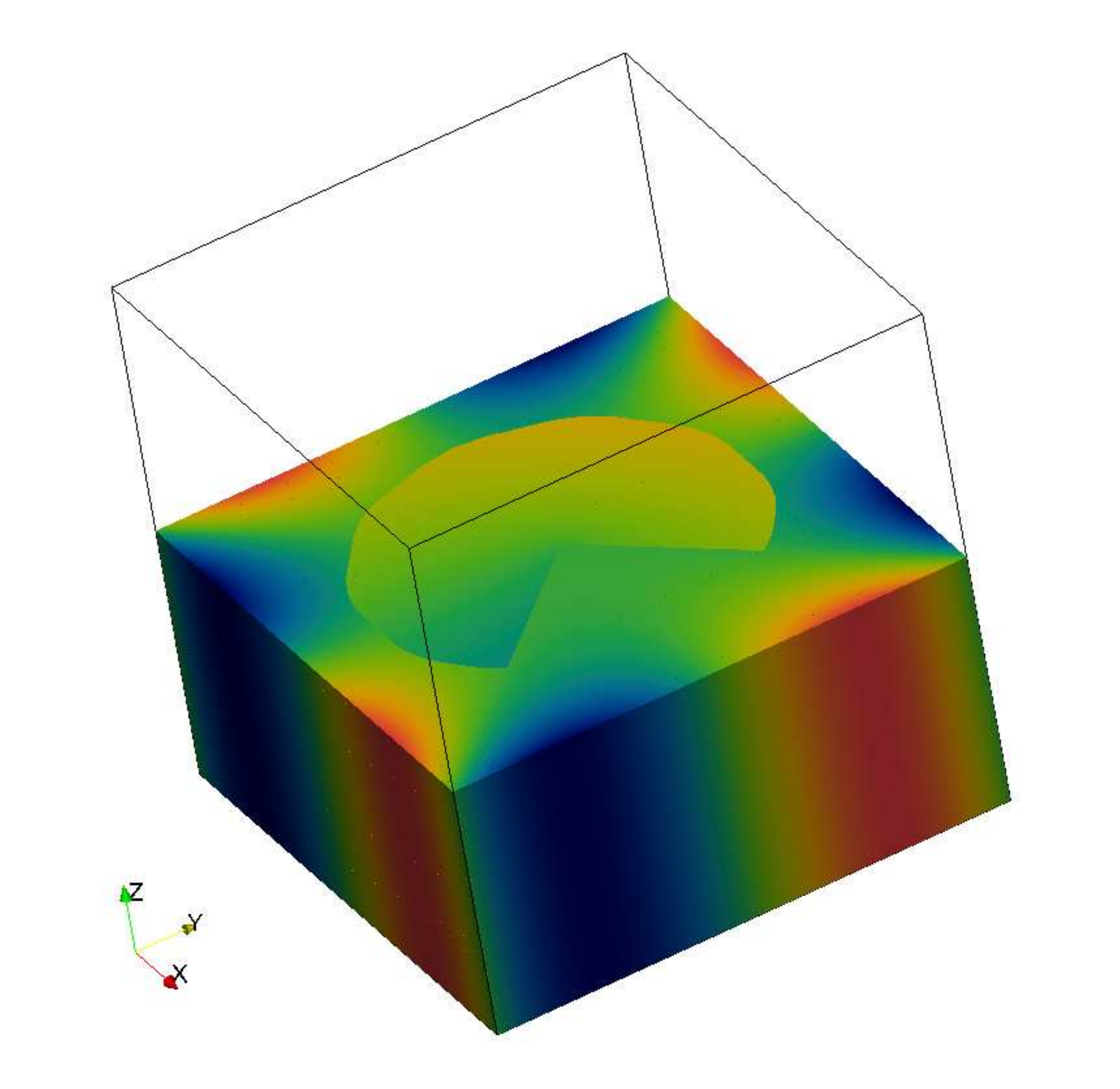}
		\caption{$k^+/k^- = 10^6$}
	\end{subfigure}
	\caption[The Fichera-corner benchmark~\eqref{eq:fichera} on the
	popcorn-pacman interface]{The Fichera-corner benchmark~\eqref{eq:fichera} on
	the popcorn-pacman interface in two different situations. We only show mesh
	and solution at the bottom half of the simulated cube, to show the results
	on the $z = 0$ plane. Material contrast determines which of the solution
	sides dominate the numerical error. In the two left plots,
	$k^+/k^- = 1$ leads to a situation where error and, thus,
	refinements concentrate in $\Omega^-$. Conversely, in the two right plots,
	$k^+/k^- = 10^6$ yields higher errors and mesh refinements in
	$\Omega^+$.}
	\label{fig:pacman}
\end{figure}

\begin{itemize}
	\item The \emph{single-shock} benchmark is also adapted
	from~\cite{demkowicz2006computing} and applied to the gyroid interface. The
	solution $u(r): \Omega \subset \mathbb{R}^d \to \mathbb{R}$ is
	\begin{equation}\label{eq:single-shock}
		u(r) \doteq \arctan( \tau ( r - r_0 ) ), \ \tau = 60, \ r = \| 
		\mathbf{x} - \mathbf{x}_0 \|_2, \ r_0 = 2.5, \ \mathbf{x}_0 = 
		(x_0, y_0, z_0) = (-1,-1,1),
	\end{equation}
	where $\| \cdot \|_2$ denotes the Euclidean norm. Numerical solution
	of~\eqref{eq:single-shock} in the gyroid is represented in
	Figure~\ref{fig:gyroid-sol}. As in the previous
	benchmark~\eqref{eq:fichera}, the analytical solution does not depend on the
	material parameters, but numerical error (in the energy norm) does. Apart
	from that, $u$ is smooth in $\Omega$, although it sharply varies in the
	neighbourhood of the shock, and continuous across $\Gamma$, although with a
	kink if $k^+ \neq k^-$. We notice that the shock may intersect
	$\Gamma$, e.g.~it crosses several times the gyroid 0-level set.
	\item The \emph{cylindrical inclusion} benchmark is applied to a cylindrical
	interface. It adapts the linear elasticity problem in~\cite[Section
	7.3]{sukumar_modeling_2001}. The displacement in cylindrical coordinates is
	given by:
	\begin{equation}\label{eq:sukumar}
		u_r(r) \doteq \begin{cases}
			\left[ \left( 1 - \frac{b^2}{a^2} \right) c + \frac{b^2}{a^2} 
			\right] r, &0 \leq r < a, \vspace{0.1cm} \\
			\left( r - \frac{b^2}{r} \right) c + \frac{b^2}{r}, &a \leq r 
			\leq b.
		\end{cases}, \quad u_\theta \equiv 0, \quad u_z \equiv 0,
	\end{equation}
	where $a = 0.4$, $b = 2.0$ and
	\[
		c = \frac{ ( \lambda_- + \mu_- + \mu_+ ) b^2 }{ ( \lambda_+ +
		\mu_+ ) a^2 + ( \lambda_- + \mu_- ) ( b^2 - a^2 ) + \mu_+ b^2 }.
	\]
	In the experiments, $\Omega \subset \{ 0 \leq r < b \}$ and $\Omega^- = \{ 0
	\leq r < a \}$. The numerical solution is represented in
	Figure~\ref{fig:sukumar}. As in~\eqref{eq:single-shock}, $u$ is continuous
	across $\Gamma$, but it has a kink if material properties are discontinuous.
\end{itemize}

\begin{figure}[!h]
	\centering
	\begin{subfigure}[t]{0.24\textwidth}
		\includegraphics[width=\textwidth]{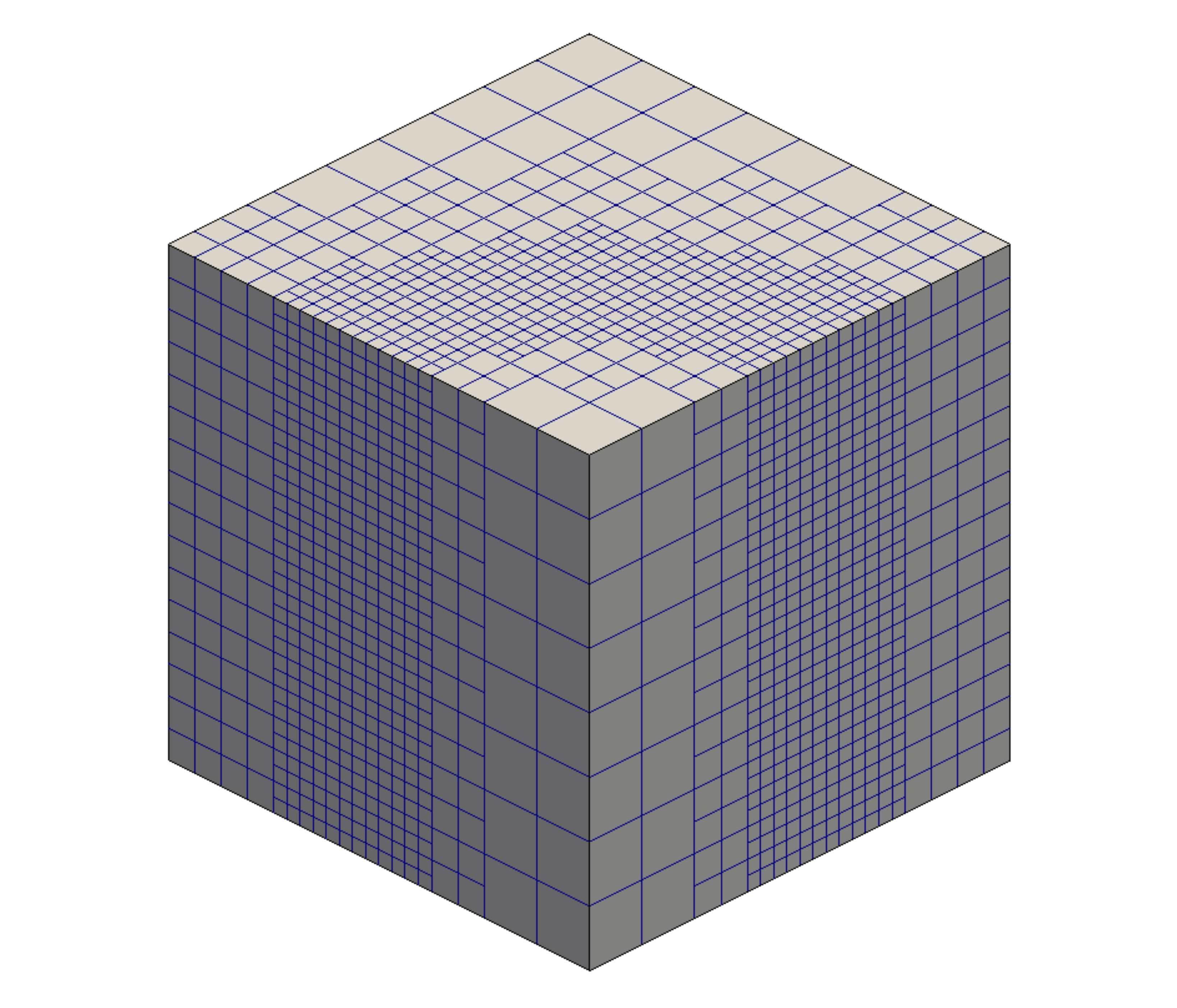}
		\caption{}
	\end{subfigure}
	\begin{subfigure}[t]{0.24\textwidth}
		\includegraphics[width=\textwidth]{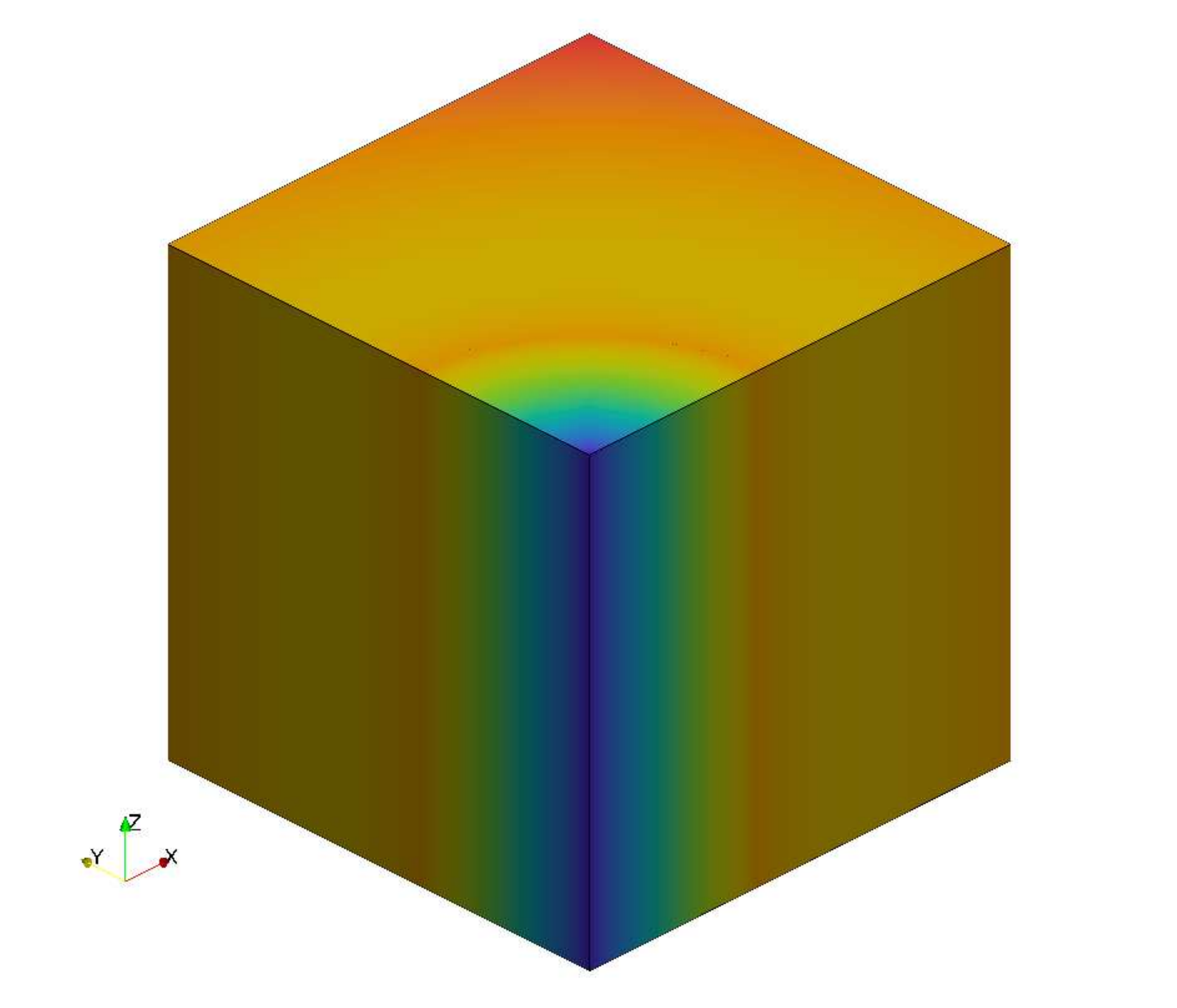}
		\caption{}
	\end{subfigure}
	\caption[Error-driven adaptive mesh and solution of the linear elasticity
	problem~\eqref{eq:sukumar} on the cylinder for $\mu_+/\mu_- \neq
	1$]{Error-driven adaptive mesh and solution of the linear elasticity
	problem~\eqref{eq:sukumar} on the cylinder for $\mu_+/\mu_- \neq 1$. The
	solution has a kink along the interface and error concentrates at the side
	of the interface outside the cylinder.}
	\label{fig:sukumar}
\end{figure}

Table~\ref{tab:params} summarizes the main parameters and computational strategies
used in the numerical examples. The variety of complex shapes and benchmarks
considered above are intended to exhibit the good behaviour of interface
Ag\ac{fem}, in as many situations as possible. Our numerical tests consider first
numerical verification of the theoretical results proved in
Section~\ref{sec:dis-for}. In Section~\ref{sec:conv-tests}, we carry out
convergence tests in uniform and $h$-adaptive meshes to show that interface
Ag\ac{fem} recovers optimal convergence rates. Afterwards, we examine, in
Section~\ref{sec:robustness}, robustness w.r.t.~cut location and material
contrast, by means of geometry and material perturbations. We show that the
condition number of the linear system, after diagonal scaling, is independent of
cut location and material contrast. Finally, in Section~\ref{sec:weak-scaling}, we
assess good parallel performance and scalability with a weak-scaling analysis of
some selected cases from the convergence tests. For each type of numerical test,
we perform a subset of the possible matrix of cases in Table~\ref{tab:params}. We
provide details for each subset, when dealing with the corresponding test. But
before all that, we inform next about the computational infrastructure and
software employed.

\begin{table}[ht!]
	\centering
	\begin{small}
		\begin{tabular}{ll}
			\toprule
			Description & Considered methods/values \\
			\midrule
			Model problem & interface Poisson, interface linear elasticity
			\vspace{0.12cm} \\
			Problem geometry & 2D: circle, flower, pacman shape; 3D: cylinder,
			\\ & popcorn flake, spiral, popcorn pacman, gyroid \vspace{0.12cm}
			\\ Benchmark & out-FE-space~\eqref{eq:out-fe-space}, Fichera
			corner~\eqref{eq:fichera}, \\ & single 
			shock~\eqref{eq:single-shock}, cylindrical 
			inclusion~\eqref{eq:sukumar} \vspace{0.12cm} \\
			Experimental computer environment & serial and parallel
			\vspace{0.12cm} \\
			Parallel mesh generation and partitioning tool & \p4est{}
			library~\cite{burstedde_p4est_2011} \vspace{0.12cm} \\
			Mesh topology & single quadtree (2D) or octree (3D) \vspace{0.12cm}
			\\
			Remeshing strategy & uniform, $h$-adaptive with Li and
			Bettess~\cite{li1995theoretical} criterion \vspace{0.12cm} \\
			Well-posed cut cell criterion & $\eta_0 = 0.25$ \vspace{0.12cm} \\
			\ac{fe} spaces & Ag\ac{fe} and Std\ac{fe} \vspace{0.12cm} \\
			Cell type and \ac{fe} interpolation & Q1 and Q2 hexahedral cells \vspace{0.12cm} \\
			Linear solver & sparse direct (serial) \\
			& preconditioned conjugate gradients (parallel) \vspace{0.12cm} \\
			Parallel preconditioner & smoothed-aggregation
			\gamg{}~\cite{GAMGweb} \vspace{0.12cm} \\
			\gamg{} stopping criterion & $\| \mathbf{r} \|_2/ \| \mathbf{b} \|_2
			< 10^{-9}$ \vspace{0.12cm} \\
			Weights in averaged normal fluxes & $w^+ = \frac{k^-}{k^+ + k^-}$
			and $w^- = \frac{k^+}{k^+ + k^-}$ (Poisson) \vspace{0.02cm} \\
			& $w^+ = \frac{\mu^-}{\mu_+ + \mu_-}$ and $w^- = \frac{\mu^+}{\mu_+ 
			+ \mu_-}$ (elasticity) \vspace{0.12cm} \\
			Coef. in Nitsche's penalty term for Ag\ac{fem} & $\beta = 10.0 \ 
			q^2$, $q$ is the \ac{fe} interpolation order \\
			\bottomrule
		\end{tabular}
	\end{small}
	\caption{Summary of main parameters and computational strategies used in
	the numerical examples}
	\label{tab:params}
\end{table}

\subsection{Experimental environment}
\label{sec:exp-env}

Serial experiments are launched at the TITANI cluster of the Universitat
Polit\`{e}cnica de Catalunya (Barcelona, Spain), whereas parallel experiments
are carried out at the Marenostrum-IV (MN-IV) supercomputer, hosted by the
Barcelona Supercomputing Centre. A \ac{mpi} parallel implementation of the
interface $h$-Ag\ac{fem} method is available at \fempar{}~\cite{badia-fempar}.
\fempar{} is linked against \p4est{} v2.2~\cite{burstedde_p4est_2011}, as the
octree Cartesian grid manipulation engine, and \petsc{}
v3.11.1~\cite{petsc-user-ref} distributed-memory linear algebra data structures
and solvers. Besides, condition number estimates are computed outside
\fempar{} with \matlab{} function
\textit{condest}.\footnote{MATLAB is a trademark of THE MATHWORKS INC.}  



Concerning linear solvers, a sparse direct solver from the MKL PARDISO
package~\cite*{_intel_????} is employed for serial tests. In contrast, a
preconditioned \ac{cgm} method is adopted for parallel tests. The selected
preconditioner is a smoothed-aggregation \ac{amg} scheme called
\gamg{}~\cite{GAMGweb}. The linear solver is set up as in~\cite{Verdugo2019},
with the aim of reducing, as much as possible, deviation from the default
configuration given by \gamg{}. In order to advance the convergence test down to
low global energy-norm error values, without being polluted by linear solver
accuracy, convergence of \gamg{} is declared when $\| \mathbf{r} \|_2/ \|
\mathbf{b} \|_2 < 10^{-9}$ within the first $500$ iterations, where
$\mathbf{r}\doteq \mathbf{b}-\mathbf{A}\mathbf{x}^{\rm cg}$ is the
unpreconditioned residual. Both (serial and parallel) solvers and preconditioner
are readily available through the Krylov Methods \texttt{KSP} module of
\petsc{}.

\subsection{Convergence tests}
\label{sec:conv-tests}

We study the convergence of interface Ag\ac{fem} in two stages. In the first
one, we choose benchmark~\eqref{eq:out-fe-space} and examine the rate at which
the relative energy norm error decays with uniform mesh refinements. In a second
stage, we consider the remaining benchmarks and observe the behaviour for both
uniform and error-driven $h$-adaptive mesh refinements. All experiments run on
five MN-IV nodes, i.e.~we use a total of 240 CPUs, with each CPU mapped to a
different \ac{mpi} task.

For the first part, we consider the circle, flower, popcorn and spiral interface
geometries. In the first three cases, the level sets are centred at the origin
of coordinates and the physical domain is the unit $[0,1]^3$ cube, while the
physical domain of the spiral case is the $[-1,1]^2 \times [0,2]$ cuboid. In all
cases, the interface cuts the external boundary. Besides, the circle has radius
$0.7$ and the flower level-set function in polar coordinates is
$\varphi(r,\theta) = r - 0.7 ( 1 + 0.3 \sin(5\theta) )$. We refer
to~\cite{burman_cutfem_2015,Badia2018a} for the remaining level-set function
expressions. The cuboid is initially meshed with a uniform Cartesian grid.
Figure~\ref{fig:verif_unif} gathers all convergence tests on uniform meshes for
problem~\eqref{eq:out-fe-space}. In agreement to
Proposition~\ref{prop:a-priori-V}, we observe that Ag\ac{fem} consistently
recovers optimal convergence rates in the $H^1$-seminorm (equivalent to the
energy norm) for all cases considered, including first and second order
interpolations and extreme material contrasts.

\begin{figure}[!h]
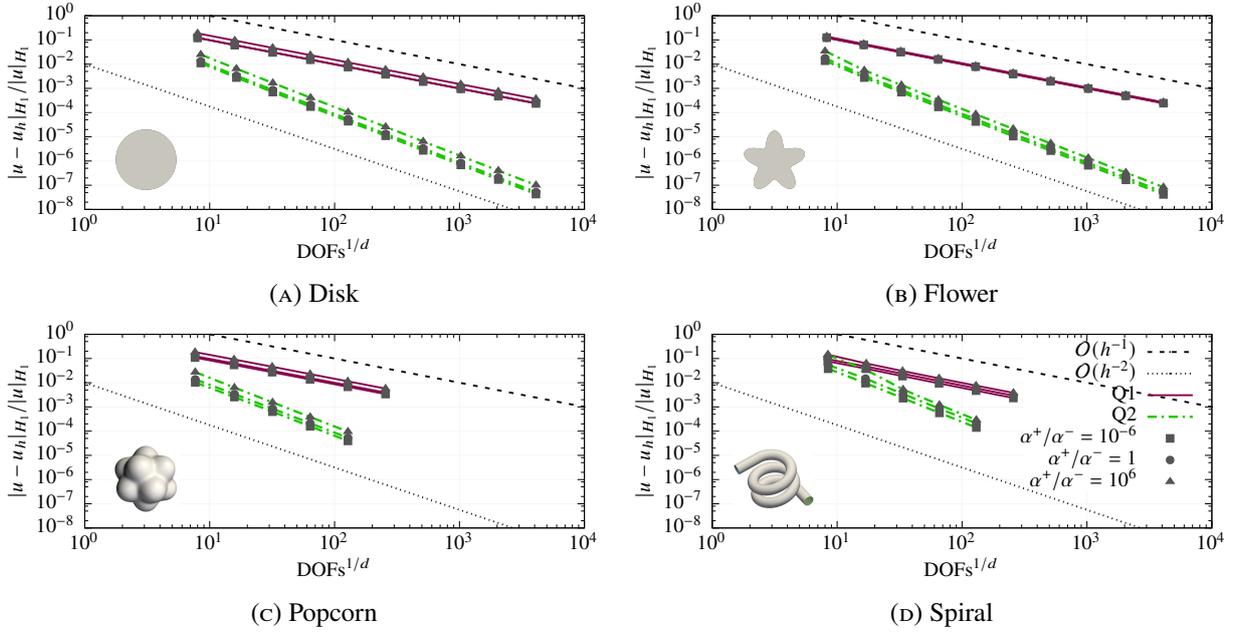

	\centering
	\begin{subfigure}[t]{0.48\textwidth}
		\scalebox{0.7}{\input{figures/c0d1e0f0g0h0i1.tex}}
		\caption{Disk}
	\end{subfigure}
	\begin{subfigure}[t]{0.48\textwidth}
		\scalebox{0.7}{\input{figures/c1d1e0f0g0h0i1.tex}}
		\caption{Flower}
	\end{subfigure} \\
	\begin{subfigure}[t]{0.48\textwidth}
		\scalebox{0.7}{\input{figures/c2d1e0f0g0h0i1.tex}}
		\caption{Popcorn}
	\end{subfigure}
	\begin{subfigure}[t]{0.48\textwidth}
		\scalebox{0.7}{\input{figures/c3d1e0f0g0h0i1.tex}}
		\caption{Spiral}
	\end{subfigure}
	\caption[Convergence tests on uniform meshes]{Convergence tests on uniform
	meshes: For benchmark~\eqref{eq:out-fe-space} and an initial uniform mesh,
	Ag\ac{fem} consistently shows optimal convergence rates as the mesh is
	uniformly refined.}
	\label{fig:verif_unif}
\end{figure}

For tests with uniform and $h$-adaptive mesh refinements, we consider (a) the
Fichera-corner~\eqref{eq:fichera} on the pacman (2D) and popcorn-pacman (3D)
shapes, (b) the single-shock~\eqref{eq:single-shock} on the gyroid and (c) the
cylindrical inclusion~\eqref{eq:sukumar} on a cylinder. The physical domains are
$[0,1]^d$, $[-2,2]^3$ and $[0,1]^3$, resp. Geometry and numerical solutions for
each case are represented in Figures~\ref{fig:pacman},~\ref{fig:gyroid-sol}
and~\ref{fig:sukumar}. We note that, in (a) the interface is in the interior of
$\Omega$, while in (c) we exploit radial symmetry. \myadded{We recall that the
\ac{amr} process is driven by computing the exact discretisation error and the Li
and Bettess convergence criterion~\cite{li1995theoretical}, see
Section~\ref{sec:exp-setup}.} As shown in Figure~\ref{fig:verif_amr}, optimal
convergence rates are retained both with uniform and $h$-adaptive mesh
refinements, regardless of extreme material contrast values and order of
approximation. Let us further justify this result:

\begin{figure}[!h]
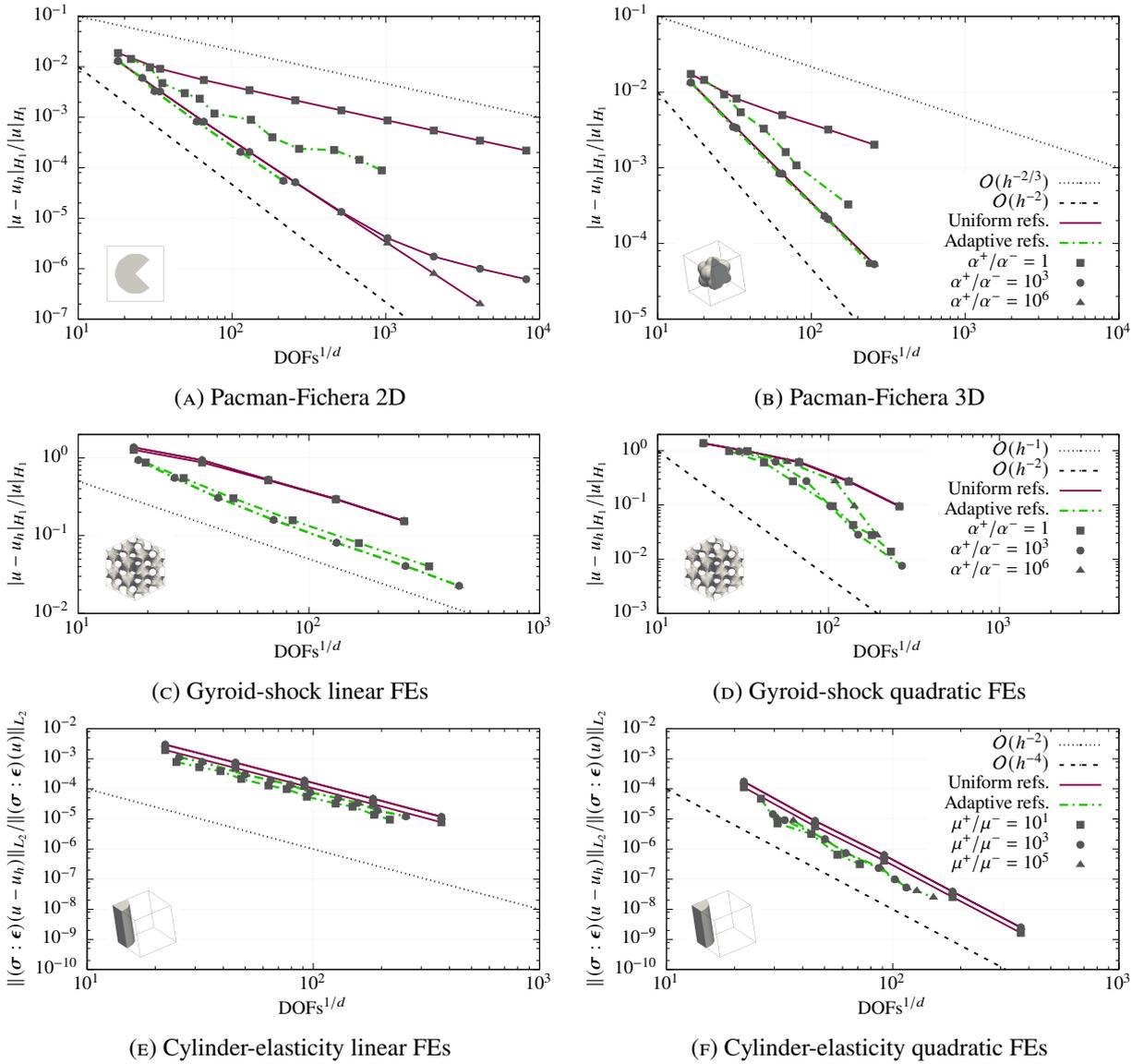

	\centering
	\begin{subfigure}[t]{0.48\textwidth}
		\scalebox{0.7}{\input{figures/b0c0d1e1f0h0i0-errors.tex}}
		\caption{Pacman-Fichera 2D}\label{fig:verif_amr_a}
	\end{subfigure}
	\begin{subfigure}[t]{0.48\textwidth}
		\scalebox{0.7}{\input{figures/b1c1d1e1f0h0i0-errors.tex}}
		\caption{Pacman-Fichera 3D}\label{fig:verif_amr_b}
	\end{subfigure} \\
	\begin{subfigure}[t]{0.48\textwidth}
		\scalebox{0.7}{\input{figures/b1c2d1e0f0h0i1-errors.tex}}
		\caption{Gyroid-shock linear \acp{fe}}\label{fig:verif_amr_c}
	\end{subfigure}
	\begin{subfigure}[t]{0.48\textwidth}
		\scalebox{0.7}{\input{figures/b1c2d1e1f0h0i1-errors.tex}}
		\caption{Gyroid-shock quadratic \acp{fe}}\label{fig:verif_amr_d}
	\end{subfigure} \\
	\begin{subfigure}[t]{0.48\textwidth}
		\scalebox{0.7}{\input{figures/b0c0d0e0f0h0i0-errors.tex}}
		\caption{Cylinder-elasticity linear \acp{fe}}\label{fig:verif_amr_e}
	\end{subfigure}
	\begin{subfigure}[t]{0.48\textwidth}
		\scalebox{0.7}{\input{figures/b0c0d0e1f0h0i0-errors.tex}}
		\caption{Cylinder-elasticity quadratic \acp{fe}}\label{fig:verif_amr_f}
	\end{subfigure}
	\caption[Convergence tests on $h$-adaptive meshes]{Convergence tests on
	$h$-adaptive meshes: (a)-(b) h-adaptivity test with the Fichera-corner
	problem~\eqref{eq:fichera} for quadratic \acp{fe}: Ag\ac{fem} reproduces the
	behaviour of standard \ac{fem} in body-fitted meshes, i.e. convergence rates
	with uniform refinements is limited by solution regularity, whereas optimal
	convergence rates are restored with AMR. (c)-(d) h-adaptivity test with the
	single-shock problem~\eqref{eq:single-shock} on the gyroid: $h$-Ag\ac{fem}
	holds (asymptotically) optimal convergence rates. (e)-(f) h-adaptivity test
	with the cylindrical inclusion problem~\eqref{eq:sukumar} on the cylinder:
	energy norm error using $h$-Ag\ac{fem} decays at optimal
	\emph{superconvergent} rates.}
	\label{fig:verif_amr}
\end{figure}

Even though the solution to the Fichera-corner does not depend on material
parameters, convergence rates do. In the Fichera case with uniform refinements,
global error decreases at a rate of 2:3, when discrete error concentrates in
$\Omega^-$, since $u^- \in H^{1+\frac{2}{3}}(\Omega^-)$ has limited regularity.
Conversely, standard convergence rates hold, when discrete error concentrates in
$\Omega^+$, where $u^+$ is smooth. Material contrast regulates which side of
$\Gamma$ initially contributes more to numerical error, although when $h \to 0$
global error always converges at the slowest rate. We see that, for
$k^+/k^- = 1$, global error clearly concentrates in $\Omega^-$, while
it concentrates in $\Omega^+$ for $k^+/k^- = 10^6$. For an
intermediate value, e.g.~$k^+/k^- = 10^3$, discrete error initially
concentrates in $\Omega^+$, but for $h$ small enough it shifts to $\Omega^-$.

As expected, $h$-adaptive refinements eliminate the influence of regularity of
$u^-$ on the convergence rates. However, as shown in Figure~\ref{fig:pacman},
different values of the material contrast produce different refinement patterns,
in consistence with the discrete error distribution, as discussed above. In
particular, mesh refinements concentrate in $\Omega^-$ (or $\Omega^+$), when
$k^+/k^-$ is small (or large).

Since the single-shock case in the gyroid is rather intricate, convergence rates
are initially suboptimal; optimal convergence rates are reached asymptotically
(especially for quadratic \acp{fe}). We observe that, in front of uniform
refinements, \ac{amr} is capable of entering faster into the asymptotic regime.
However, the pace at which this is achieved depends on material contrast. In
particular, larger values of $k^+/k^-$ slow down reaching optimal rates.

Apart from that, results for the linear elasticity problem also deserve attention.
We identify that the energy norm of the error decreases at a rate of 1:2 for
linear \acp{fe} and 1:4 for quadratic \acp{fe}. This means we obtain
superconvergence ($\mathcal{O}(h^{q+1})$) for linear \acp{fe} and ultraconvergence
($\mathcal{O}(h^{q+2})$) for quadratic \acp{fe}. Although we do not have
conclusive evidence, we believe this behaviour is explained by the fact that
Gauss-Legendre quadrature points on hexahedral cells are superconvergent stress
recovery points~\cite{zienkiewicz1992superconvergent}. In our case, when the cell
is not intersected by $\Gamma$, local errors $\| (\s : \eps) ( \u -\u_h )
\|_{\L^2(\Omega)}$ are integrated with standard Gauss-Legendre quadrature rules.
As a result, even though the approximated solution is not superconvergent, local
error is computed at points that are superconvergent. In contrast, quadrature
rules are locally modified in cut cells, as usual in unfitted \ac{fe}
methods~\cite{burman_cutfem_2015}; thus, local errors in those cells are not
computed at superconvergent points. In spite of this, it is clear from the
convergence plots that the behaviour of the global error in the energy norm is not
influenced by cut cells, i.e.~global error retains the local superconvergence
property that (only) holds in non-cut cells.

\subsection{Robustness with respect to cut location and material contrast}
\label{sec:robustness}

For the sensitivity of Ag\ac{fem} to cut location and material contrast, we
restrict ourselves to the Poisson benchmark~\eqref{eq:out-fe-space} in the
flower and popcorn interfaces and the linear elasticity
benchmark~\eqref{eq:sukumar} in the cylinder.

Our approach is similar to the one in~\cite{annavarapu2012robust}. It consists
in carrying out a batch of simulations in a biparametric space, considering
different material contrast and cut configurations, as shown in
Figure~\ref{fig:popcorn}. The procedure is as follows. We start with a reference
simulation in a unit cube $[0,1]^d$, that takes $k^+/k^- = 1$
for~\eqref{eq:out-fe-space}, or $\mu_+/\mu_- = 1$ for~\eqref{eq:sukumar}. The
unit cube is uniformly meshed with cell size $h = 2^{-6+q}$
for~\eqref{eq:out-fe-space}, and $h = 2^{-5+q}$ for~\eqref{eq:sukumar}, where
$q$ is the \ac{fe} interpolation order. The material perturbation simply
consists in varying the material contrast $k^+/k^-$ or $\mu_+/\mu_-$
of the reference simulation in the interval $[10^{-6},10^6]$. On the other hand,
to produce different cut configurations, the unit cube is scaled to
$[0,1+ah]^d$, where $a \in [-1,1]$. We remark that the number of mesh cells is
kept constant, i.e. the cell size after scaling is $\hat{h} = (1+ah)h$.

\begin{figure}[!h]
	\centering
	\begin{subfigure}[t]{0.32\textwidth}
		\includegraphics[width=\textwidth]{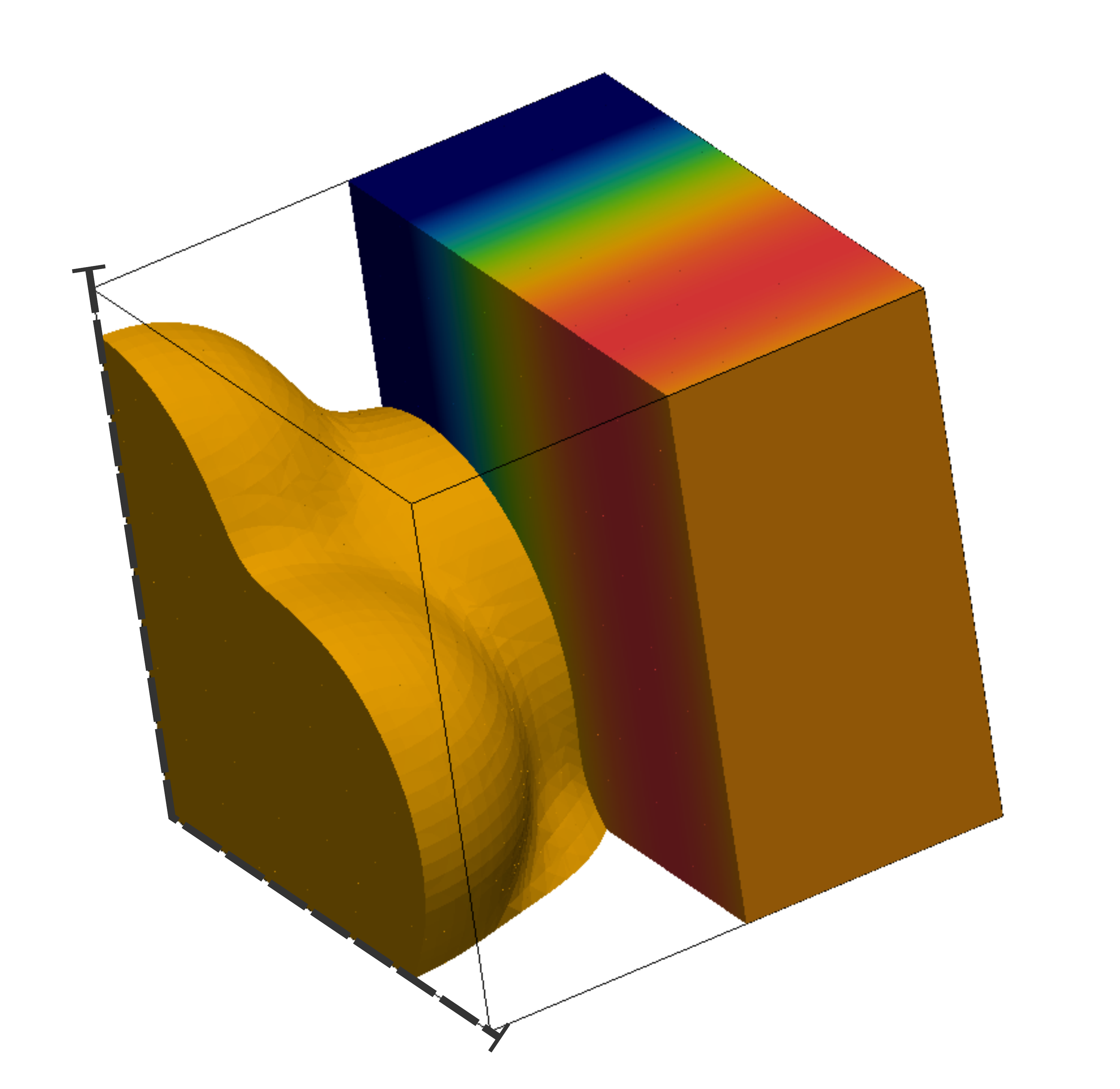}
		\caption{$k^+/k^- = 10^{-6}$ and $a = -1$}\label{fig:popcorn-a}
	\end{subfigure}
	\begin{subfigure}[t]{0.32\textwidth}
		\includegraphics[width=\textwidth]{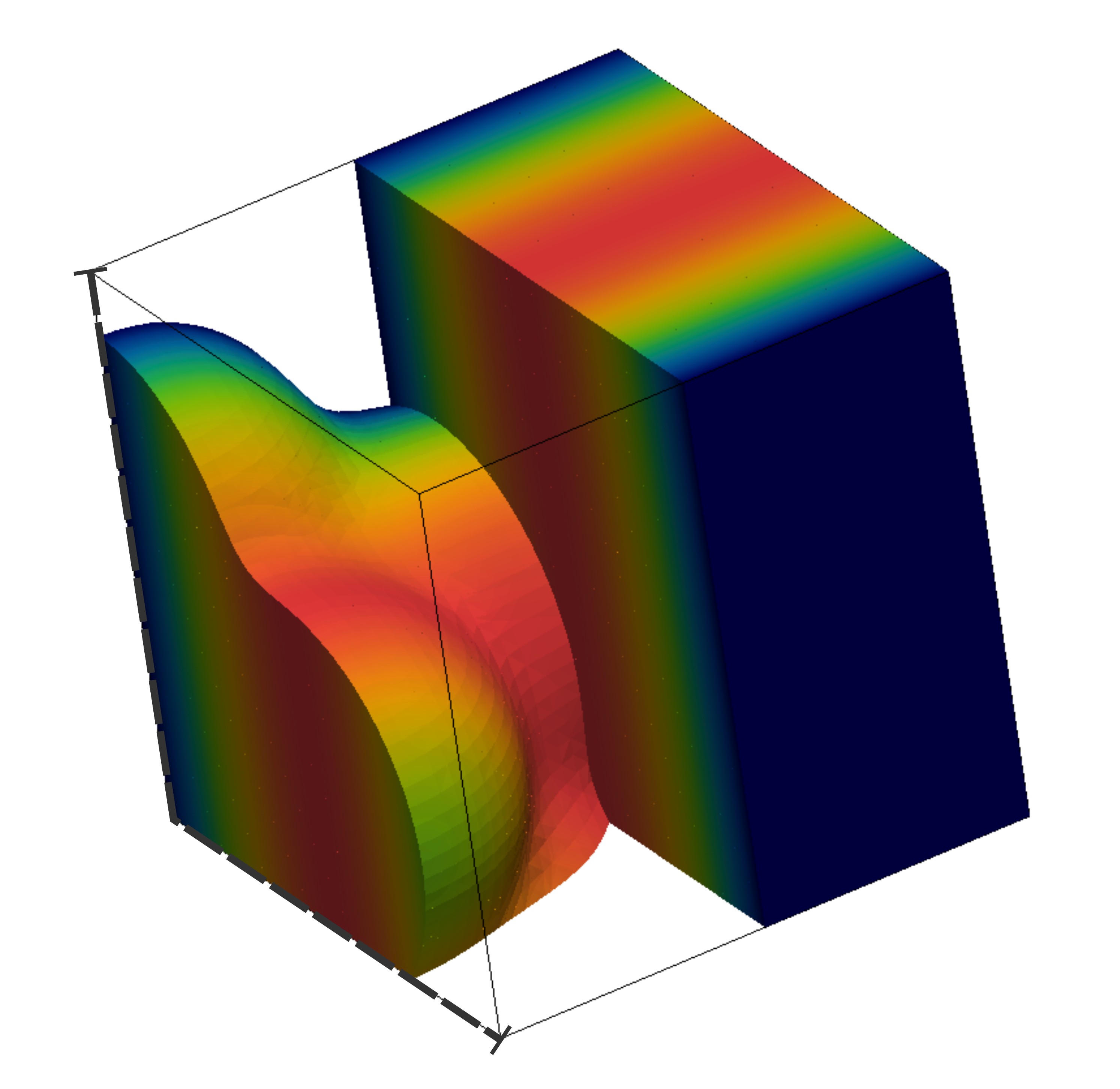}
		\caption{$k^+/k^- = 1$ and $a = 0$}\label{fig:popcorn-b}
	\end{subfigure}
	\begin{subfigure}[t]{0.32\textwidth}
		\includegraphics[width=\textwidth]{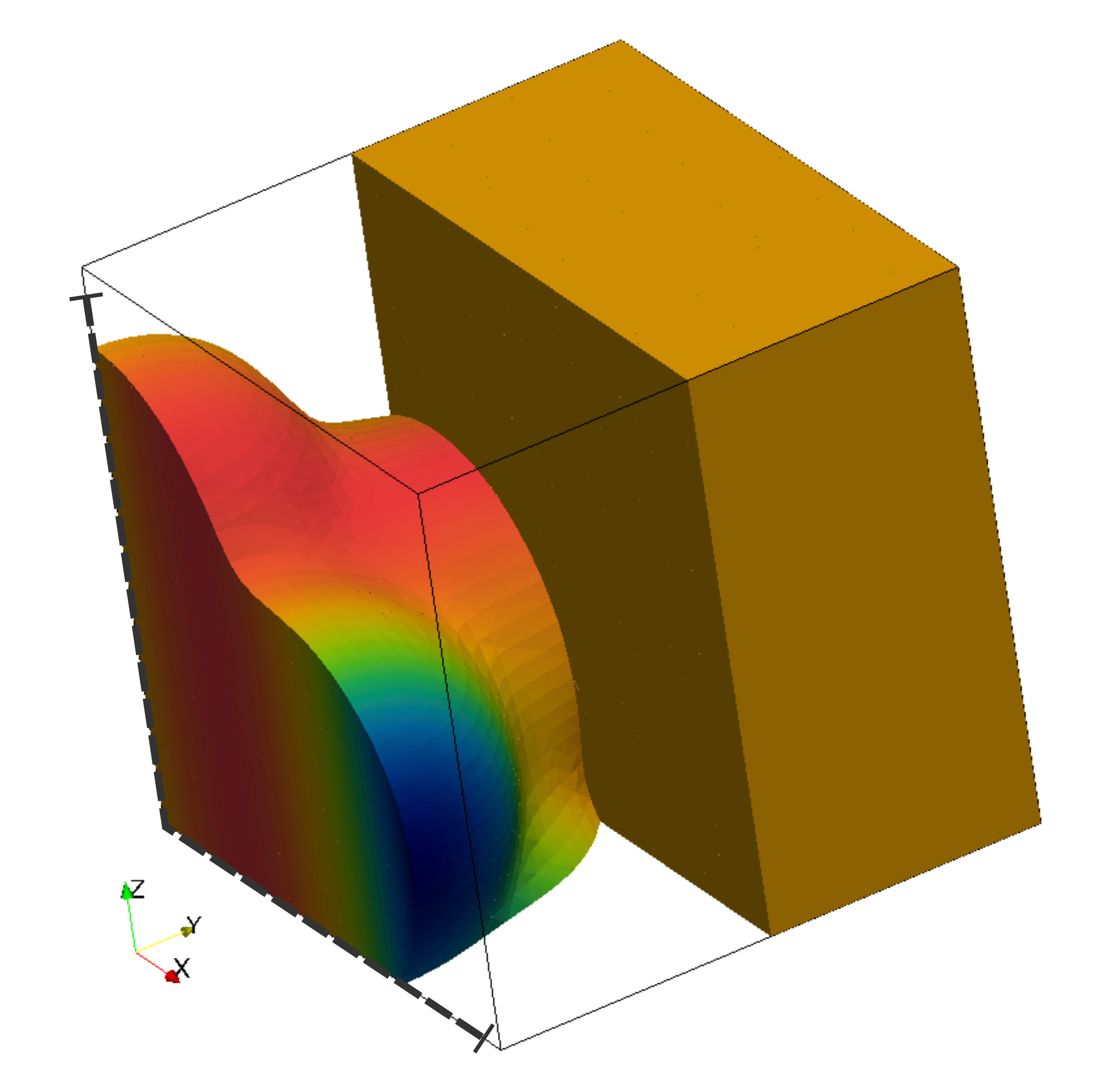}
		\caption{$k^+/k^- = 10^6$ and $a = 1$}\label{fig:popcorn-c}
	\end{subfigure}
	\caption[Illustration of the approach to study robustness w.r.t.~cut location
	and material contrast on the popcorn interface]{Illustration of the approach
	to study robustness w.r.t.~cut location and material contrast on the popcorn
	interface. Note that we only show the right half of the subdomain outside
	the popcorn flake. To study sensitivity to material contrast, we vary
	$k^+/k^-$ between $10^{-6}$ and $10^6$. To study sensitivity to
	cut location, we produce different cut locations by uniformly shrinking
	(Figure~\ref{fig:popcorn-a}) or stretching (Figure~\ref{fig:popcorn-c}) the
	physical domain with a parameter $a \in [-1,1]$ (dashed lines show the $x$
	and $z$ dimensions of the cube represented in Figure~\ref{fig:popcorn-b}, as
	reference to compare the different cube scalings).}
	\label{fig:popcorn}
\end{figure}

Given this setting, we launch simulations with Ag\ac{fem} for different pairs of
$(k^+/k^-,a)$ or $(\mu_+/\mu_-,a)$, until we sweep the range
$[10^{-6},10^6] \times [-1,1]$. We consider both serial and parallel
computations; the latter are carried out in a single MN-IV node, i.e.~48 tasks.
Along the sweep, we gather $H^1$-seminorm errors and condition number estimates.
Afterwards, we condense the results into colour maps that plot the values each
of these quantities in the $(k^+/k^-,a)$ or $(\mu_+/\mu_-,a)$ planes.
We discuss next some of the results obtained with this procedure, represented in
Figures~\ref{fig:sensitivity_h1},~\ref{fig:sensitivity_cn}
and~\ref{fig:sensitivity_le}.

\begin{figure}[!h]
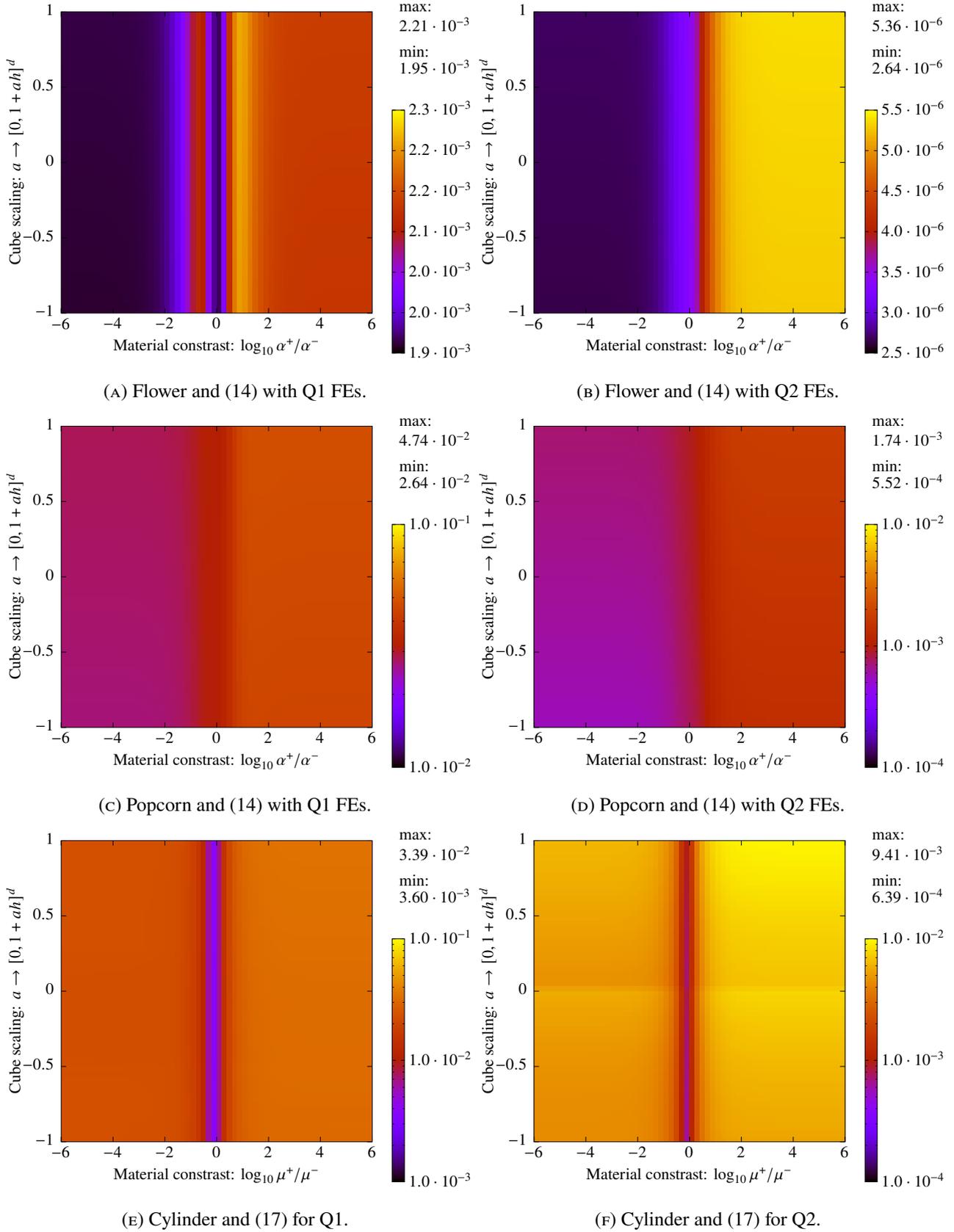

	\centering
	\begin{subfigure}[t]{0.47\textwidth}
		\scalebox{0.48}{\input{figures/run_a0b0c1d0e0f0g1_n1_h1_semi.tex}}
		\caption{Flower and~\eqref{eq:out-fe-space} with Q1 \acp{fe}.}
	\end{subfigure} \quad
	\begin{subfigure}[t]{0.47\textwidth}
		\scalebox{0.48}{\input{figures/run_a0b0c1d1e0f0g1_n1_h1_semi.tex}}
		\caption{Flower and~\eqref{eq:out-fe-space} with Q2 \acp{fe}.}
	\end{subfigure} \\ \vspace{-0.5cm}
	\begin{subfigure}[t]{0.47\textwidth}
		\scalebox{0.48}{\input{figures/run_a1b1c1d0e0f0g1_n1_h1_semi.tex}}
		\caption{Popcorn and~\eqref{eq:out-fe-space} with Q1 \acp{fe}.}
	\end{subfigure} \quad
	\begin{subfigure}[t]{0.47\textwidth}
		\scalebox{0.48}{\input{figures/run_a1b1c1d1e0f0g1_n1_h1_semi.tex}}
		\caption{Popcorn and~\eqref{eq:out-fe-space} with Q2 \acp{fe}.}
	\end{subfigure} \\ \vspace{-0.5cm}
	\begin{subfigure}[t]{0.47\textwidth}
		\scalebox{0.48}{\input{figures/run_a0b0c0d0e0f0g0_n1_h1_semi.tex}}
		\caption{Cylinder and~\eqref{eq:sukumar} for Q1.}
	\end{subfigure} \quad
	\begin{subfigure}[t]{0.47\textwidth}
		\scalebox{0.48}{\input{figures/run_a0b0c0d1e0f0g0_n1_h1_semi.tex}}
		\caption{Cylinder and~\eqref{eq:sukumar} for Q2.}
	\end{subfigure}
	\caption[Sensitivity test of Ag~\ac{fem} w.r.t. material contrast and cut
	location. $H^1$-seminorm of relative error]{Sensitivity test of Ag~\ac{fem}
	w.r.t. material contrast and cut location: For the cases described in
	Section~\ref{sec:robustness}, the $H^1$-seminorm relative error, i.e.
	$|u-u_h|_{H^1}/|u|_{H^1}$, is barely sensitive to material contrast and cut
	location.}
	\label{fig:sensitivity_h1}
\end{figure}

As seen in Figure~\ref{fig:sensitivity_h1}, numerical errors in the
$H^1$-seminorm are barely sensitive to material contrast and cut location. This
behaviour is consistently observed in all three cases and linear/quadratic
\acp{fe}. Although, for the linear elasticity case~\eqref{eq:sukumar}, the error
decreases one order of magnitude around $\mu_+/\mu_- = 1$, this is attributed to
the fact that the solution is more regular when $\mu_+/\mu_- = 1$ (it does not
have a kink), not to the material contrast.

In Figure~\ref{fig:sensitivity_cn}, we plot condition numbers obtained with one of
the three cases, namely the Poisson equation~\eqref{eq:out-fe-space} on the
popcorn interface. We have additionally swept the parametric space with
Std\ac{fe}, for comparison with Ag\ac{fe}; it clearly illustrates the effect of
the latter on the conditioning of the matrix. As shown in
Figures~\ref{fig:sensitivity_cn_a} and~\ref{fig:sensitivity_cn_b}, the condition
number of the linear system is extremely high for Std\ac{fe}. \myadded{While these
large estimates are likely affected by a large numerical error, they clearly
demonstrate the high sensitivity of Std\ac{fem} to the cut configuration.}
Besides, the problem can be so ill-conditioned that the local eigenvalue solver to
compute $\beta$ breaks down. In contrast, Ag\ac{fem} is fully robust and brings
down condition numbers to values that the solvers can cope with, see
Figures~\ref{fig:sensitivity_cn_c} and~\ref{fig:sensitivity_cn_d}. Besides,
dependence on cut location vanishes completely, although there is a clear
sensitivity to material contrast. Nonetheless, this dependence is not present in
the condition number of the \emph{diagonally-scaled} system matrix. Indeed, as
seen in Figures~\ref{fig:sensitivity_le_a} and~\ref{fig:sensitivity_le_b}, the
condition number after diagonal scaling becomes barely sensitive to both cut
location and material contrast. Furthermore, condition numbers are around
$\mathcal{O}(10^4)$, in the worst case, which is a rather low value for unfitted
3D+Q2 simulations. The same outcome is observed for the linear elasticity case, as
shown in Figures~\ref{fig:sensitivity_le_c} and~\ref{fig:sensitivity_le_d}. 

\begin{figure}[!h]
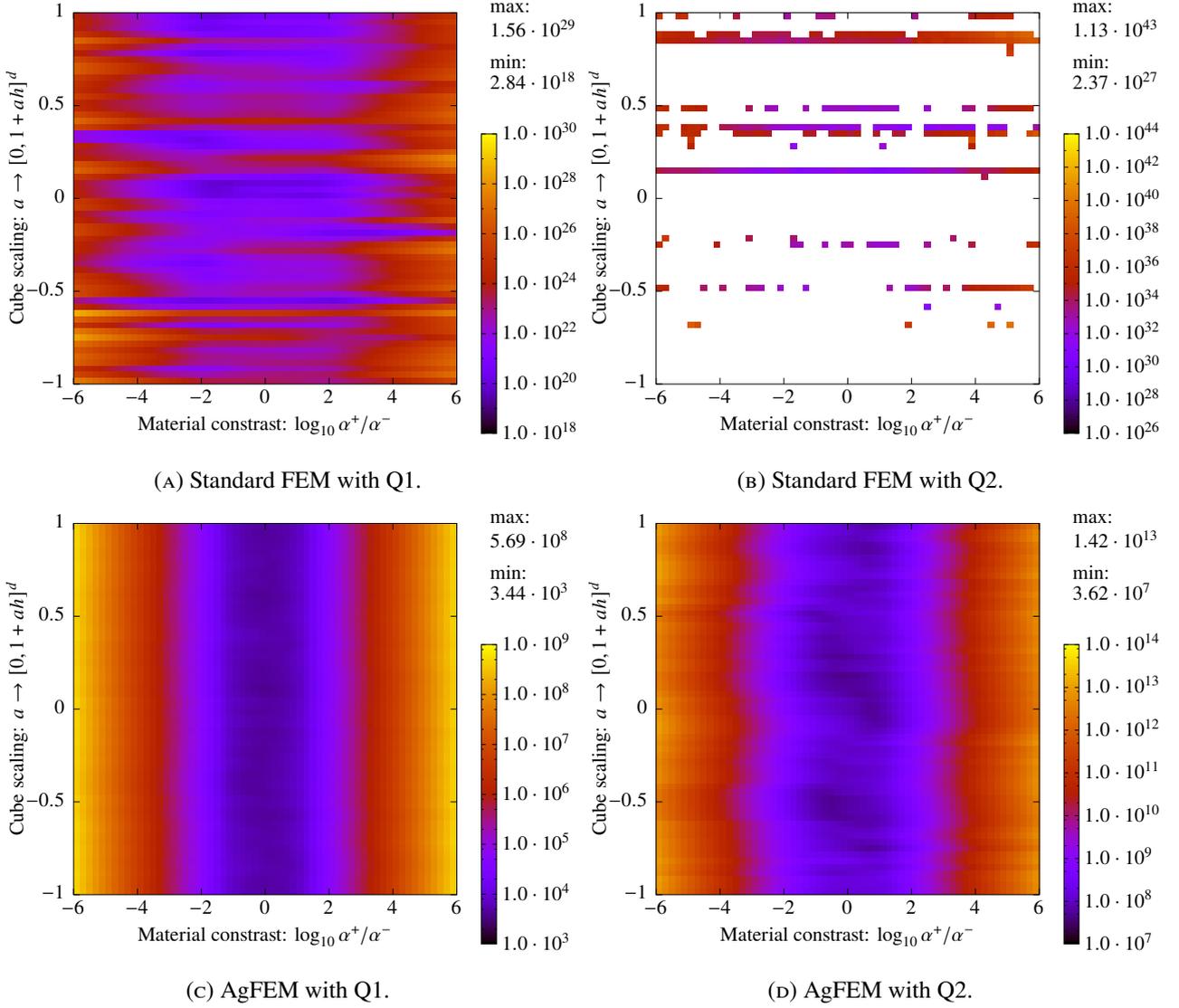

	\centering
	\begin{subfigure}[t]{0.47\textwidth}
		\scalebox{0.48}{\input{figures/run_a1b1c0d0e0f0g1_n1_condition_numbers.tex}}
		\caption{Standard \ac{fem} with Q1.}\label{fig:sensitivity_cn_a}
	\end{subfigure} \quad
	\begin{subfigure}[t]{0.47\textwidth}
		\scalebox{0.48}{\input{figures/run_a1b1c0d1e0f0g1_n1_condition_numbers.tex}}
		\caption{Standard \ac{fem} with Q2.}\label{fig:sensitivity_cn_b}
	\end{subfigure} \\ \vspace{-0.5cm}
	\begin{subfigure}[t]{0.47\textwidth}
		\scalebox{0.48}{\input{figures/run_a1b1c1d0e0f0g1_n1_condition_numbers.tex}}
		\caption{Ag\ac{fem} with Q1.}\label{fig:sensitivity_cn_c}
	\end{subfigure} \quad
	\begin{subfigure}[t]{0.47\textwidth}
		\scalebox{0.48}{\input{figures/run_a1b1c1d1e0f0g1_n1_condition_numbers.tex}}
		\caption{Ag\ac{fem} with Q2.}\label{fig:sensitivity_cn_d}
	\end{subfigure}
	\caption[Sensitivity test w.r.t. material contrast and cut location for
	popcorn example. Condition number of system matrix]{Sensitivity test w.r.t.
	material contrast and cut location for popcorn example. Examination of
	$\mathtt{condest}(A)$ exposes how lack of robustness and dependency on cut
	location in standard \ac{fem} is not present in Ag\ac{fem}.}
	\label{fig:sensitivity_cn}
\end{figure}

\begin{figure}[!h]
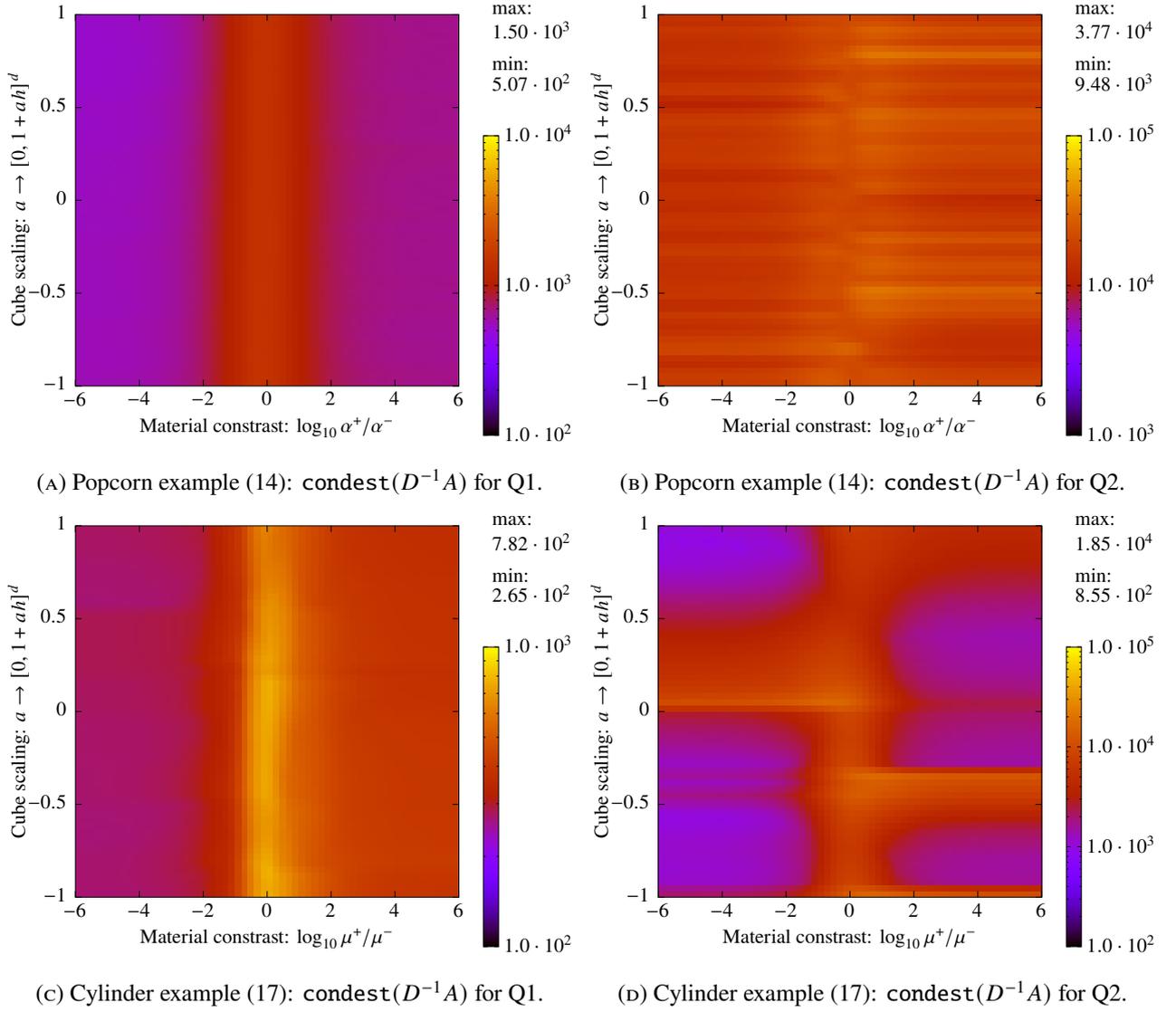

	\centering
	\begin{subfigure}[t]{0.47\textwidth}
		\scalebox{0.48}{\input{figures/run_a1b1c1d0e1f0g1_n1_condition_numbers.tex}}
		\caption{Popcorn example \eqref{eq:out-fe-space}: 
		$\mathtt{condest}(D^{-1}A)$ for Q1.}\label{fig:sensitivity_le_a}
	\end{subfigure} \quad
	\begin{subfigure}[t]{0.47\textwidth}
		\scalebox{0.48}{\input{figures/run_a1b1c1d1e1f0g1_n1_condition_numbers.tex}}
		\caption{Popcorn example \eqref{eq:out-fe-space}: 
		$\mathtt{condest}(D^{-1}A)$ for Q2.}\label{fig:sensitivity_le_b}
	\end{subfigure} \\ \vspace{-0.5cm}
	\begin{subfigure}[t]{0.47\textwidth}
		\scalebox{0.48}{\input{figures/run_a0b0c0d0e0f0g0_n1_condition_numbers.tex}}
		\caption{Cylinder example \eqref{eq:sukumar}: 
		$\mathtt{condest}(D^{-1}A)$ for Q1.}\label{fig:sensitivity_le_c}
	\end{subfigure} \quad
	\begin{subfigure}[t]{0.47\textwidth}
		\scalebox{0.48}{\input{figures/run_a0b0c0d1e0f0g0_n1_condition_numbers.tex}}
		\caption{Cylinder example \eqref{eq:sukumar}: 
		$\mathtt{condest}(D^{-1}A)$ for Q2.}\label{fig:sensitivity_le_d}
	\end{subfigure}
	\caption[Sensitivity test w.r.t. material contrast and cut location for
	popcorn example. Condition number of diagonally-scaled system matrix]{In
	Ag\ac{fem}, condition number of the diagonally-scaled system matrix,
	i.e.~$\mathtt{condest}(D^{-1}A)$, does not depend on cut location or
	material contrast and is effectively controlled; all condition numbers are
	down to $\mathcal{O}(10^4)$, in the worst case.}
	\label{fig:sensitivity_le}
\end{figure}

\subsection{Weak-scaling analysis}
\label{sec:weak-scaling}

We carry out weak-scaling tests for three $h$-adaptive cases studied in the
convergence tests: (1) the Pacman-Fichera 3D with quadratic \acp{fe} for
$k^+/k^- = 1$ (Figure~\ref{fig:verif_amr_b}) and the gyroid-shock with
(2) linear (Figure~\ref{fig:verif_amr_c}) and (3) quadratic
(Figure~\ref{fig:verif_amr_d}) \acp{fe} for $k^+/k^- = 10^3$. In the
analysis, we aim (a) to deploy a testing methodology that accounts for the fact
that cells (and \acp{dof}) that cut the interface are replicated and (b) to
demonstrate that both the cell aggregation scheme and the set up of the
Ag\ac{fe} space $\V_h$ are computationally (weakly) scalable. In the sequel we
use $N_\square$ and $n_\square$ to denote global (i.e.~referring to the whole
mesh/domain) and local (i.e.~referring to the processor-owned submesh/subdomain)
sizes/cardinalities of a quantity $\square$.

Our strategy is analogous to the one detailed in~\cite{inpreparation2020}; it
consists in repeating the convergence test, adjusting the number of processors
to compute each point in the error plot. The goal is to impose that a suitable
quantity remains (approximately) invariant across the whole convergence test. In
addition, given a point, it is desirable that the invariant \emph{also} holds
across processors, in order to reduce noise in the results due to
interprocessor imbalance. In \ac{fe} simulations, the typical invariant is the
(local) number of (free) \acp{dof} each processor owns, since complexity of
major phases (e.g.~solving the linear system) depends on the number of
\acp{dof}. However, it is difficult to balance \acp{dof} across processors in
our meshes, which have both free and (hanging and ill-posed) constrained
\acp{dof} that overlap at the interface. For this reason, we choose as invariant
the local number of active cells $n_{\act,\mathrm{cells}}$, where the global
counterpart is $N_{\act,\mathrm{cells}} = N_{\T_{h,\act}^+} +
N_{\T_{h,\act}^-}$, i.e.~the number of cells in $\T_h$, but counting cells at
the interface twice.

According to this, we consider the sequence of optimal \ac{amr} meshes, obtained
in the convergence test, and compute the number of processors for the
weak-scaling analysis as\[
	P^i = P^1 \left\lfloor
	\frac{N_{\act,\mathrm{cells}}^i}{N_{\act,\mathrm{cells}}^1}
	\right\rfloor, \ i > 1,
\]where superscript $i > 1$ refers to each element in the sequence of optimal
meshes (points in the error curve), $P^1$ is a fixed initial number of
processors and $\lfloor \cdot \rfloor$ is the \textit{floor} function; given a
real number $x$, $\lfloor x \rfloor$ is the greatest integer less than or equal
to $x$. Table~\ref{tab:partitions} gathers the sequences $\left\{P^i\right\}_{i
> 1}$ obtained following this procedure, for the three cases that are studied in
this section. We observe that (1) it is clearly more straightforward to equally
distribute active cells among processors than \acp{dof} and (2) the (average)
local number of free \acp{dof} grows mildly with $i > 1$. Hence, this approach
allows us to (conservatively) examine how the problem scales with \acp{dof},
avoiding cumbersome strategies to balance \acp{dof}.

\begin{table}[!h]
	\centering
	\begin{small}
		\begin{tabular}{lrrrrrr}
		\toprule
		\multicolumn{7}{c}{Pacman-Fichera 3D AMR-Q2 and
		$n_{\act,\mathrm{cells}} \approx 1.2k$} \\
		\midrule
		$P$ & 1 & 5 & 14 & 36 & 58 & 457 \\
		$N_{\act,\mathrm{cells}}$ & 1.2k & 5.8k & 16k & 43k & 67k & 533k \\
		$N_{\mathrm{dofs}}$ & 8.0k & 42k & 118k & 315k & 510k & 4,031k \\
		$n_{\mathrm{dofs}}$ & 8.0k & 8.3k & 8.3k & 8.6k & 8.6k & 8.8k \\
		\midrule
		\multicolumn{7}{c}{Gyroid-shock AMR-Q1 and
		$n_{\act,\mathrm{cells}} \approx 46k$} \\
		\midrule
		$P$ & 2 & 9 & 57 & 556 & 2,150 & \\
		$N_{\act,\mathrm{cells}}$ & 92k & 440k & 2,637k & 19,471k & 98,516k & \\
		$N_{\mathrm{dofs}}$ & 66k & 348k & 2,288k & 18,056k & 89,822k & \\
		$n_{\mathrm{dofs}}$ & 33k & 36k & 40k & 42k & 42k & \\
		\midrule
		\multicolumn{7}{c}{Gyroid-shock AMR-Q2 and
		$n_{\act,\mathrm{cells}} \approx 4.7k$} \\
		\midrule
		$P$ & 1 & 4 & 13 & 33 & 99 & 556 \\
		$N_{\act,\mathrm{cells}}$ & 4.7k & 19k & 62k & 157k & 474k & 2,641k \\
		$N_{\mathrm{dofs}}$ & 27k & 118k & 409k & 1,065k & 3,306k & 19,430k \\
		$n_{\mathrm{dofs}}$ & 27k & 30k & 32k & 32k & 33k & 35k \\
		\bottomrule
		\end{tabular}
	\end{small}\vspace{0.2cm}
	\caption[Number of subdomains $P$, global active cells
	$N_{\act,\mathrm{cells}}$, global \acp{dof} $N_{\mathrm{dofs}}$ and local
	\acp{dof} $n_{\mathrm{dofs}}$ for the cases considered in the weak scaling
	tests]{Number of subdomains $P$, global active cells
	$N_{\act,\mathrm{cells}}$, global \acp{dof} $N_{\mathrm{dofs}}$ and local
	\acp{dof} $n_{\mathrm{dofs}}$ for the cases considered in the weak scaling
	tests of Figure~\ref{fig:weak_scaling}. For each case, local active cells $n_{\act,\mathrm{cells}}$, remains
	quasi-constant with $P$. Besides, $n_{\mathrm{dofs}}$ (slowly) increases 
	monotonically.}
	\label{tab:partitions}
\end{table}

Once established the weak-scaling methodology, our purpose is to show that
remarkable scalability of ($h$-adaptive) Ag\ac{fem}, reported in previous works
for problems with unfitted boundary~\cite{Verdugo2019,inpreparation2020}, is
preserved for interface problems. As those works have already addressed weak
scalability of the whole \ac{fe} simulation pipeline, we focus on reporting wall
clock times spent in the two main Ag\ac{fem}-specific phases, i.e.~those phases
particular of our approach, not present in other unfitted techniques. The two
phases are (1) cell aggregation, see Section~\ref{sec:interface-ag}, and (2)
setup of the Ag\ac{fe} space, see Section~\ref{sec:interface-agfem}. As finding
the optimal mesh for each $i > 1$ is an iterative \ac{amr} process, we only
monitor these quantities for the optimal mesh (last iteration). We note that,
even though (1) and (2) are critical phases of the simulation, from the
computational viewpoint, they are not the most prominent ones. Thus, Ag\ac{fem}
does not affect much overall run time with respect to a standard (ill-posed)
Galerkin method.

To allocate the MPI tasks in the MN-IV supercomputer, we resort to the default
task placement policy of Intel MPI (v2018.4.057) with partially filled nodes.
For each point of the test, the number of nodes $N^i$ is selected as $N^i =
\left\lceil P^i/48 \right\rceil$, where $\lceil \cdot \rceil$ is the
\textit{ceiling} function; given a real number $x$, $\lceil x \rceil$ is the
smallest integer more than or equal to $x$. If $P^i$ is not multiple of 48, the
placement policy fully populates the first $N-1$ nodes with 48 MPI tasks per
node; the remaining $P^i - 48(N-1)$ MPI tasks are mapped to the last node.

Figure~\ref{fig:weak_scaling} gathers all the quantities surveyed in weak
scaling tests. The main phases of $h$-adaptive Ag\ac{fem} exhibit remarkable
scalability for the three cases considered. We observe that the number of local
active cells $n_{\act,\mathrm{cells}}$ and \acp{dof} $n_{\mathrm{dofs}}^i$, $i >
1$ for the gyroid-shock AMR-Q1 case are significantly larger than for the other
two cases. That is why this case yields the largest computational times.

\begin{figure}[!h]
	\centering
	\scalebox{0.8}{\input{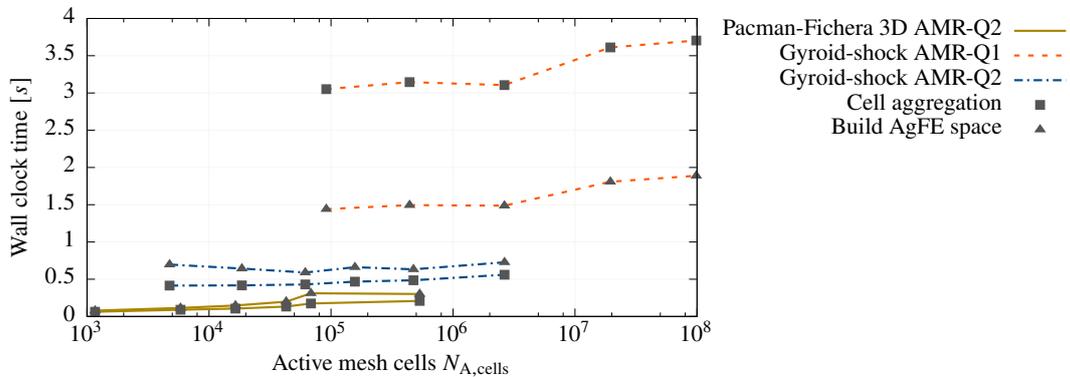}}
	\vspace{-0.15cm}
	\caption{Weak scaling tests on selected interface problems
	from convergence tests in Section~\ref{sec:conv-tests} up to 2,150 MPI
	tasks, as reported in Table~\ref{tab:partitions}.}
	\label{fig:weak_scaling}
\end{figure}

\section{Conclusions}
\label{sec:conclusions}

This work addressed a novel $h$-adaptive aggregated \ac{fe} method for
large-scale (unfitted) interface elliptic boundary value problems. Our
methodology is grounded on the well-established approach of weakly coupling
interface-overlapping discretisations~\cite{hansbo2002unfitted} and the recently
developed $h$-adaptive Ag\ac{fe} method~\cite{inpreparation2020} for unfitted
boundary elliptic problems. The study of the new method is accompanied with
complete theoretical characterisation and thorough numerical experimentation on
a suite of Poisson and linear elasticity ($hp$-\ac{fem}) benchmarks with complex
interface shapes.

As main contributions of the paper, we have introduced a (a) natural extension of
the (distributed-memory) cell aggregation algorithm in~\cite{Verdugo2019} for
$n$-interface problems. We have shown that (b) Ag\ac{fe} spaces easily blend to
the typical Cartesian-product approximation structures of interface-overlapping
meshes. We have proven (c) well-posedness and optimal approximation properties of
a \ac{sipm}-Ag\ac{fem} discrete formulation for the irreducible linear elasticity
problem. Robustness w.r.t.~cut location is ensured, by inheriting cut-independent
estimates from Ag\ac{fem} in unfitted boundaries, while robustness w.r.t.~material
contrast is achieved, by using the same weighted average of body-fitted \ac{dg}
methods. Besides, the resulting method admits (d) straightforward implementation
on top of an existing large-scale implementation of Ag\ac{fem} for unfitted
boundary problems. To conclude, exhaustive numerical tests have exposed (e)
optimal ($h$-adaptive) approximation capability, robustness \myadded{with respect
to} cut location and material contrast and remarkable scalability on parallel
adaptive Cartesian tree-based meshes.

Our study offers compelling insight and evidence of the potential of Ag\ac{fem}
as an effective large-scale \ac{fe} solver for complex multiphase and
multiphysics problems modelled by \acp{pde}. Extension to any of those problems
is object of future work. Additionally, the paper provides useful guidance in
applying other unfitted \ac{cg} methods to interface problems, especially those
relying on cell aggregation.

\section*{Acknowledgements}

\thethanks

\begin{small}

\setlength{\bibsep}{0.0ex plus 0.00ex}
\bibliographystyle{myabbrvnat}
\bibliography{art043}

\end{small}

\end{document}